\documentclass[12pt]{article}
\usepackage{graphicx}
\usepackage{subfigure}
\usepackage{booktabs}   % For \toprule and \bottomrule
\usepackage{hhline}     % For \hhline
\usepackage{epstopdf} 

\usepackage{booktabs}
\usepackage{color}
\usepackage{latexsym}
\usepackage{setspace}
\usepackage{amsmath}
\usepackage{amssymb}
\usepackage{latexsym}
\usepackage{amsmath}
\usepackage{amssymb}
\usepackage[numbers]{natbib}
\usepackage{hyperref}
\hypersetup{
	colorlinks=true,
	linkcolor=blue,
	filecolor=blue,
	urlcolor=blue,
}
\usepackage{subfigure}
\usepackage{graphics}
\usepackage{graphicx}
\usepackage{multirow}
\makeatletter \oddsidemargin  -.5cm \evensidemargin -.5cm \textwidth
17.5cm \topmargin -.65in \textheight 24cm
\makeatother   % CHANGED: this \makeatletter was never closed

\newtheorem{defn}{Definition}[section]
\newcommand{\bd}{\begin{definition}}
	\newcommand{\ed}{\end{definition}}
\newcommand{\br}{\begin{remark}}
	\newcommand{\er}{\end{remark}}
\newcommand{\bea}{\begin{eqnarray}}
	\newcommand{\eea}{\end{eqnarray}}
\newcommand{\beann}{\begin{eqnarray*}}
	\newcommand{\eeann}{\end{eqnarray*}}
\newtheorem{theorem}{Theorem}[section]

\newtheorem{lemma}[theorem]{Lemma}

\newtheorem{corollary}[theorem]{Corollary}
\newtheorem{remark}{Remark}[section]
\newtheorem{example}{Example}[section]

\numberwithin{equation}{section} \numberwithin{equation}{section}
\begin{document}
	\title{Estimation of the entropy of normal populations under prior information}
\author{%
	Somnath Mondal$^{*1}$ and Lakshmi Kanta Patra$^{1}$\\[8pt]
	{\small $^{1}$\textit{Department of Mathematics, Indian Institute of Technology Bhilai, India}}}%
	\date{}
	\maketitle
	\renewcommand{\thefootnote}{\fnsymbol{footnote}}
	\footnotetext[1]{\texttt{somnathmondal@iitbhilai.ac.in}}
	
	\begin{abstract}
		In many practical applications, prior knowledge about unknown parameters is available in the form of ordering constraints, and incorporating such information into the inferential procedure can improve the efficiency of the resulting estimators. In this paper, we consider the problem of estimating Shannon's entropy of two normal populations with ordered location parameters under a general location-invariant loss function. We derive the maximum likelihood estimator, the restricted maximum likelihood estimator, and the uniformly minimum variance unbiased estimator. Furthermore, we obtain an estimator that dominates the best affine equivariant estimator. In addition, we propose a smooth estimator that improves upon the best affine equivariant estimator. Moreover, we obtain improved estimators with respect to the Pitman closeness criterion. We consider interval estimation and construct several confidence intervals, such as asymptotic confidence intervals, bootstrap confidence intervals, highest posterior density credible intervals, and the intervals based on generalized pivotal quantities. A Monte Carlo simulation study is conducted to assess the performance of the proposed point and interval estimators. Finally, the proposed methods are illustrated by using two real data sets.
	\end{abstract}

	
%	\noindent\textbf{MSC (2020):} 62F10; 62C20
	
	\section{Introduction} 
The entropy is one of the most significant concepts in information theory. This concept was introduced by \cite{shannon1948mathematical}, who described entropy in a mathematical way to measure the uncertainty in a probabilistic system. 
Let $X$ be a random variable whose probability density function is $p(x;\nu)$ where $\nu$ is an unknown parameter. The Shannon entropy of a continuous random variable $X$ is defined as
\begin{equation}
	H(\lambda)=-E[\ln p(X;\nu)] = - \int_{-\infty}^{\infty}p(x;\nu) \ln p(x;\nu) dx.
\end{equation}
It measures the average amount of uncertainty or information associated with the outcome of the random variable $X$.

In communication theory, the main problem is how to efficiently encode, store, and transmit information over a noisy communication channel. Shannon solved this problem in his theory. \cite{shannon1948mathematical} gave two very important results based on entropy. First, the lowest possible average rate at which the data is transmitted from a source by compressing without losing information is equal to its entropy (the source coding theorem). Second, the reliable transmission of the information through a channel is possible only when the entropy of the source is less than or equal to the capacity of the channel. These results determine the maximum rate at which information can be reliably transmitted through a channel. So the concept of entropy not only has a theoretical importance but also a practical tool, which is used to design communication and data compression systems. Later researchers realized that entropy is useful in many other fields such as	% biology (to measure genetic diversity and analyze DNA sequences), economics (to study the income inequality and market uncertainty), cryptography (to design secure encryption algorithm and assess randomness of keys), probability theory (to study stochastic process and measure uncertainty of random variables) and statistics (model selection criteria and measure information loss). Estimation of Shannon entropy for different statistical distribution plays an important role in the field of
economics, information theory, reliability and life-testing studies and molecular science. In the fields of molecular sciences, the estimation of molecular entropy is very important to study of different chemical and biological phenomena because it measures the level of disorder and energy distribution in the complex systems (see \cite{nalewajski2002applications}). In economics, entropy estimation provides a systematic approach to exploit information in data to narrow the assumptions about the parameters of an econometric model (see \cite{golan1996maximum}). \cite{kamavaram2002entropy} came up with an entropy framework to analyze the uncertainty in architectural based software reliability models. It demonstrated that systems with uniform operational profiles and components with moderate reliability have greater uncertainty and require more testing effort. It provided measures to aid in identifying critical components and in the effective allocation of testing resources. So the estimation of entropy is a important problem in the field of parametric inference. However, there are a few works available for the estimation of entropy from a decision-theoretic point of view. \cite{misra2005estimation} considered the problem of estimating the entropy of a multivariate normal distribution under quadratic loss function. They proved that under the the BAEE is admissible in the class of estimators that depends only on the determinant of the sample covariance matrix. Further, they showed that this BAEE is inadmissible and obtained  \cite{stein1964}-type and \cite{brewster1974improving}-type improved estimators. In addition, the Brewster-Zidek type estimator is shown to be generalized Bayes estimator. \cite{kayal2013estimation} studied the problem of estimating entropy of several exponential distributions under a squared error loss function. They established a general inadmissibility result for scale equivariant estimators and extended these results to the case of unequal sample sizes. \cite{patra2020measuring} consider the problem of estimating entropy for  two exponential distributions with a common scale parameter and ordered location parameters. They showed that the unrestricted BAEE is inadmissible. The authors  obtained various improved estimators under location invariant loss functions. The performance of these estimators was studied under various sampling schemes, including i.i.d., record values, and censored samples, using squared error and linex loss functions. Some recent works where the aforementioned techniques have been used to the estimation of unknown parameters for various distributions using a decision-theoretic approach can be found in \cite{petropoulos2020improved}, \cite{patra2021minimax}, \cite{patra2021improved}, \cite{patra2021componentwise}, \cite{mondal2025estimating}, \cite{patra2025inadmissibility}, \cite{mandal2026estimation}.

%   Further work on entropy estimation for exponential, logistic, half-logistic, and generalized half-logistic distributions can be found in \cite{kang2012estimation}, \cite{seo2014entropy}, \cite{seo2015bayesian} and \cite{seo2017objective}. 
%In the area of interval estimation, various techniques has been proposed. Some works in this directions are \cite{jordan1996exact}, \cite{philip1999exact}, \cite{khatun2022hypothesis}, \cite{jena2023point}.% where confidence intervals are constructed for the power of scale parameters of two normal populations.

In this work we consider the estimation of entropy of two normal population. Consider that $(X_{i1},\dots,X_{in})$ denote a random sample drawn from the $i$-th normal population having mean $\mu_{i}$ and common variance $\sigma^{2},$ for $i=1,2$ subject to the restriction $\mu_1\le \mu_2$. Assume that two normal populations	$N(\mu_{1},\sigma^{2})$ and $N(\mu_{2},\sigma^{2})$ are independent. Define
\begin{eqnarray*}
	X_i=\sum_{j=1}^{n}X_{ij}/n ~~~\mbox{and} ~~~
	S^2=\sum_{i=1}^{2}\sum_{j=1}^{n}\left(X_{ij}-X_i\right)^2.
\end{eqnarray*}
Then, $X_i \sim N(\mu_i,\sigma^2/n)$ and $V=S^2/\sigma^{2}\sim \chi^2_{2(n-1)}$. Observe that $(\underline{X},S^2)$ forms a complete and sufficient statistic, with $\underline{X}=(X_{1},X_{2})$. For this system the Shannon entropy is obtained as $H(\sigma)=1+\ln(2\pi) + 2\ln\sigma$. Now the problem of estimating $H(\sigma)$ can be reduced to the problem of estimating $\tau = \ln \sigma$. In this study, we consider the problem of estimating $\tau$ through both point and interval estimation of approaches. The point estimation of $\tau$ is studied under a class of location invariant loss function $L(\delta-\tau)$, subject to the condition $\mu_1\le \mu_2$. Let the loss function $L(t)$ be assumed to satisfy the following conditions
\begin{enumerate}
	\item [(C1)] $L(t)$ is a strictly convex with $L(0)=0$.
	\item [(C2)] The integral associated with $L(t)$ are finite and allow differentiable under the integral sign.
\end{enumerate}

In addition to the point estimation we also consider the confidence interval estimation of $\tau$. For this purpose we have used several techniques namely asymptotic confidence interval, bootstrap confidence intervals, generalized confidence intervals and HPD credible interval using Markov Chain Monte Carlo (MCMC) method. Several authors and researchers have studied confidence interval estimation for a variety of statistical models and distributions. We review some recent works where the aforementioned techniques have been used. \cite{ahmed2020inference} studied the estimation of unknown parameters of the Chen distribution under progressive Type-II censoring with competing risks. They obtain maximum likelihood estimates using the EM algorithm and construct asymptotic as well as bootstrap confidence intervals. From a Bayesian perspective, they compute Bayes estimates using MCMC under different loss functions. \cite{ghanbari2022estimation} considered the problem of estimating the stress–strength parameter $\mathcal{R}$ under the Marshall–Olkin model with progressively Type-II censored data. Both maximum likelihood and Bayesian methods (using Lindley’s approximation and Gibbs sampling) are applied, and various confidence intervals for $\mathcal{R}$ are derived. In this direction, one may refer to several related works can be found in \cite{khatun2022hypothesis}, \cite{jena2023point}, \cite{khatun2024point}.

%An usual method to estimate $\tau$ is in terms of the complete sufficient statistics $S$ (such as MLE, UMVUE). 
%It has been noticed that in the statistical literature (see \cite{brewster1974improving}, \cite{stein1964}) for estimating scale parameter, the estimators depends only on the sufficient statistic $S$ are not best. If we use all the information from the both statistic $(\underline{X}, S)$ we will get the better estimator than our usual estimators  of $\tau$. 
%This is one of the main reason for us top find a better estimator of $\tau$.
To begin with, we apply the principle of invariance. Let us consider an affine group of transformations
$\mathcal{G}=\{g_{a,\underline{b}}:g_{a,\underline{b}}(\underline{l})
=a\underline{l}+\underline{b}, a>0,\underline{l}=(l_{1},l_{2}),\underline{b}=(b_{1},b_{2}),
\underline{l},\underline{b}\in \mathbb{R}\times \mathbb{R}\}$ and  an affine equivariant estimator for $\tau$ under $\mathcal{G}$ is obtained as $\delta_d(\underline{X},S)=\ln S + d.$  
%When there is no constraints are imposed on the location parameters 
The following lemma gives the the BAEE for $\tau$ under a general location invariant loss function $L(t)$.
\begin{lemma}\label{BAEE}
	Under a location-invariant loss function $L(t)$, BAEE of $\tau$ is given by
	\begin{equation*}
		\delta_{0}(\underline{X},S)=\ln \sqrt{V} + d_0,
	\end{equation*}
	where $d_0$ is the unique solution of
	\begin{equation}\label{d0}
		E(L'(\ln \sqrt{V} + d_0)) = 0.
	\end{equation} 
\end{lemma}
Here, $L'(\cdot)$ denotes the derivative of the loss function $L(\cdot)$. % and expectation is taken with respect to the distribution where the location parameter is $\underline{0}=(0,0)$ and the scale parameter is $1$.
\begin{remark}\label{p}
	The risk function of the BAEE $\delta_{0}$ can be express as $R(\delta_0;\underline{\mu},\tau)=EL(\ln S + d_0)$, which we will denote as $p$. This risk value remains constant since it does not depend on the parameters. As stated in \cite{kiefer1957invariance}, when there is no order constraint on $\mu_{1}$, $\mu_{2}$; the estimator $\delta_{0}=\ln S + d_0$ is minimax in respect to a general loss function $L(t)$.
\end{remark}	
%	It is important to noting \cite{stein1964} and \cite{brewster1974improving} pioneered two significant methods that we have used to derived further improved estimators over the BAEE. There are several authors and researchers who have applied this technique to derived improved estimators in various estimation problems. Some recent uses of these techniques an be seen in \cite{kayal2013estimation}, \cite{misra2005estimation}, \cite{mondal2024improved}, \cite{petropoulos2017estimation}.
We next consider two location invariant loss function, namely the squared error loss function $L_1(t)$ and linex loss function $L_2(t)$. Under the squared error loss, overestimation and underestimation are penalized equally because of its symmetric nature. In contrast, the linex loss function assigns greater weight to overestimation compared to underestimation. %However for the linex loss function the penalize more in overestimation than underestimation. 
\begin{example}\label{ex1}
	Let $L_1(t)=t^2$, the BAEE of $\tau$ is obtained as $\delta_{01}=\ln\sqrt{V}-\frac{1}{2}\left[\ln2+\psi(n-1) \right]$, where $\psi(x)=\frac{d}{dx}\left(\ln \Gamma(x)\right)$, $x>0$ denote the digamma function.
\end{example}
\begin{example}\label{ex2}
	Let $L_2(t)=e^{a_1t}- a_1t-1,\ a_1 \neq0$, the the BAEE of  $\tau$ is obtained as
	$\delta_{02}=\ln \sqrt{V} - \frac{1}{a_1}\ln\left[2^{\frac{a_1}{2}}\frac{\Gamma\left(n-1+\frac{a_1}{2}\right)}{\Gamma\left(n-1\right)}\right]$.
\end{example}
\begin{remark}
	The uniformly minimum variance unbiased estimator (UMVUE) of $\tau$ is $\delta_{U}=\ln S - \frac{1}{2}\left[\ln2 + \psi(n-1)\right]$, and it is a member of the class of affine equivariant estimators. Note that $\delta_{U}$ is the BAEE of $\tau$ under the squared error loss.
\end{remark}
We now incorporate the prior information $\mu_1 \leq \mu_2$ and establish a general minimaxity result for the problem of estimating $\tau$.
\begin{theorem}\label{minma}
	Let $\Omega_0 \subset \mathbb{R}^3$ be such that there exists a sequence $\left\{\underline{a}_n=(a_{n_1},a_{n_2}):\ n\geq 1\right\}$ satisfying
	\begin{equation}\label{1.3}
		\lim\limits_{n}\inf\left\{(\underline{\mu},\tau):\ (\mu_{1}+a_{n_1}, \mu_{2}+a_{n_2}, \tau)\in \Omega_0\right\}=\mathbb{R}^3.
	\end{equation}
	Assume that $\delta$ is an estimator of $\tau$ for which the risk function satisfies
	\begin{align*}
		R(\delta;\underline{\mu},\tau)\leq R(\delta_0;\underline{\mu},\tau)=p,~~~~ \forall~ (\underline{\mu},\tau)\in \Omega_0
	\end{align*}
	where p is defined in the Remark \ref{p}. Then, $\delta$ is minimax estimator of $\tau$ over $\Omega_0$ under any location invariant loss function $L(t)$ .
\end{theorem}
\noindent \textit{Proof:} The proof is a special case of Theorem 1 of \cite{marchand2012unified}.\\

We have the following corollaries from the Theorem \ref{minma}.
\begin{remark}\label{re1}
	Suppose $\delta$ is an estimator of $\tau$ whose risk satisfies $R(\delta;\underline{\mu},\tau)\le p$ for all $\mu_1\leq\mu_2$, where $p$ denotes the constant risk  associated with the BAEE $\delta_0$. Then $\delta$ is a minimax for estimating $\tau$ over the parameter space defined by $\mu_1\leq \mu_2$ under any location invariant loss function $L(t)$.
\end{remark}
\begin{remark}
	By choosing $\delta=\delta_0$ in Remark \ref{re1}, it follows that, under the loss function $L(t)$, the BAEE $\delta_0$ is minimax for estimator of $\tau$ under the prior constraint $\mu_1\leq \mu_2$.
\end{remark}
\begin{remark}
	It can be observed that the estimators derived in Example \ref{ex1} and \ref{ex2} are minimax under the loss functions $L_1(t)$ and $L_2(t)$, respectively, when $\mu_1 \leq \mu_2$.
\end{remark}
Now we derive the expressions for the risk and bias of the BAEE. Observe that, for any affine equivariant estimator, the risk as well as the bias do not depend on the parameter $(\underline{\mu},\sigma)$. Hence, from now on, for any real number $d$, we denote the risk $R(\delta_d;\underline{\mu},\sigma)$ and for the bias $B(\delta_d;\underline{\mu},\sigma)$.
\begin{theorem}
	Let $d_{01}=-\frac{1}{2}\left[\ln2+\psi(n-1)\right]$, $d_{02}= -  \frac{1}{a_1}\ln\left[2^{\frac{a_1}{2}}\frac{\Gamma\left(n-1+\frac{a_1}{2}\right)}{\Gamma\left(n-1\right)}\right]$, and for $y>0$ let $\psi_1(y)=\frac{d}{dy}\psi(y)$ denotes the trigamma function. Then we have 
	\begin{enumerate}
		\item [(i)] $B(\delta_{01};\underline{\mu},\sigma)= 0.$
		
		\item [(ii)] $B(\delta_{02};\underline{\mu},\sigma)= \frac{1}{2} \psi(n-1) - \frac{1}{a_1}\ln \left(\frac{\Gamma\left(n-1+\frac{a_1}{2}\right)}{\Gamma(n-1)}\right).$
		
		\item [(iii)] $R(\delta_{01};\underline{\mu},\sigma)=\frac{1}{4}\psi_1(n-1).$
		
		\item [(iv)] $R(\delta_{02};\underline{\mu},\sigma)= 2^{\frac{a_1-1}{2}} \left(\frac{\Gamma(\frac{n+a_1-2}{2})}{\Gamma(n-1)}\right)^{1-\frac{1}{a_1}} -\frac{a_1}{2}\psi(n-1) +\ln\left(\frac{\Gamma(\frac{2n+a_1-2}{2})}{\Gamma(n-1)}\right)-1.$
		%		
		%		\item[(v)] The risk difference $\nabla $ is $R(\delta_{d_{02}})-R(\delta_{d_{01}})$. 
	\end{enumerate}
\end{theorem}
\noindent\textit{Proof:} The proofs of $(i)$ and $(ii)$ are straightforward. Statement $(iii)$ can be established using the fact that, for a $\chi_p^2$ random variable, $\mbox{Var}(\ln \chi_p^2)=\psi_1(\frac{p}{2})$. The proof of $(iv)$ is also straightforward.\\

The remaining sections of this paper are arranged as follows. In Section \ref{sec2}, we derive an estimator which dominates the BAEE under a general location invariant loss function $L(t)$. It is worth noting that this estimator is not smooth. In the next Section \ref{sec3} we find a smooth estimator which dominates the BAEE under $L(t)$. In Section \ref{sec4} we consider estimation of $\tau$ under the generalized Pitman closeness criterion with a general bowl shaped loss function. In Section \ref{sec5} we compare numerically the risk performance of the improved estimator  with the respect to the BAEE using Monte Carlo simulations. Further in Section \ref{sec6} we consider interval estimation of the parameter $\tau$ under the same model assumption. We have discussed asymptotic confidence interval using delta method, bootstrap confidence intervals, generalized confidence intervals and HPD credible interval using Markov Chain Monte Carlo method. 
% In Subsection \ref{sec61} we obtain the asymptotic confidence interval using delta method. In Subsection \ref{sec62} we obtain bootstrap confidence intervals. In Subsection \ref{sec63} consider the generalized confidence intervals. In Subsection \ref{sec64} we obtain highest posterior density (HPD) credible interval using Markov Chain Monte Carlo (MCMC) method. 
In Section \ref{sec7} we compare the performance of all interval estimators with respect to coverage probability (CP), average length (AL) and probability coverage density (PCD). For implementation purpose in Section \ref{sec8} we present a real life data analysis.

\section{Improved estimation of $\tau$}\label{sec2}
In this section, we derive improved estimator of $\tau$. The improvement is based on the information contained in $(\underline{X},S)$ under a general location invariant loss function $L(t)$ which satisfies the condition (C1) and (C2). For this purpose we consider a class of estimator as
\begin{equation}\label{d1}
	\mathcal{D}_1=\left\{\delta_{\phi}: \delta_{\phi}(\underline{X},S)=\ln \sqrt{V} + \phi(W) \right\}
\end{equation}
where $\phi(\cdot)$ be a measurable function and $W=\frac{Z_2-Z_1}{\sqrt{nV}}$, $Z_1=\frac{\sqrt{n}(X_1-\mu_{1})}{\sigma} \sim N(0,1)$, $Z_2=\frac{\sqrt{n}(X_2-\mu_{1})}{\sigma}\sim N(\eta,1)$ with $\eta=\frac{\mu}{\sigma}$ and $\mu=\sqrt{n}(\mu_2-\mu_1)$. Under the restriction $\mu_{1}\leq \mu_{2}$, we have $\mu\geq 0$. One can see that the BAEE belongs to the class $\mathcal{D}_1$ corresponding to the choice $\phi(W)=d_0$. In order to derive an improved estimator within the class $\mathcal{D}_1$, we need to analyze the risk function of the estimator $\delta_{\phi}$. Observing that the risk of the estimator $\delta_{\phi}$ is depends on $\mu/\sigma$, we can set $\sigma=1$ without affecting the generality of the analysis. Now the risk function of an estimator of the form (\ref{d1}) is
\begin{align*}
	R(\delta_{\phi},\mu)
	&= E_{\mu}\!\left[L\!\left(\ln(\sqrt{V}) + \phi(W)\right)\right] = E_{\mu}^{W}\!\left[
	E_{\mu}\!\left\{
	L\!\left(\ln(\sqrt{V}) + \phi(W)\right)
	\mid W
	\right\}
	\right].
\end{align*}
To derive an improved estimator we will study the conditional risk $$R_1(\delta_{\phi},\mu)=E_{\mu}\left[L(\ln (\sqrt{V})+\phi(W))|W=w\right].$$ The condition distribution of $V$ given $W=w$ is obtained as 
$$f_{\mu}(v|w)\propto h(w)\exp\left[-\frac{1}{2}\left(v+\frac{1}{2}(\sqrt{nv}w-\mu)^2\right)\right]v^{\frac{2n-3}{2}},\ v>0,\  w\in \mathbb{R}.$$ 
We have $R_1(\delta_{\phi},\mu)$ is strictly bowl shaped in $\phi(w)$ which is  minimized at $\phi_{\mu}(w)$ which is the unique solution of 
$E\left[L^{\prime}(\ln (\sqrt{V})+\phi(W))|W=w\right]=0$. We have $\frac{f_{\mu}(v|w)}{f_{0}(v|w)}$ is increasing in $v$ for $w>0$. This implies that $E_{\mu}\left[L^{\prime}(\ln (\sqrt{V})+d)|W=w\right] \ge E_{0}\left[L^{\prime}(\ln (\sqrt{V})+d)|W=w\right]$ for all $d \in \mathbb{R}$. So we have for $\mu>0$
\begin{align*}
	E_{\mu}\!\left[L'\!\left(\ln(\sqrt{V})+\phi_{0}(w)\right)\mid W=w\right]  &\ge E_{0}\!\left[L'\!\left(\ln(\sqrt{V})+\phi_{0}(w)\right)\mid W=w\right] \\
	&= 0\\
	&= E_{\eta}\!\left[L'\!\left(\ln(\sqrt{V})+\phi_{\mu}(w)\right)\mid W=w\right].
\end{align*}
This proves that $\phi_{\mu}(w) \le \phi_{0}(w)$, where $E_{0}\left[L^{\prime}(\ln (\sqrt{V})+\phi_{0}(w))|W=w\right]=0$, that is 
\begin{eqnarray}
	\int_{0}^{\infty}
	L'\!\left(\ln\sqrt{v}+\phi_{0}(w)\right)
	\exp\!\left[
	-\frac{1}{2}v\!\left(1+\frac{nw^2}{2}\right)
	\right] v^{\frac{2n-3}{2}} \, dv
	= 0.
\end{eqnarray}
Take the transformation $u=v\left(1+\frac{nw^2}{2}\right)$, we have  $$EL^{\prime}\left(\ln \sqrt{U}-\ln \sqrt{1+\frac{nw^2}{2}}+\phi_{0}(w)\right)=0$$ where $U\sim Gamma (\frac{2n-1}{2},2)$ random variable. Suppose $m_0$ is unique solution of  $EL^{\prime}\left(\ln \sqrt{U}+m_0\right)=0$ then we have $\phi_{0}(w)=m_0 +\ln \sqrt{1+\frac{nw^2}{2}}$. Consider a function 
\begin{equation}\label{st0}
	\phi^*_{0}(w)=\left\{
	\begin{array}{ll}
		\min\left\{d_0,m_0+\ln \sqrt{1+\frac{nw^2}{2}}\right\},~~~w>0\\
		d_0, ~~~~~~~~~~~~~~~~~~~~~~~~~~~~~~~~~~\mbox{otherwise}
	\end{array}
	\right.
\end{equation}
Now $\phi_{\mu}(w) \le \phi^*_{0}(w) \le d_0$ and the later inequality is strict on a set of positive probability provided $d_0>m_0$. Since the risk function is strictly increasing in $(\phi_{\mu}(w), \infty)$.  So $R_1(\delta_{\phi^*_{0}},\mu) \le R_1(\delta_{0},\mu)$ and strict inequality holds provided $d_0>m_0$. Again for $w<0$, we have $\frac{f_{\mu}(v|w)}{f_{0}(v|w)}$ is decreasing in $v$ this gives us 
$\phi_{\mu}(w) \ge \phi_{0}(w)$. Consider a function 
\begin{equation}\label{st1}
	\phi^{**}_{0}(w)=\left\{
	\begin{array}{ll}
		\max\left\{d_0,m_0+\ln \sqrt{1+\frac{nw^2}{2}}\right\},~w<0\\
		d_0, ~~~~~~~~~~~~~~~~~~~~~~~~~~~~~~~~\mbox{otherwise}
	\end{array}
	\right.
\end{equation}
and similarly we have $R_1(\delta_{\phi^*_{0}},\mu) < R_1(\delta_{0},\mu)$ for all $\mu>0$ and for all $w<0$. 
We state this result as follows: 
\begin{theorem}
	The estimator 
	\begin{equation}\label{st2}
		\delta_{S}=\left\{
		\begin{array}{ll}
			\ln \sqrt{V} + \max\left\{d_0,m_0+\ln \sqrt{1+\frac{nW^2}{2}}\right\}, ~W<0\\
			\ln \sqrt{V} + \min\left\{d_0,m_0+\ln \sqrt{1+\frac{nW^2}{2}}\right\}, ~W>0
		\end{array}
		\right.
	\end{equation}
	dominates $\delta_0$ with respect to a location invariant bowl shaped loss function $L(t)$ provided $d_0>m_0$, where the random variable $U\sim Gamma (\frac{2n-1}{2},2)$ and $m_0$ is unique solution of  $EL^{\prime}\left(\ln \sqrt{U}+m_0\right)=0$.
\end{theorem}
A general form of the above theorem can be proved similarly. We state the result with out proof. 
\begin{theorem}\label{stth1}
	The estimator 
	\begin{equation}\label{st}
		\delta_{S}=\left\{
		\begin{array}{ll}
			\ln \sqrt{V} + \max\left\{\phi(W), m_0+\ln \sqrt{1+\frac{nW^2}{2}}\right\}, ~W<0\\
			\ln \sqrt{V} + \min\left\{\phi(W), m_0+\ln \sqrt{1+\frac{nW^2}{2}}\right\}, ~W>0
		\end{array}
		\right.
	\end{equation}
	dominates $\delta_{\phi}$ with respect to a location invariant bowl shaped loss function $L(t)$.
\end{theorem}

We give the following examples as an application of the above result, where dominating estimators are obtained over the BAEE under the loss function $L_1(t)$ and $L_2(t)$ respectively. 
\begin{example}
	For the quadratic loss function $L_1(t)$ we have, $m_0=-\frac{1}{2}\left[\psi\left(\ln2+\frac{2n-1}{2}\right)\right]$. The improved estimator of $\tau$ over BAEE is obtained as 
	\begin{equation*}
		\delta_{S1} =
		\begin{cases}
			\ln\sqrt{V}
			+ \max\Big\{-\tfrac12[\ln2+\psi(n-1)], -\tfrac12\!\left[\ln2+\psi\!\left(\tfrac{2n-1}{2}\right)\right]
			+ \ln\sqrt{1+\tfrac{nW^2}{2}}
			\Big\}, W<0, \\
			
			\ln\sqrt{V}
			+ \min\Big\{
			-[\ln2+\psi(n-1)], -\tfrac12\!\left[\ln2+\psi\!\left(\tfrac{2n-1}{2}\right)\right]
			+ \ln\sqrt{1+\tfrac{nW^2}{2}}
			\Big\}, W>0.
		\end{cases}
	\end{equation*}
\end{example}
\begin{example}
	Under the linux loss function $L_2(t)$ we get, $m_0= -\frac{1}{a_1}\ln\left(2^{\frac{a_1}{2}}\frac{\Gamma\left(\frac{2n+a_1-1}{2}\right)}{\Gamma\left(\frac{2n-1}{2}\right)}\right)$. The improved estimator of $\tau$ over BAEE is obtained as 
	\begin{equation*}
		\delta_{S2} =
		\begin{cases}
			\ln\sqrt{V}
			+ \max\Bigg\{
			-\tfrac{1}{a_1}\ln\!\left(2^{\tfrac{a_1}{2}}
			\frac{\Gamma\!\left(\tfrac{2n+a_1-2}{2}\right)}{\Gamma(n-1)}\right), -\tfrac{1}{a_1}\ln\!\left(2^{\tfrac{a_1}{2}}
			\frac{\Gamma\!\left(\tfrac{2n+a_1-1}{2}\right)}
			{\Gamma\!\left(\tfrac{2n-1}{2}\right)}\right)
			+ \ln\sqrt{1+\tfrac{nW^2}{2}}
			\Bigg\}, ~W<0, \\
			
			\ln\sqrt{V}
			+ \min\Bigg\{
			-\tfrac{1}{a_1}\ln\!\left(
			2^{\tfrac{a_1}{2}}
			\frac{\Gamma\!\left(\tfrac{2n+a_1-2}{2}\right)}
			{\Gamma(n-1)}
			\right), -\tfrac{1}{a_1}\ln\!\left(
			2^{\tfrac{a_1}{2}}
			\frac{\Gamma\!\left(\tfrac{2n+a_1-1}{2}\right)}
			{\Gamma\!\left(\tfrac{2n-1}{2}\right)}
			\right)
			+ \ln\sqrt{1+\tfrac{nW^2}{2}}
			\Bigg\}, ~ W>0.
		\end{cases}
	\end{equation*}
\end{example}
%####################################################################################
\subsection{Restricted MLE}
The unrestricted MLE of $\sigma^2$ is $T_{M}=\frac{S^2}{2n}$ and the restricted MLE of $\sigma^2$ is given as (see \cite{gupta1992pitman})
\begin{equation}\label{equ2.3}
	T_{RM}=\left\{
	\begin{array}{ll}
		\frac{S^2}{2n},~~~~~~~~~~~~~~~~~~~~~~\mbox{ if }~~~X_1 \le X_2
		\vspace{.2 cm}\\
		\frac{S^2}{2n}+\frac{1}{4}(X_1-X_2)^2,~~\mbox{ if }~~~X_1>X_2.
	\end{array}
	\right.
\end{equation}
Now using invariance property of MLE, we have unrestricted MLE of $\tau$ is 
$\delta_{M}=\ln S -\ln \sqrt{2n}$. The restricted MLE is obtained as 
\begin{equation}\label{equ2.3}
	\delta_{RM}=\left\{
	\begin{array}{ll}
		\ln S+\ln\sqrt{\frac{1}{2n}},~~~~~~~~~~~~~~~~\mbox{ if } ~~W \ge 0
		\vspace{.2 cm}\\
		\ln \left(\frac{S}{\sqrt{2n}}\right)+\ln \sqrt{\left(1+\frac{2n}{4}W^2\right)},\mbox{ if } W<0.
	\end{array}
	\right.
\end{equation}
%\begin{remark}
%	Using the Jensen's inequality on can show that $\psi(2(n-1))<\ln(2n)$. This proves that the unrestricted MLE underestimates $\ln \sigma$.
%\end{remark}
As an application of Theorem \ref{stth1} we have the following result. 
\begin{theorem}
	For estimating $\tau$ with respect to a location invariant loss $L(t)$ function the estimators 
	\begin{equation*}\label{imle}
		\delta_{IM}=\left\{
		\begin{array}{ll}
			\ln S+ \max\Big\{-\tfrac12\ln(2n),\ m_0 \ + ~\ln \sqrt{1+\frac{nW^2}{2}}\Big\},W<0\\
			\ln S+ \min\Big\{-\tfrac12\ln(2n),\ m_0 \ + ~\ln \sqrt{1+\frac{nW^2}{2}}\Big\},W>0
		\end{array}
		\right.
	\end{equation*}
	and 
	\begin{equation}\label{irmle}
		\delta_{IRM}=\left\{
		\begin{array}{ll}
			\ln S+ \max\left\{\ln \sqrt{\left(1+\frac{2n}{4}W^2\right)/(2n)},\ m_0 + \ln \sqrt{1+\frac{nW^2}{2}}\right\},~W<0\\
			\ln S+ \min\left\{\ln\sqrt{1/(2n)},\  m_0 \ + \ln \sqrt{1+\frac{nW^2}{2}}\right\},~ W>0
		\end{array}
		\right.
	\end{equation}
	dominates $\delta_{M}$ and $\delta_{RM}$ respectively. 
\end{theorem}
\section{Smooth improved estimator}\label{sec3}
In this section we will find an smooth estimator of $\tau$ that improves upon the BAEE $\delta_{0}$ under the restriction $\mu_{1}\leq \mu_{2}$. To obtain smooth improved estimator of $\tau$, we have to analyze the conditional risk function given by 
\begin{equation}\label{eq4.1}
	\mathcal{H}_{\mu}(r)=E_{\mu}\left[L(\ln \sqrt{v} + r) \vert |W|\leq \alpha\right], \ r>0,
\end{equation}
where $\alpha$ is a positive real number. From (\ref{eq4.1}), we have
\begin{equation}\label{eq4.3}
	\mathcal{H}_{\mu}(r)=\int_{0}^{\infty}L(\ln \sqrt{v} + r)f_{\mu}(v,\alpha)dv,
\end{equation}
where
\begin{equation*}
	f_{\mu}(v,\alpha) \propto v^{\frac{2n-3}{2}} e^{-\frac{v}{2}}\int_{-\alpha}^{\alpha}e^{-\frac{1}{4}\left(\sqrt{nv}w - \mu\right)^2}dw.
\end{equation*}
Let $z=\ln \sqrt{v}$. Then we get the density of $Z$ as
\begin{equation}
	f_{\mu}(z,\alpha) \propto e^{(2n-1)z} e^{-\frac{1}{2}e^{2z}}\int_{-\alpha}^{\alpha}e^{-\frac{1}{4}\left(\sqrt{n}e^zw - \mu\right)^2}dw.
\end{equation}
Thus (\ref{eq4.3}) can be rewritten as 
\begin{equation}
	\mathcal{H}_{\mu}(r)=\int_{-\infty}^{\infty}L(z + r)f_{\mu}(z,\alpha)dz.
\end{equation}
Let we denote $r_{\mu}(\alpha)$ be the minimizer of $\mathcal{H}_{\mu}(r)$ which implies that $r_0(\alpha)$ is the minimizer of $\mathcal{H}_{0}(r)$. In the next lemma, we obtain some characteristic of the function $r_{\mu}(\alpha)$, where $\mu \in \mathbb{R}$ and $\alpha>0$.

\begin{lemma} \label{lmm4.1}
	\begin{enumerate}
		\item [(i)] $\mathcal{H}_{\mu}(r)$ is strictly bowl shaped function of $r$ for each $\alpha>0$;
		\item[(ii)] $r_{\mu}(\alpha)\leq r_0(\alpha)$;
		\item[(iii)] $r_0(\alpha)$ is increasing in $\alpha$.
	\end{enumerate}
\end{lemma}
\noindent\textbf{Proof:} The detailed proof is given in the appendix.

We consider an estimator of the form
\begin{equation}
	\delta_{r(\alpha)}(\underline{X},S) = \begin{cases} 
		\ln S + r(\alpha), & |W| \leq \alpha, \\
		\ln S + d_0, & \text{otherwise}.
	\end{cases}
\end{equation}
\begin{theorem} Let $d_0$ be the unique solution to (\ref{d0}) and $r_0(\alpha)$ be the minimizer of $\mathcal{H}_0(r)$. Then, the estimator
	\begin{equation}
		\delta_{r_0(\alpha)}(\underline{X},S) = \begin{cases} 
			\ln S + r_0(\alpha), & |W| \leq \alpha, \\
			\ln S + d_0, & \text{otherwise}.
		\end{cases}
	\end{equation} 
	has smaller risk than the BAEE $\delta_{0}$ under the general location invariant loss function $L(t)$. 
\end{theorem}
\noindent\textbf{Proof:} We have from lemma \ref{lmm4.1} that $r_0(\alpha)$ is increasing in $\alpha$. Then for $\alpha>0$ it is easy to show that  $r_0(\alpha)\leq \lim\limits_{\alpha \to 0}r_0(\alpha)=d_0$. Thus by using the risk function's convexity property, we obtain the desired result. Furthermore let $\alpha'$ such that $0<\alpha'<\alpha<\infty$. Again from Lemma \ref{lmm4.1}, $r_0(\alpha')$ is increasing in $\alpha'$. Then following the above argument, one can show that the estimator of the form $\delta_{r_0(\alpha),\alpha,\alpha'}(\underline{X},S)=\ln S +\psi_{\alpha,\alpha'}(W)$ has smaller risk than the estimator $\delta_{0}$, where
\begin{equation*}
	\psi_{\alpha,\alpha'}(W) = \begin{cases} 
		r_0(\alpha'), & |W| \leq \alpha',\\
		r_0(\alpha), & \alpha'<|W| \leq \alpha,\\
		d_0, & \text{otherwise}.
	\end{cases}
\end{equation*} 
Using the same argument as above for a finite partition of $[0,\infty)$ as $0 < \alpha_{l,0} < \alpha_{l,1} < \alpha_{l,2}, \dots ,< \alpha_{l,n_l}$.  We define a function as 
\begin{equation*}
	\psi_{l}(w) = \begin{cases} 
		r_0(\alpha_{l,1}), & \alpha_{l,0}<W \leq \alpha_{l,1},\\
		r_0(\alpha_{l,2}), & \alpha_{l,1}<W \leq \alpha_{l,2},\\
		r_0(\alpha_{l,3}), & \alpha_{l,2}<W \leq \alpha_{l,3},\\
		\vdots\\
		r_0(\alpha_{l,n_{l}}), & \alpha_{l,n_{l-1}}<W \leq \alpha_{l,n_{l}},\\
		d_0, & \text{otherwise}.
	\end{cases}
\end{equation*} 
Now we assume that  $\displaystyle \max_{1\leq t \leq n_{l}}|\alpha_{l,t}- \alpha_{l, t-1}| \rightarrow 0$ and $\alpha_{l,n_{l}} \rightarrow 0$ as $l \rightarrow 0$. Then, $\phi_l(w) \rightarrow \phi_{*}(w)$, where $\phi_{*}(w)$ is given as 
\begin{equation*}
	\psi_{*}(W) = \begin{cases} 
		r_0(|W|), & |W| \leq \alpha,\\
		d_0, & \text{otherwise}.
	\end{cases}
\end{equation*} 
Now by using the Fatou's Lemma we get the desired result.
\begin{theorem}
	Under the general location invariant loss function $L(t)$, the risk of the estimator $\delta_{SE}(\underline{X},S)= \ln S + \psi_{*}(W)$ is nowhere larger than that of $\delta_{0}$.
\end{theorem}
\begin{remark}
	The smooth improved estimator $\delta_{SE}(\underline{X},S)$ is minimax.	
\end{remark}
\begin{example}
	Consider the loss function $L_1(t)=t^2$. Then the estimator $\delta^1_{SE}(\underline{X},S)= \ln S + \psi_{*1}(W)$, where  \begin{equation*}
		\psi_{*1}(W) = \begin{cases} 
			r_0(|W|), & |W| \leq \alpha,\\
			d_0, & \text{otherwise}.
		\end{cases}
	\end{equation*} with 	
	\begin{equation*}
		\begin{aligned}
			r_0(|w|)
			= -\tfrac12 \Bigg( \psi(a) + \ln 4 -\frac{\displaystyle
				\int_{0}^{n|w|^2}
				t^{-\tfrac12}(2+t)^{-a}\ln(2+t)\,dt}
			{\displaystyle
				\int_{0}^{n|w|^2}
				t^{-\tfrac12}(2+t)^{-a}\,dt}
			\Bigg), ~\text{where } a = n - \tfrac12 .
		\end{aligned}
	\end{equation*}
\end{example}
\begin{example}
	Let $L_2(t)=e^{a_1t}- a_1t-1,\ a_1 \neq0$. Then the estimator $\delta^2_{SE}(\underline{X},S)= \ln S + \psi_{*2}(W)$, where  \begin{equation*}
		\psi_{*2}(W) = \begin{cases} 
			r_0(|W|), & |W| \leq \alpha,\\
			d_0, & \text{otherwise}.
		\end{cases}
	\end{equation*} with 	
	\begin{equation*}
		r_0(|w|)=  \frac{1}{a_1} \ln\left(\frac{\Gamma(a)\int_{0}^{n|w|^2} t^{-\tfrac{1}{2}} (2+t)^{-a}\, dt}	{4^{\tfrac{a_1}{2}}\Gamma\left(a+\tfrac{a_1}{2}\right)\int_{0}^{n|w|^2} t^{-\tfrac{1}{2}} (2+t)^{-\left(a+\tfrac{a_1}{2}\right)}\,dt}
		\right),\, \mbox{where}\, \ a=n-\frac{1}{2}.
	\end{equation*}
\end{example}
\subsection{A family of improved estimators} 
In this subsection we derive a class of improved estimators using the integral expression of risk difference (IERD) of \cite{kubokawa1994unified}. We consider the class of estimator which is of the form
\begin{equation}\label{e1}
	\delta_{\phi}(X,S) = \ln \sqrt{V} + \phi(T)
\end{equation}
where $T=\frac{(Z_2-Z_1)^2}{nV}$. The joint probability density function of $T=\frac{U_1}{nV}=\frac{U_1}{U_2}$ and $V$ is
\begin{equation}
	f(t,v;\eta)=n^2vg_{U_1}(ntv:\eta)h_{U_2}(nv);\ t>0,\ v>0
\end{equation} 
where $U_1=(Z_2-Z_1)^2 \sim 2\chi_1^2\left(\frac{\eta^2}{2}\right)$
%\frac{1}{2^n\Gamma(n-1)\sqrt{\pi}}e^{-\left(\frac{w\sqrt{v}-\eta}{2}\right)^2}v^{n-\frac{3}{2}}
Further we define 
\begin{align*}
	G_{\eta}(x)=\int_{0}^{x}g_{U_1}(u_1:\eta)du_1 \ ~\mbox{and}\ ~ G(x)=\int_{0}^{x}g_{U_1}(u_1;0)du_1\ 
\end{align*}
where $g_{U_1}(u_1;\eta)$ is density of $U_1$ and $h_{U_2}(u_2)$ is density of $U_2$.
\begin{theorem}\label{thierd}
	Let $\phi(y)$ be a function satisfying the following properties:
	\begin{enumerate}
		\item [(i)] $\phi(y)$ is a nondecreasing function with $\lim\limits_{y\rightarrow \infty}\phi(y)=d_{0}$,
		\item [(ii)] The integral $\int_{0}^{\infty}L'\left(\ln\sqrt{v}+\phi(y)\right)G(nyv)h(nv)dv$ is non-negative.
	\end{enumerate}
	Under these conditions and for any location-invariant loss function $L(t)$, the risk of the estimator $\delta_{\phi}(\underline{X},S)$ (as defined in (\ref{e1})) is never greater than the risk of the estimator $\delta_{0}$.
\end{theorem}
\noindent\textbf{Proof:} The detailed proof is given in the appendix.
\begin{corollary}
	Under the $L_1(t)$ loss function, the risk of the estimator $\delta_{\phi}$ given in (\ref{d1}) is nowhere larger than that of $\delta_{01}$ provided $\phi(\cdot)$ satisfies the following conditions:
	\begin{enumerate}
		\item[(i)] $\phi(y)$ is increasing function in $y$ and $\lim\limits_{y\rightarrow\infty}\phi(y)=-\frac{1}{2} \left[ \ln2+\psi(n-1) \right]$;
		\item[(ii)] $\phi(y)\geq\phi_{*1}(y)$, where
		%		\begin{align*}
			%		\phi_{*1}(y)=- \frac{\int_{0}^{\infty}\ln(v)v^{n-2}e^{-\frac{v}{2}}\int_{0}^{nyv}x^{-\frac{1}{2}}e^{-\frac{x}{4}}dx dv}{2\int_{0}^{\infty}v^{n-2}e^{-\frac{v}{2}}\int_{0}^{nyv}x^{-\frac{1}{2}}e^{-\frac{x}{4}}dx dv}.
			%		\end{align*}
		\begin{align*}
			\phi_{*1}(y) =-\tfrac{1}{2}\left(\psi(a) + \ln 4 -\frac{ \int_{0}^{ny} t^{-\tfrac{1}{2}} (2+t)^{-a} \ln(2+t)\,dt}{ \int_{0}^{ny} t^{-\tfrac{1}{2}} (2+t)^{-a}\,dt}\right), \ \mbox{with}\ a=n-\frac{1}{2}.
		\end{align*}		
	\end{enumerate}
\end{corollary}
\begin{corollary}
	Under the loss function $L_2(t)$, the risk of the estimator $\delta_{\phi}$ given in (\ref{d1}) is nowhere larger than that of $\delta_{02}$ provided the function $\phi(\cdot)$ satisfies the following conditions:
	\begin{enumerate}
		\item[(i)] $\phi(y)$ is increasing function in $y$ and $\lim\limits_{y\rightarrow\infty}\phi(y)= - \frac{1}{a_1}\ln\left[2^{\frac{a_1}{2}}\frac{\Gamma\left(n-1+\frac{a_1}{2}\right)}{\Gamma\left(n-1\right)}\right]$;
		\item[(ii)] $\phi(y)\geq\phi_{*2}(y)$, where	
		\begin{align*}
			\phi_{*2}(y) = \frac{1}{a_1} \ln \left(\frac{ \Gamma(a)\int_{0}^{ny} t^{-\tfrac{1}{2}} (2+t)^{-a}\, dt} { 4^{\tfrac{a_1}{2}}\Gamma\left(a+\tfrac{a_1}{2}\right)\int_{0}^{ny} t^{-\tfrac{1}{2}} (2+t)^{-\left(a+\tfrac{a_1}{2}\right)}\,dt}
			\right),\, \mbox{with}\, \ a=n-\frac{1}{2}.
		\end{align*}		
	\end{enumerate}
\end{corollary}
\begin{remark}
	From the above corollaries we observe that the \cite{brewster1974improving}-type estimators is coincide with \cite{kubokawa1994unified} IERD type estimators. It can be observed that the boundary estimators of this class of imporved estimators for the loss functions $L_1(t)$ and $L_2(t)$ are $\delta_{\phi_{*1}} = \delta_{\psi_{*1}}$ and $\delta_{\phi_{*2}}=\delta_{\psi_{*2}}$ respectively.
\end{remark}
\section{Improved estimation using Generalized Pitman Closeness}\label{sec4}
In this section, we derive the estimation of $\tau$ under the generalized Pitman closeness criterion. A concise discussion of the Pitman closeness criterion can be found in \cite{garg2024unified}. The Pitman closeness criterion was first proposed by \cite{pitman1937closest}. The generalized Pitman closeness (GPC) criterion (see \cite{nayak1990estimation} and \cite{kubokawa1991equivariant}) under a location invariant loss function $W(\cdot,\cdot)$ is defined as follows. 
\begin{defn}
	Let $X$ be a random variable having a distribution depends on an unknown parameter $\theta \in \Theta$. Let $\delta_1$ and $\delta_2$ be two estimators of a real-valued estimand $\xi(\theta)$. Also, let $G(\delta,\xi(\theta))$ be a location invariant loss function for estimating $\xi(\theta)$. Then, the GPC of $\delta_1$ relative to $\delta_2$ is defined by
	\begin{align*}
		\mbox{GPC}(\delta_1,\delta_2;\theta)= P_{\theta}\big[G(\theta,\delta_1)<G(\theta,\delta_2)\big] + \frac{1}{2}P_{\theta}\big[G(\theta,\delta_1)=G(\theta,\delta_2)\big], \ \ \theta\in\Theta.
	\end{align*}
	The estimator $\delta_1$ is said to be closer to $\xi(\theta)$ than the estimator $\delta_2$, under the GPC criterion, if $\mbox{GPC}(\delta_1,\delta_2;\theta)\geq \frac{1}{2}$ $\forall \  \theta \in \Theta$, and strict inequality hold for some $\theta \in \Theta$.
\end{defn}
We need the following lemma taken from \cite{garg2024unified} to prove the main result of this section.
\begin{lemma}(\cite{garg2024unified})\label{pt1} Let $Y$ be a random variable having a probability density function and let $m_Y$ be the median of $Y$. Let $L(t)$ be a non negative function such that $L(0)=0$, $L(t)$ is strictly increasing for $t>0$ and strictly decreasing for $t<0$. Then, for $-\infty<c_1<c_2\leq m_Y$ or $-\infty<m_Y\leq c_2<c_1$, $\mbox{GPC}=P[L(Y-c_2)<L(Y-c_1)]+\frac{1}{2}P[L(Y-c_2)=L(Y-c_1)]>\frac{1}{2}$.
\end{lemma}
Note that, the problem of estimating $\tau$ for unrestricted case (the parametric space $\Theta=\mathbb{R}^3$), the GPC criterion is invariant under the group of transformation $\mathcal{G}$. Any location equivariant estimator is of the form $\delta_c=\ln \sqrt{V} - c$, $c=-d \in \mathbb{R}$. Under the unrestricted parameter space $\Theta=\mathbb{R}^3$, an immediate consequence of the Lemma \ref{pt1} is that the Pitman closest equivariant estimator (PCE) of the location parameter $\tau$, under the GPC criterion is $\delta_{PCE}=\ln \sqrt{V} - m_0$, where $m_0$ is the median of the random variable $\ln V$. We denote $\underline{\gamma}=(\mu_1,\mu_2,\sigma)$ and $\Theta_0=\big\{(\mu_1,\mu_2,\sigma):\ \mu_1,\ \mu_2\in \mathbb{R},\ 0<\sigma<\infty\}$.
\subsection{Improve estimation of $\tau$}
Consider two location equivariant estimator of $\tau$, given by $\delta_{\psi_1}(\underline{X},S)=\ln S -\psi_1(W)$ and $\delta_{\psi_2}(\underline{X},S)=\ln S -\psi_2(W)$ where $\psi_1$ and $\psi_2$ are real-valued functions on $\mathbb{R}$. Then, the GPC of $\delta_{\psi_1}(\underline{X},S)$ with respect to $\delta_{\psi_2}(\underline{X},S)$ is define by
\begin{align*}
	\mbox{GPC}(\delta_{\psi_1},\delta_{\psi_2};\eta)&= P_{\eta}\big[L(\ln S - \psi_1(W) - \tau)  < L(\ln S - \psi_2(W) - \tau)\big]\  + \\
	&~~~~~~~~~~~~~~~~~~~~~~~~~~~~~~~~~~~~~~\frac{1}{2} P_{\eta}\big[L(\ln S - \psi_1(W) - \tau) = L(\ln S - \psi_2(W) - \ln \sigma)\big]\\
	&=E^W\Big[P\big[L(\ln \sqrt{V} - \psi_1(W)) < L(\ln \sqrt{V} - \psi_2(W))|W\big]\Big]\ +\\
	&~~~~~~~~~~~~~~~~~~~~~~~~~~~~~~~~~~~~\frac{1}{2} E^W\Big[P\big[L(\ln \sqrt{V} - \psi_2(W)) < L(\ln \sqrt{V} - \psi_2(W))|W\big]\Big]\\
\end{align*}
For any fixed $w$, we define 
\begin{align*}
	h_{\eta}(\psi_1(w),\psi_2(w),w) = P_{\eta}\big[L(\ln \sqrt{V} - \psi_1(w)) &< L(\ln \sqrt{V} - \psi_2(w))\big| W=w\big]\ + \\
	&~~~~\frac{1}{2} P_{\eta}\big[L(\ln \sqrt{V} - \psi_1(w) ) = L(\ln \sqrt{V} - \psi_2(w))\big | W=w\big]
\end{align*}
For any fixed $\eta\geq0$ and $w$, let $m_\eta(w)$ represent the median of the conditional distribution of $Z=\ln V$ given $W=w$. The distribution of $Z=\ln V$, conditional on $W=w$ is expressed as
$$f_{\eta}(z|w)\propto e^{(2n-1)z-\frac{1}{2}[e^{2z} + \frac{1}{2}(\sqrt{n}e^zw-\eta)^2]}$$
Thus, \begin{align*}
	\int_{0}^{m_\eta(w)}e^{(2n-1)z-\frac{1}{2}[e^{2z} + \frac{1}{2}(\sqrt{n}e^zw-\eta)^2]}dz =\frac{1}{2}\int_{-\infty}^{\infty} e^{(2n-1)z-\frac{1}{2}[e^{2z} + \frac{1}{2}(\sqrt{n}e^zw-\eta)^2]}dz.
\end{align*}
For any fixed $w$, by using the Lemma \ref{pt1}, we obtain $h_{\eta}\big(\psi_1(w),\psi_2(w),w\big) >1/2$ $\forall \ \eta \geq 0$, if $\psi_2(w)<\psi_1(w)\leq m_\eta(w)$ $\forall \ \eta \geq 0$ or if $m_\eta(w)\leq\psi_1(w)<\psi_2(w)$ $\forall \ \eta \geq 0$. Moreover, for any fixed $w$, $h_{\eta}(\psi_1(w),\psi_2(w),w)=1/2$ $\forall \ \eta \geq 0$.
\begin{theorem}
	Consider a location equivariant estimator of $\tau$ given by $\delta_{\psi}(\underline{X},S)=\ln S -\psi(W)$. For any $w$, let $l(w)$ and $u(w)$ be functions satisfying $l(w)\leq m_\eta(w)\leq u(w)$ $\forall \ \eta \geq 0$, where $m_\eta(w)$ denotes the median of the conditional distribution of $Z=\ln V$ given $W=w$. Define $\widetilde{\psi}(w)=\max\left\{l(w),\min\left\{\psi(w), u(w)\right\}\right\}$. Then under GPC criterion with a general loss function $L(t)$ the estimator $\delta_{\widetilde{\psi}}(\underline{X},S)= \ln S - \widetilde{\psi}(W)$ is closer to $\tau$ in the Pitman sense than the estimator $\delta_{\psi}(\underline{X},S)=\ln S - \psi (W)$ $\forall \ \underline{\gamma} \in \Theta_0$, provided that $P\big[l(w)\leq \psi(w)\leq u(w)\big]<1$ $\forall \ \underline{\gamma} \in \Theta_0$.
\end{theorem}
\noindent \textit{Proof:} The GPC of the estimator $\delta_{\widetilde{\psi}}(\underline{X},S)=\ln S -\widetilde{\psi}(W)$ relative to $\delta_{\psi}(\underline{X},S)=\ln S -\psi(W)$ can be written as $\mbox{GPC}(\delta_{\widetilde{\psi}},\delta_{\psi},\tau)=\int_{-\infty}^{\infty}h_{\eta}(\widetilde{\psi}(w),\psi(w),w)f_w(w|\eta) dw,\ \eta \geq 0$. Let $A = \left\{w: \psi(w)<l(w)\right\}$, $B=\left\{w: l(w)<\psi(w)<u(w)\right\}$, $C=\left\{w: \psi(w)>u(w)\right\}$. It is clear to us 
\begin{equation*}
	\widetilde{\psi}(w) = \begin{cases} 
		l(w), & w \in A,\\
		\psi(w), & w\in B,\\
		u(w), & w\in C.
	\end{cases}
\end{equation*} 
Because $l(w)\leq m_{\eta}(w)\leq u(w)$ $\forall \ \eta \geq 0$ and $w$. Using Lemma \ref{pt1}, we have, $h_{\eta}(\widetilde{\psi}(w),\psi(w),w)>\frac{1}{2}\ \forall\ \eta \geq 0$ provided $w \in A \cup C$.
For $w\in B$, $h_{\eta}(\widetilde{\psi}(w),\psi(w),w)=\frac{1}{2}$ $\forall \ \eta \geq 0$. Again since $P_{\underline{\gamma}}(A\cup C)>0$ $\forall\ \underline{\gamma} \in \Theta_0$ we have
\begin{align*}
	\mbox{GPC}(\delta_{\widetilde{\psi}},\delta_{\psi},\underline{\gamma})&=\int_Ah_{\eta}(\widetilde{\psi}(w),\psi(w),w)f_w(w|\eta) dw +\int_B h_{\eta}(\widetilde{\psi}(w),\psi(w),w)f_w(w|\eta) dw \\
	&~~~~~~~~~~~~~~~~~~~~~~~~~~~~~~~~~~~~~~~~~~~~~~~~~~~~~~~~~~+\int_Ch_{\eta}(\widetilde{\psi}(w),\psi(w),w)f_w(w|\eta) dw\\
	&>\frac{1}{2} ~~~~~~\forall~ \underline{\gamma}\in \Theta_0
\end{align*} 
\begin{corollary}
	Let $l(w)$ and $u(w)$ be as defined as above, suppose that $P\big[l(w)\leq m_{0} \leq u(w)\big]<1$ $\forall \ \underline{\gamma} \in \Theta_0$. For any fixed $w>0$ defined $\widetilde{\psi}(w)=\max \left\{ l(w), \min\left\{ m_{0},u(w)\right\}\right\}=\min\left\{m_{0},u(w)\right\}$. Then for any every $\underline{\gamma} \in \Theta_0$ the estimator $\delta_{\widetilde{\psi}}(\underline{X},S)=\ln S-\widetilde{\psi}(W)$ is Pitman closest to $\tau$ than the (Pitman closest afine equivariant estimator) PCAEE $\delta_{p}(\underline{X},S)=\ln S - m_{0}$ under the GPC criterion.
\end{corollary}
Note that the following corollary provides improvements over the unrestricted BAEE $\delta_{0}(S)=\ln S -d_{0}$.
\begin{corollary}
	Let $l(w)$ and $u(w)$ be defined as above, and suppose $P\big[l(w)\leq d_{*} \leq u(w)\big]<1$ $\forall \ \underline{\gamma} \in \Theta_0$. For each $w$, define  $\xi^*(w)=\max \left\{ l(w), \min\left\{ d_{*} , u(w) \right\} \right\}$. Then, for any every $\underline{\gamma} \in \Theta_0$, the estimator $\delta_{\xi^*}(\underline{X},S)=\ln S-\xi^*(W)$ is closer to $\tau$ in the Pitman sense than the Pitman closest affine equivariant estimator $\delta_{p}(\underline{X},S)=\ln S - d_{*}$ under the GPC criterion.
\end{corollary} 

\section{Confidence interval estimation}\label{sec6}
In this section, we study interval estimation of the parameter $\tau=\ln\sigma$, using random samples from two normal populations with identical variances but distinct means. In particular, we consider asymptotic confidence intervals, bootstrap confidence intervals, generalized confidence intervals and highest posterior density intervals. Finally, the performance of these interval estimation procedures is evaluated through simulation studies. The comparison is carried out in terms of coverage probability (CP), average length (AL), and probability coverage density (PCD), in order to identify the most efficient method under different sample sizes and parameter settings.

%we will consider various confidence interval estimation of $\tau$ and we will compare the performance of proposed confidence interval using simulation. The used technique can be seen in \cite{jena2023point}  for estimating the confidence interval for power of scale parameters for two normal population.
\subsection{Asymptotic confidence intervals}\label{sec61}
In this subsection, we derive an asymptotic confidence interval for the parameter $\tau$. To do so, we consider the log-likelihood function based on the observations $\underline{x}_1=(x_{11},x_{12},\dots,x_{1n})$ and $\underline{x}_2=(x_{21},x_{22},\dots,x_{2n})$ taken from the population $N(\mu_1,\sigma^2)$ and $N(\mu_2,\sigma^2)$ which is given by
\begin{equation}
	\begin{array}{ll}
		l(\mu_{1},\mu_{2},\sigma)= -\ n\ln(2\pi)- n\ln\sigma^2-\frac{1}{2\sigma^2} \ -
		\big[\sum_{i=1}^{n}\left(x_{1i}-\mu_1\right)^2+\sum_{i=1}^{n}\left(x_{2i}-\mu_2\right)^2\big]
	\end{array}
\end{equation}
Denote $\theta_1=\mu_{1}$, $\theta_2=\mu_{2}$ and $\theta_3=\sigma$. Under this notation, the Fisher information matrix can be written as 
$$I(\theta_1,\theta_2,\theta_3)=(I_{ij})$$
where, $I_{ij}=E(P_{ij})$, and $P_{ij}=-\frac{\partial^2L}{\partial\theta_i\partial\theta_j}$, $i, j =1, 2, 3$. By using the basic calculus, calculating all the above expressions and after simplification, we get the inverse of the fisher information matrix as bellow
$$I(\mu_{1},\mu_{2},\sigma)^{-1}=\begin{pmatrix}
	\frac{\sigma^2}{n}&0&0\\
	0&\frac{\sigma^2}{n}&0\\
	0&0&\frac{\sigma^2}{4n}
\end{pmatrix}$$
Now we use well known Delta method to derive the asymptotic confidence interval for the parameter $\tau$. Let $g_1(\mu_1,\mu_2,\sigma)=\tau$, with a non zero gradient. Then, $\nabla g_1(\mu_1,\mu_2,\sigma)=\left(\frac{\partial g_1}{\partial\mu_1},\frac{\partial g_1}{\partial\mu_2},\frac{\partial g_1}{\partial \sigma}\right)^t=\left(0,0,\frac{1}{\sigma}\right)^t$. Consequently, $g_1(\hat{\mu_1}_{M},\hat{\mu_2}_{M},\sqrt{T_{M}})=\delta_{M}$ is asymptotically distributed as $N\left(\tau,\nabla g_1^tI^{-1}\nabla g_1\right)$, where we have
$$\nabla g_1^tI^{-1}\nabla g_1=\begin{pmatrix}
	0\\
	0\\
	\frac{1}{\sigma}
\end{pmatrix}^t\begin{pmatrix}
	\frac{\sigma^2}{n}&0&0\\
	0&\frac{\sigma^2}{n}&0\\
	0&0&\frac{\sigma^2}{4n}
\end{pmatrix}\begin{pmatrix}
	0\\
	0\\
	\frac{1}{\sigma}
\end{pmatrix}=\frac{1}{4n}$$ 
Utilizing this result, then the $(1-\alpha)100\%$ asymptotic confidence interval of $\tau$ is obtained as 
\begin{equation}
	\left(\delta_{M} - z_{\alpha/2}\frac{1}{2\sqrt{n}},\ \delta_{M} + z_{\alpha/2}\frac{1}{2\sqrt{n}}\right)
\end{equation}
where $z_{\alpha}$ denotes the $\alpha$-th quantile of a standard normal random variable. 
\subsection{Parametric bootstrap confidence intervals}\label{sec62}
In this section, we construct confidence interval for the parameter $\tau$ using the parametric bootstrap approach. We discuss two widely used bootstrap procedures: the bootstrap percentile (boot-p) and the bootstrap-t (boot-t) confidence intervals. The bootstrap-p interval introduced by \cite{efron1982jackknife} and the bootstrap-t interval proposed by \cite{hall1988bootstrap}. The performance of these bootstrap intervals depends on factors such as sample size, the underlying distribution, and the degree of skewness present in the data. Therefore, a comparative analysis of bootstrap-p and bootstrap-t intervals is essential to determine their effectiveness in practical situations. The detailed algorithms for obtaining these two approximate confidence intervals for the parameter $\tau$ are presented below.

\subsubsection{Bootstrap-p method}\label{sec621}
\begin{itemize}
	\item[] Step I: Consider two independent random samples $(X_{11},X_{12},\dots,X_{1n})$ and $(X_{21},X_{22},\dots,X_{2n})$ from two normal populations $N(\mu_{1},\sigma^2)$ and $N(\mu_{2},\sigma^2)$ respectively. Compute the MLEs $\hat{\mu_1}_{M}$, $\hat{\mu_2}_{M}$ and $\sqrt{T_{M}}$ numerically. Using these MLEs, we calculate the MLE of the parameter $\tau$ as $\lambda=\delta_{M}$.
	\item[] Step II: By using the MLEs $\hat{\mu_1}_{M}$, $\hat{\mu_2}_{M}$ and $T_{M}$ generate bootstrap samples $(X_{11}^*,X_{12}^*,\dots,X_{1n}^*)$ and $(X_{21}^*,X_{22}^*,\dots,X_{2n}^*)$ from  $N(\hat{\mu_1}_{M},T_{M})$ and $N(\hat{\mu_2}_{M},T_{M})$ respectively, and compute the bootstrap MLEs $\hat{\mu_1}_{BM}$, $\hat{\mu_2}_{BM}$ and $T_{BM}$. Consequently, compute the bootstrap MLE of $\tau$ as $\lambda^*=\delta_{BM}.$
	\item[] Step III: Repeat the step-II for a large number of times, say $K$ times and obtain the bootstrap MLEs of $\tau$ as $\lambda_1^*, \lambda_2^*, \dots, \lambda_K^*$.
	\item[] Step IV: Utilizing the bootstrap MLEs, the $(1-\alpha)100\%$ bootstrap-p confidence interval for $\tau$ is obtained as 
	$$\left(\lambda^*_{Boot-p}(\alpha/2),\lambda^*_{Boot-p}(1-\alpha/2)\right),$$
	where $\lambda^*_{Boot-p}(x)=F_1^{-1}(x)$, $F_1(x)=P(\lambda^*\leq x)$.
\end{itemize}
\subsubsection{Bootstrap-t method}\label{sec622}
\begin{itemize}
	\item[] Step I: First step is identical with the step I of the bootstrap-p method.
	\item[] step II: This step is also same as step II of the bootstrap-p method. Compute $\lambda^*=\ln\hat{\sigma}_{BML}$.
	\item[] Step-III: Calculate the statistics $\mathcal{U}=\frac{\lambda^*-\lambda}{\sqrt{Var(\lambda^*)}}$.
	\item[] Step-IV: Repeat the step-II and step-III for a large number of times, say $K$ times. Thus the bootstrap MLEs of $\tau$ are obtain as $\lambda_1^*,\lambda_2^*,\dots,\lambda_K^*$ and the corresponding value of $\mathcal{U}$ are $\mathcal{U}_1,\mathcal{U}_2,\dots,\mathcal{U}_K$.
	\item[] Step V: The approximate $(1-\alpha)100\%$ bootstrap-t confidence interval of the parameter $\tau$ are obtain as $$\left(\lambda^*_{Boot-t}(\alpha/2),\lambda^*_{Boot-t}(1-\alpha/2)\right)$$
	where $\lambda^*_{Boot-t}(x)=\lambda + H_1^{-1}(x)\sqrt{var(\lambda^*)}$, $H_1(x)=P(\mathcal{U}\leq x)$.
\end{itemize}
\subsection{Generalized confidence interval}\label{sec63}
In this subsection, we apply the generalized variable approach proposed by \cite{weerahandi1993generalized} to constructing generalized confidence intervals for functions of parameter. In recent studies, this method has been successfully applied by \cite{krishnamoorthy2003inferences} to perform hypothesis testing and to obtain confidence intervals for the common mean across multiple normal populations. Similar approaches and their extensions have also been discussed in related studies such as \cite{jana2019inference} and \cite{khatun2022hypothesis}. We start by describing the procedure for developing a generalized pivotal variable, which is used to compute this interval. The following definition provides a foundation for constructing the generalized confidence intervals. 
\begin{defn}
	Let $X$ be any random variable whose distribution depends on the parameter $\theta$. A statistic $\mathcal{T}=\mathcal{T}(X;x,\theta)$ is called a generalized pivot quantity (GPQ) for constructing the generalized confidence interval of $\theta$ if it satisfies the following properties.
	\begin{itemize}
		\item [(i)] For a fixed $X=x$, the distribution of $\mathcal{T}(X;x,\theta)$ is independent of unknown parameter.
		\item [(ii)] The observed value of $\mathcal{T}$ (when $X=x$) is the parameter of interest $\theta$.
	\end{itemize}
\end{defn}
We now present generalized pivotal quantities for constructing confidence intervals of $\tau$. Let $(x_1,x_2,s^2)$ be observed value of the sufficient statistics $(X_1,X_2,S^2)$. Hence the GPQ for the parameter $\tau$, constructed from the sufficient statistics $(X_1,X_2,S^2)$ and based on the unbiased estimator $\delta_{U}$ is given by
\begin{equation*}
	\delta_{U}=\ln S - \frac{1}{2} \left[ \ln 2 + \psi(n-1)\right].
\end{equation*}
As $V=S^2/\sigma^2$, then we have 
\begin{equation*}
	\delta_{U}=\tau + \frac{1}{2}\ln V - \frac{1}{2} \left[ \ln 2 + \psi(n-1)\right].
\end{equation*}
So the GPQ based on the UMVUE is 
\begin{equation}
	\mathcal{T}_{U}=\delta_U^{\tiny{\text{obs}}} - \frac{1}{2}\ln V + \frac{1}{2} \left[ \ln 2 + \psi(n-1)\right].
\end{equation}
where $\delta_U^{\tiny{\text{obs}}}$ is the observed values of the UMVUE and $V$ is generated from a known chi-square distribution with degrees of freedom $2(n - 1)$. Observed that, for a fixed $(x_1,x_2,s^2)$, the distribution of $T_U$ is free from all unknown parameters. Furthermore, when $(X_1,X_2,S^2)=(x_1,x_2,s^2)$, the value of $\mathcal{T}_U$ is $\tau$, the parameter of interest. Therefore, the $(1-\alpha)100\%$ confidence interval for $\tau$ is obtained as
\begin{equation}
	\big(\mathcal{T}_U(\alpha/2),\ \mathcal{T}_U(1-\alpha/2)\big).
\end{equation}

\subsection{Highest posterior density credible interval}\label{sec64}
In this section, we will obtained the HPD credible interval of the parameter $\tau$ using the posterior density under Jeffrey's prior. The Bayes estimator of the common mean $\mu$ is obtained using this non-informative prior distribution, as presented in \cite{mitra2007some}. Consider two normal random samples $\underline{X}_1=(X_{11},X_{12},\dots,X_{1n})$ and $\underline{X}_2=(X_{21},X_{22},\dots,X_{2n})$ taken from two normal populations $N(\mu_1,\sigma^2)$ and $N(\mu_1,\sigma^2)$ respectively. Then, the Jeffrey's prior for the parameter vector $(\mu_1,\mu_2,\tau)$ is obtained from the Fisher information matrix $I(\mu_1,\mu_2,\tau$) as follows
\begin{equation}
	\amalg(\mu_1,\mu_2,\tau)\propto \frac{1}{\sigma^2},\;~~ \mu_1,\mu_2 \in \mathbb{R},~ \sigma>0
\end{equation}
The joint posterior density of $(\mu_1,\mu_2,\tau)$ for given sample values is proportional to the product of the likelihood function and the Jeffrey's prior and can be expressed as
\begin{equation}
	\begin{array}{ll}
		\amalg\big((\mu_1,\mu_2,\tau)|\mbox{Data}\big) \propto\frac{\displaystyle e^{-\frac{1}{2\sigma^2}\left[\sum_{i=1}^{2}\sum_{j=1}^{n}(x_{ij}-\mu_1)^2\right]}}{\displaystyle\sigma^{2n+2}}
	\end{array}
\end{equation}
From the joint posterior density of $(\mu_1,\mu_2,\tau)$, the corresponding conditional densities of the parameters are obtained as follows:
\begin{enumerate}
	\item [(i)] given $\mu_2$, $\tau$ and the sample the conditional density of $\mu_1$ is 
	\begin{equation}
		\amalg\big(\mu_{1}|(\mu_2,\tau,\mbox{Data})\big) \sim N\left(X_1,\frac{\sigma^2}{n}\right)
	\end{equation}
	\item [(ii)] given $\mu_1$, $\tau$ and the sample the conditional density of $\mu_2$ is 
	\begin{equation}
		\amalg\big(\mu_{2}|(\mu_1,\tau,\mbox{Data})\big) \sim N\left(X_2,\frac{\sigma^2}{n}\right)
	\end{equation}
	\item [(i)] given $\mu_1$, $\mu_2$ and the sample the conditional density of $\tau$ is 
	\begin{equation}
		\amalg\big(\tau|(\mu_1,\mu_2,\mbox{Data})\big) \propto\frac{e^{-\frac{1}{2\sigma^2}\left[\sum_{i=1}^{2}\sum_{j=1}^{n}(x_{ij}-\mu_1)^2\right]} }{\sigma^{2n+2}}
	\end{equation}
\end{enumerate}
The conditional density of $\tau$ do not belongs to any standard parametric family. 
%However, the conditional distributions of $\mu_1$ and $\mu_2$ given $\tau$ and the observed data are normal. 
To derive the confidence interval for $\tau$, we used Gibbs sampling procedure to generate $\mu_1$ and $\mu_2$. Since $\tau$ can not be sampled directly from a standard distribution, it is generated using MCMC approach via the Random Walk Metropolis–Hastings algorithm. For further details on this algorithm, see \cite{hastings1970monte} and \cite{chib1995understanding}. The following procedures describes the Gibbs sampling and Random Walk Metropolis–Hastings.

\textbf{Gibbs Sampling Algorithm:} At iteration $k-1$, let the state of the Markov chain be $\underline{\delta}^{(k-1)}=\left(\mu_1^{(k-1)}, \mu_2^{(k-1)}, \tau^{(k-1)}\right)$. Then for $k=1,2,\dots,K$, the values of $\mu_1^{(k)}$ and $\mu_2^{(k)}$ are generated as follows:
\begin{equation}
	\mu_1^{(k)}\sim N\left(X_1,\frac{(\sigma^2)^{(k-1)}}{n}\right)\sim N\left(X_1,\frac{\beta^{(k-1)}}{n}\right)
\end{equation}
\begin{equation}
	\begin{aligned}
		\mu_2^{(k)}
		&\sim N\!\left(X_2,\frac{(\sigma^2)^{(k-1)}}{n}\right) \sim N\!\left(X_2,\frac{\beta^{(k-1)}}{n}\right), ~~~~\text{with } \beta=\sigma^2.
	\end{aligned}
\end{equation}
where, $\beta^{(0)}$ is appropriately chosen initial values. In this study, we set $\beta^{(0)}=s^2$. To update the values of $\beta^{(k)}$ we employ the following Random Walk Metropolis-Hasting algorithm.

\textbf{Random walk Metropolis-Hastings algorithm:} The Random Walk Metropolis–Hastings method can be carried out using the following steps:
\begin{itemize}
	\item[] Step I: At the beginning of iteration $k$, the Markov chain is assumed to be in the state $(\mu_1^{(k)},\mu_2^{(k)},\beta^{(k-1)})$.
	
	\item[] Step II: Generate sample $\beta^*$ from $N\left(\beta^{(k-1)},\sigma_{\beta}^2\right)$.
	
	\item[] Step III: Calculate $\vartheta=\min\left(1,\ \frac{\displaystyle\amalg\left(\mu_1^{(k)},\ \mu_2^{(k)},\ \beta^*\big|\mbox{Data}\right)}{\displaystyle\amalg\left(\mu_1^{(k)},\ \mu_2^{(k)},\ \beta^{(k-1)}\big|\mbox{Data}\right)}\right)$
	
	\item[] Step IV: Generate a uniform random sample $u\sim U(0,1)$. The proposed value $\beta^*$ is accepted if $u\leq\vartheta$, and the parameters  are updated to $\beta^{(k)}=\beta^*$. Otherwise we reject the proposal and take $\beta^{(k)}=\beta^{(k-1)}$.
	
	\item[] Step V: The transition from $\underline{\delta}^{(k-1)}$ to $\underline{\delta}^{(k)}$ is thus completed. Subsequently, $\tau_k=\ln\sigma^{(k)}$ is evaluated to generate posterior samples of $\tau$.
	
	\item[] Step VI: The aforementioned Gibbs Sampling algorithm and Random Walk Metropolis–Hastings algorithms are iterated until $M$ samples are obtained after discarding the burn-in process, that is $M=K-K_0$ where $K_0$ is the number of burn-in iteration. The resulting $M$ post burn-in values are treated as approximately independent samples and are denoted by $\tau_1,\tau_2,\dots,\tau_M$.
\end{itemize}
To obtain the $(1-\alpha)100\%$ HPD credible interval for $\tau$, we used the method of \cite{chen1999monte}, which is given as
\begin{equation}
	\mbox{HPD}=\left(\tau_{(r^*)},\tau_{(r^*+[(1-\alpha)M])}\right)
\end{equation}
where $r^*$ can be chosen such that
\begin{align*}
	\tau_{(r^*+[(1-\alpha)M])}-\tau_{(r^*)} = \min_{1\leq r\leq M-[(1-\alpha)M]}\left(\tau_{r+[(1-\alpha)M]}-\tau_{r}\right)
\end{align*}
with $\tau_{(r)}$; $r=1,2,\dots,M$ is the $r$th order statistics of $\left\{\tau_1,\tau_2,\dots,\tau_M\right\}$.

\section{Performance comparison of improved estimators through simulation} \label{sec5}
In this section we will study the risk performance of the estimators $\delta_{S}$, $\delta_{\psi_*}$ with the respect to the BAEE $\delta_0$. For this numerical comparison we have generated $70,000$ random samples from two normal population namely $N(\mu_1, \sigma^2)$ and $N(\mu_2,\sigma^2)$ for various values of $(\mu_1, \mu_2)$ and $\sigma^2$. Observe that the risk function of the improved estimators are depends on the parameter $\mu_1$, $\mu_2$ and $\sigma^2$ by $\eta=\frac{\mu_2-\mu_1}{\sigma}$. As $\mu_2\geq\mu_1$, so $\eta\geq0$ hence by without loss of generality we may assume that $\mu_1=0$, $\sigma=1$. The results of the numerical simulation studies are presented by graphs for various values of $n$. The performance of the improved estimators are measured by calculating the relative risk improvement (RRI) with the respect to the BAEE. The RRI of any estimator $\delta$ with the respect to the BAEE $\delta_{0}$ is defined by
\begin{equation*}
	\mbox{RRI}(\delta)=\frac{\mbox{Risk}(\delta_{0})-\mbox{Risk}(\delta)}{\mbox{Risk}(\delta_{0})}\times100
\end{equation*}
In this study, we have plotted the RRI of the improve estimators for $\tau$ in Figures \ref{fig:Q}, \ref{fig:L1} and \ref{fig:L2} under the loss functions $L_1(t)$ and $L_2(t)$ respectively. For the linex loss function, we have plotted the graphs for $a_1=-3$ and $a_1=2$. We now present the observations under the quadratic loss function $L_1(t)$ as shown in Figure \ref{fig:Q}. 
\begin{enumerate}
	\item [(i)] The RRI of $\delta_{S1}$ is a decreasing function of $\eta$. The RRI of the estimator $\delta_{\psi_{*1}}$ increases with $\eta$ for $\eta \leq 1$, and then decreases for $\eta \geq 1$ (approximately).% However, RRI of the estimator $\delta_{\psi_{*1}}$ is not monotone in $\eta$. The RRI increases as $\eta$ increasing and achieves its maximum in the region $0.5\leq \eta \leq 1.5$.
	
	\item [(ii)] For smaller values of the parameter $\eta$, the RRI of the estimator $\delta_{S1}$ is greater than that of the estimator $\delta_{\psi_{*1}}$. %Notice that, for a narrower range of parameter values where $\delta_{S1}$ estimator outperform $\delta_{\psi_{*1}}$, the risk improvement of $\delta_{S1}$ may also exceed the maximum risk improvement of $\delta_{\psi_{*1}}$.
	
	\item [(iii)] The estimator $\delta_{S1}$ have the high risk improvement when $\eta$ is very small, but its performance decreases sharply when $\eta$ increases. 
	
	\item [(iv)] The estimator $\delta_{\psi_{*1}}$ attains its maximum risk improvement in the range $0.5<\eta<1.5$ (approximately).
	
	\item [(v)] The improvement region of the estimator $\delta_{S1}$ is smaller than $\delta_{\psi_{*1}}$.
	
	\item [(vi)] The RRI of both estimators decreases as the sample size increases or the value of the parameter $\eta$ away from the origin.
\end{enumerate}
The RRI of the proposed estimators under the Linex loss function $L_2(t)$ demonstrate similar behaviour in the simulation results (see Figure \ref{fig:L1} and \ref{fig:L2}).

We now study the risk performance of the estimator $\delta_{RML}$ with respect to the MLE $\delta_{ML}$. For this purpose, we plot the RRI of $\delta_{RML}$ for $\tau$ in Figures \ref{RML1} and \ref{RML2} under the quadratic and Linex loss functions respectively. For Linex loss function $L_2(t)$, we have analyzed the simulation study for different values of the parameter $a_1$, specifically $a_1=-3,-2,2,4$.  We no present the fallowing observations from the Figure \ref{RML1}.
\begin{enumerate}
	\item [(i)] RRI of $\delta_{RML}$ is decreasing function of $\eta$. 
	\item [(ii)] The RRI is highest for very small values of $\eta$ and decreases monotonically as $\eta$ increases. This implies that the advantage of $\delta_{RML}$ becomes negligible for larger values of the parameter $\eta$.
	\item [(iii)] Note that the magnitude of the RRI decreases as the sample size increases. Thus, the estimator $\delta_{RML}$ gives higher risk improvement when sample sizes is very small.
\end{enumerate}
We can observe the similar behaviour in RRI of $\delta_{RML}$ for the Linex loss function $L_2(t)$.
%########################################################################
\begin{figure}[h!]
	\centering
	% -------- Row 1 --------
	\subfigure[$n=4$]{
		\includegraphics[width=0.4\textwidth]{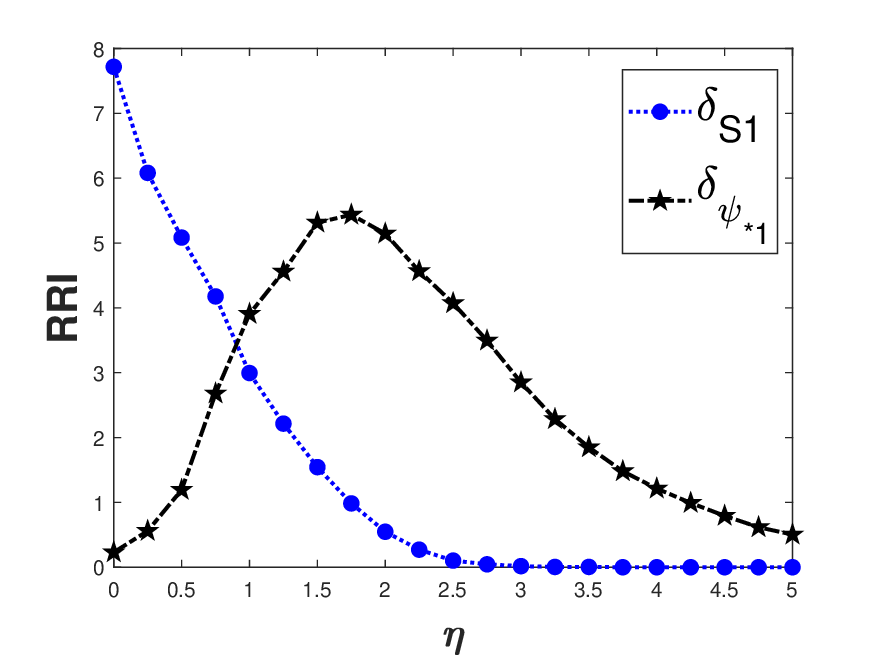}
	}
	\hfil
	\subfigure[$n=8$]{
		\includegraphics[width=0.42\textwidth]{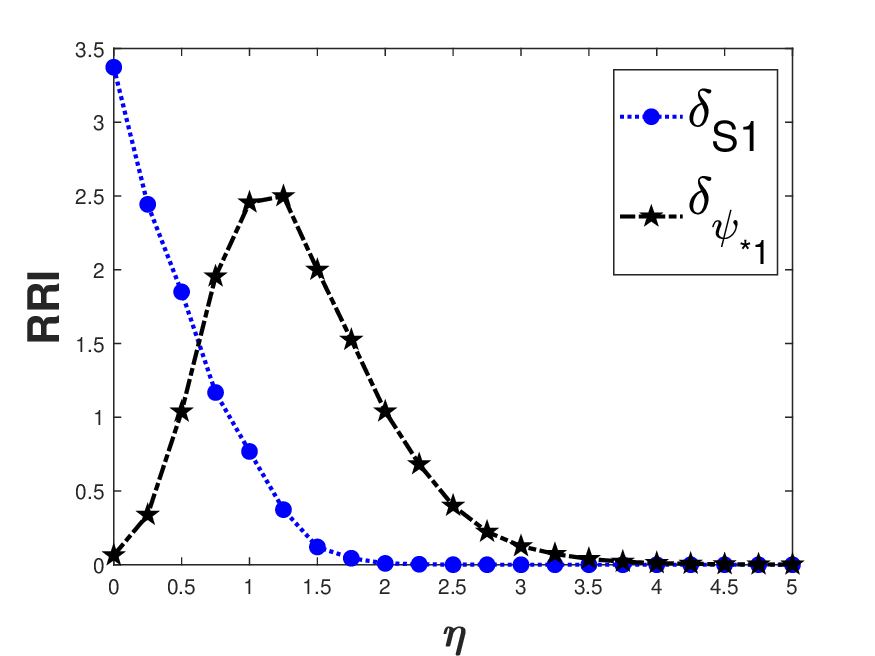}
	}
	\vspace{0.2cm}
	
	%	 -------- Row 2 --------
	\subfigure[$n=10$]{
		\includegraphics[width=0.4\textwidth]{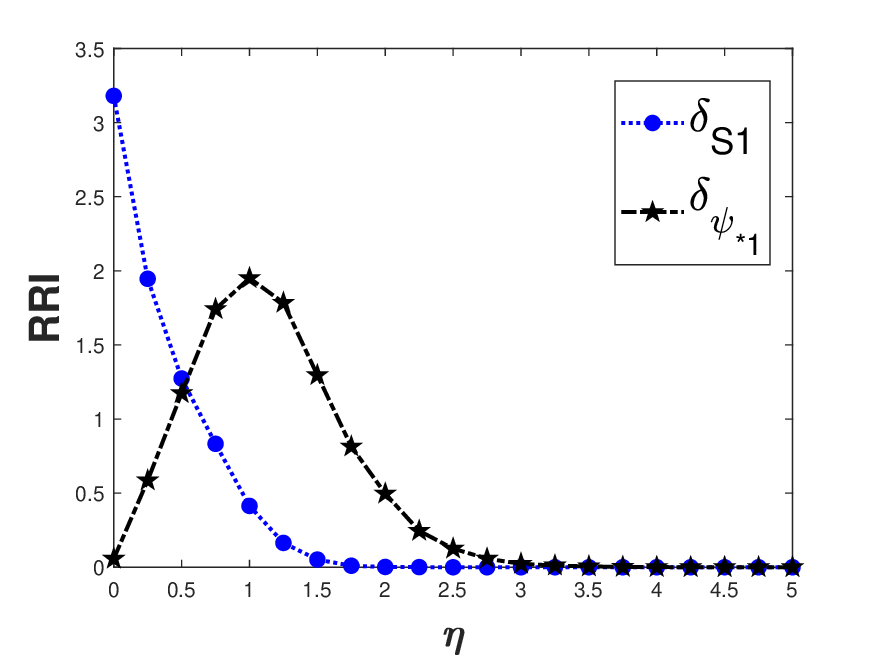}
	}
	\hfil
	\subfigure[$n=15$]{
		\includegraphics[width=0.42\textwidth]{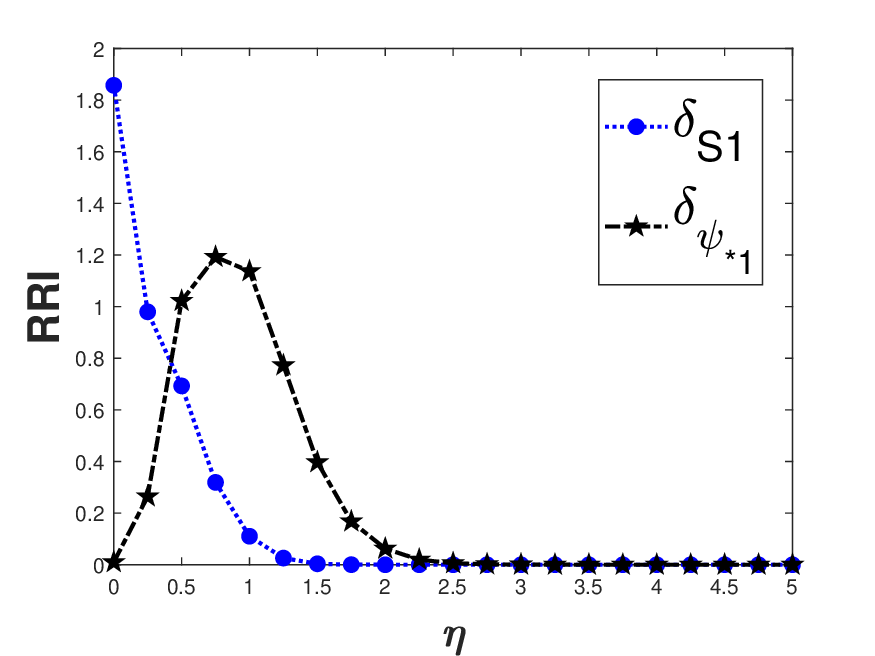}
	}
	\vspace{0.2cm}
	
	% -------- Row 3 --------
	\subfigure[$n=21$]{
		\includegraphics[width=0.42\textwidth]{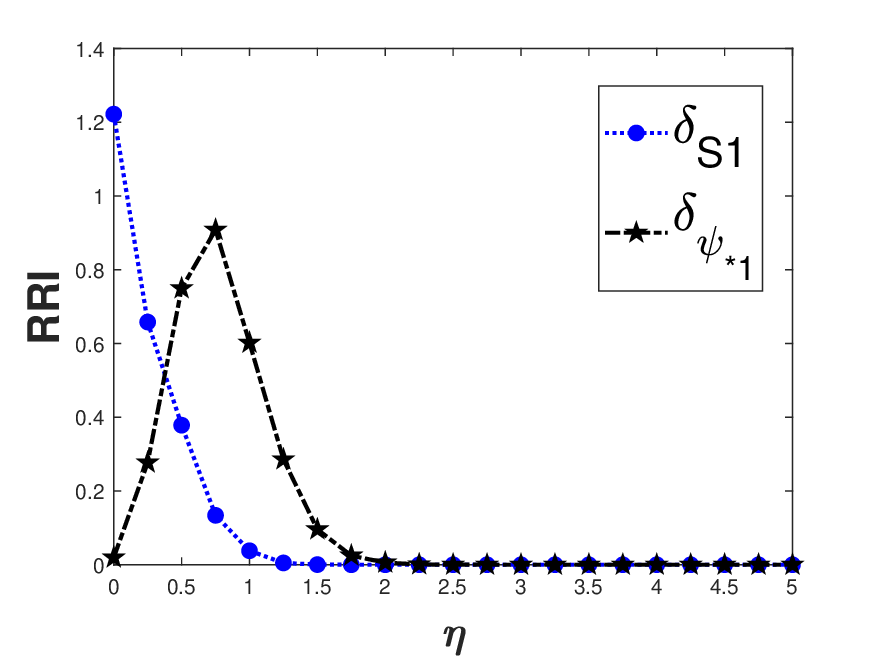}
	}
	\hfil
	\subfigure[$n=26$]{
		\includegraphics[width=0.42\textwidth]{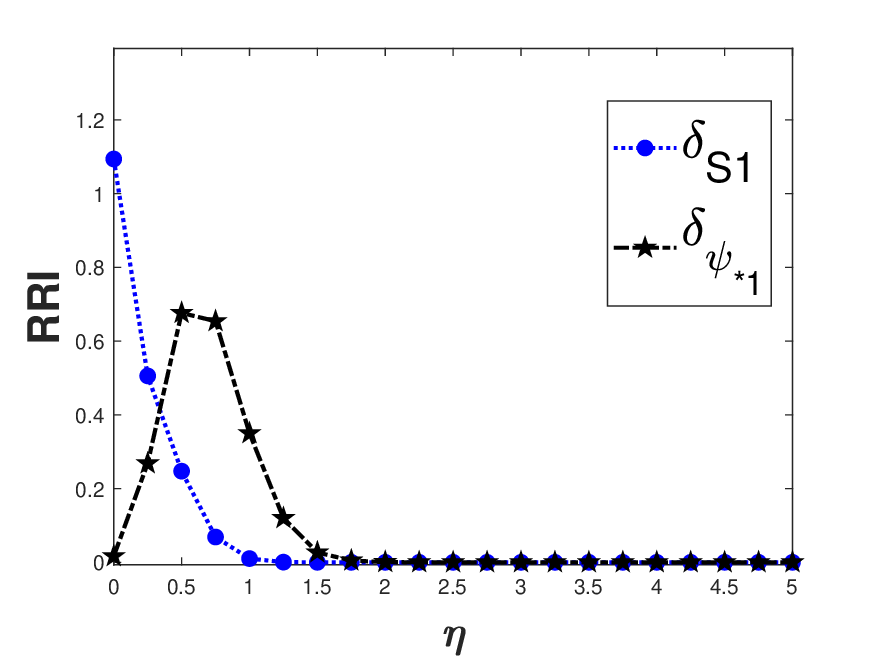}
	}
	\vspace{0.2cm}	
	\caption{RRI of various estimators of $\tau$ with respect to BAEE under $L_1(t)$}
	\label{fig:Q}
\end{figure}
%############################################################################
\begin{figure}[h!]
	\centering
	% -------- Row 1 --------
	\subfigure[$n=4$]{
		\includegraphics[width=0.4\textwidth]{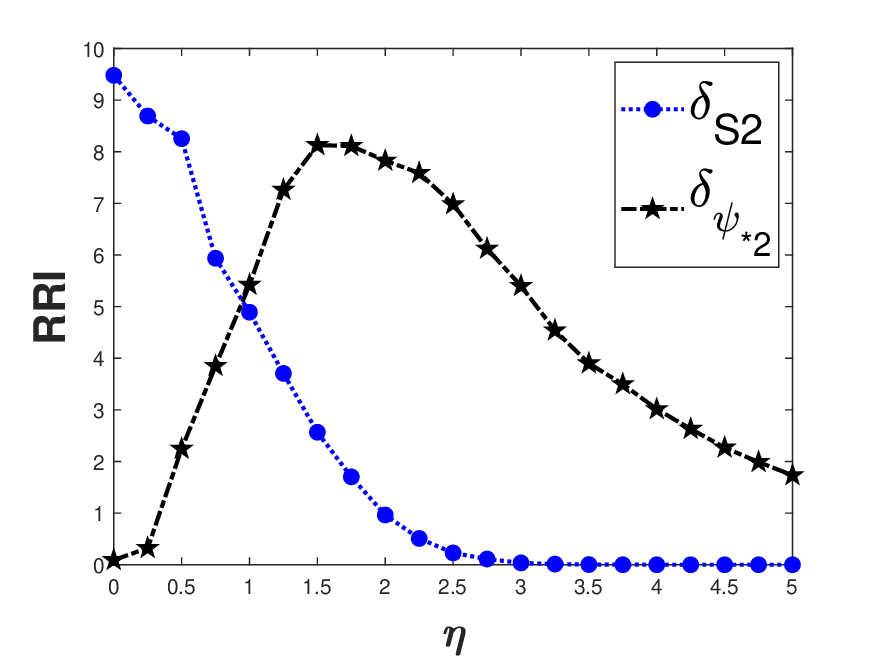}
	}
	\hfil
	\subfigure[$n=8$]{
		\includegraphics[width=0.42\textwidth]{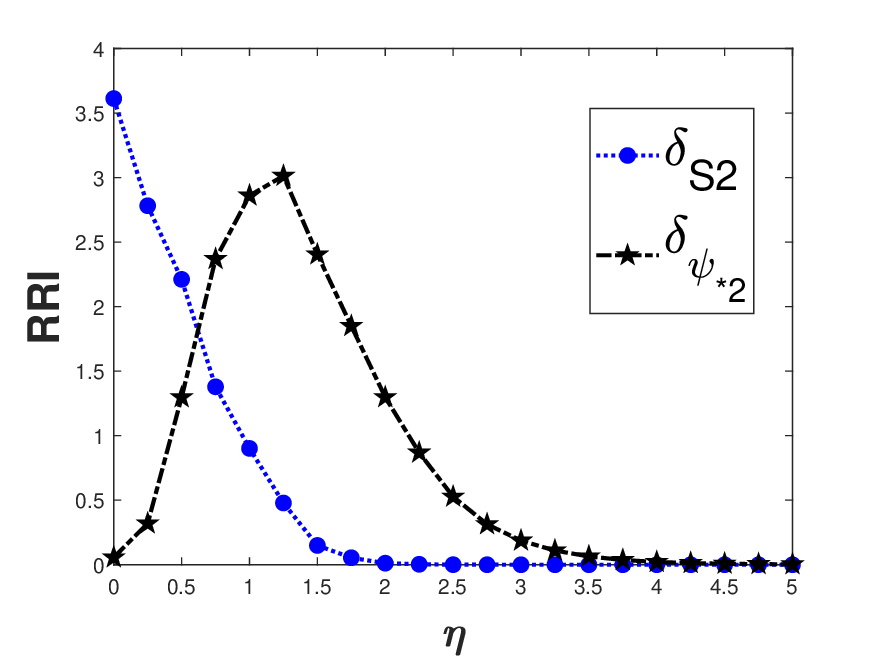}
	}
	\vspace{0.2cm}
	
	% -------- Row 2 --------
	\subfigure[$n=10$]{
		\includegraphics[width=0.42\textwidth]{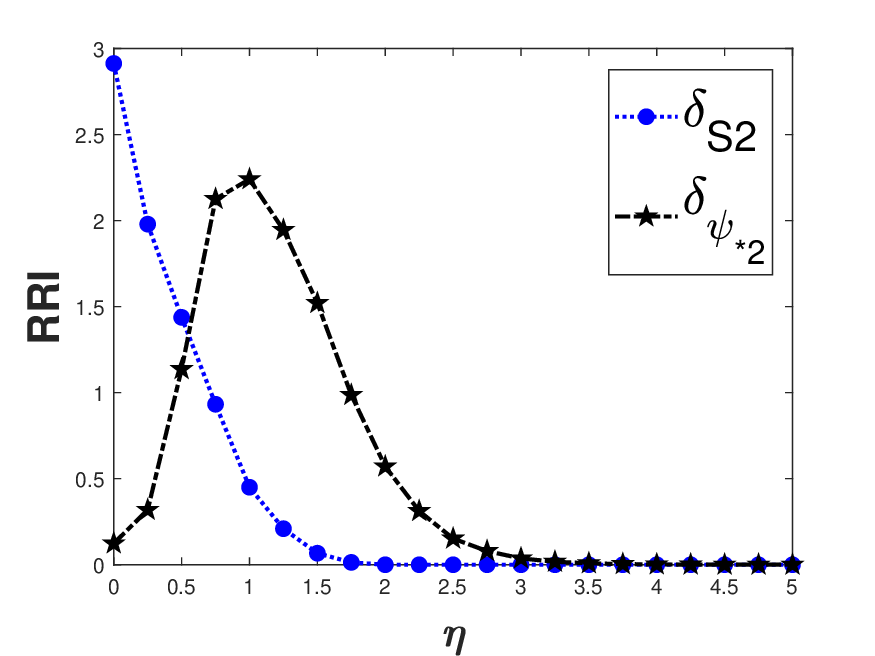}
	}
	\hfil
	\subfigure[$n=15$]{
		\includegraphics[width=0.42\textwidth]{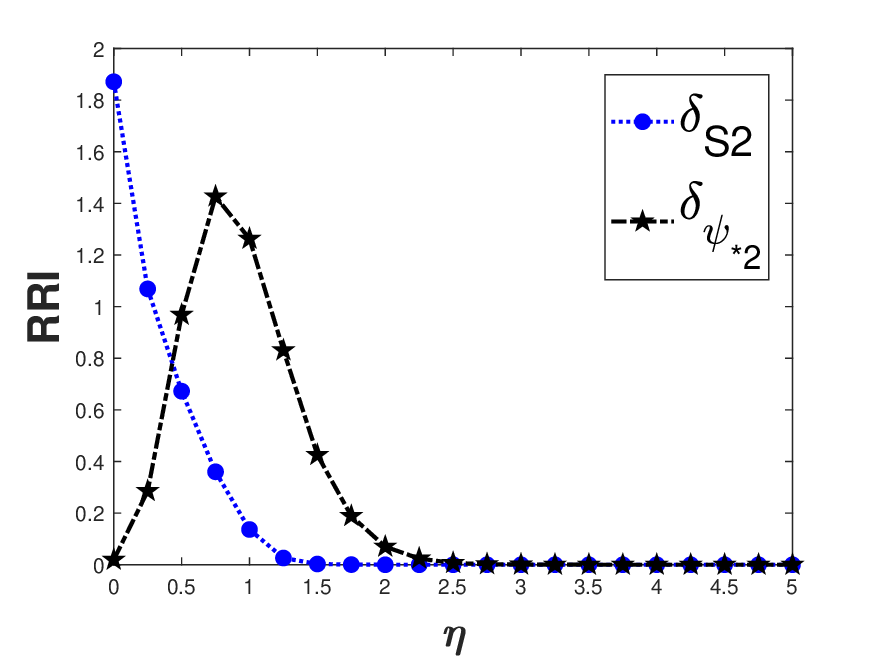}
	}
	\vspace{0.2cm}
	
	% -------- Row 3 --------
	\subfigure[$n=21$]{
		\includegraphics[width=0.42\textwidth]{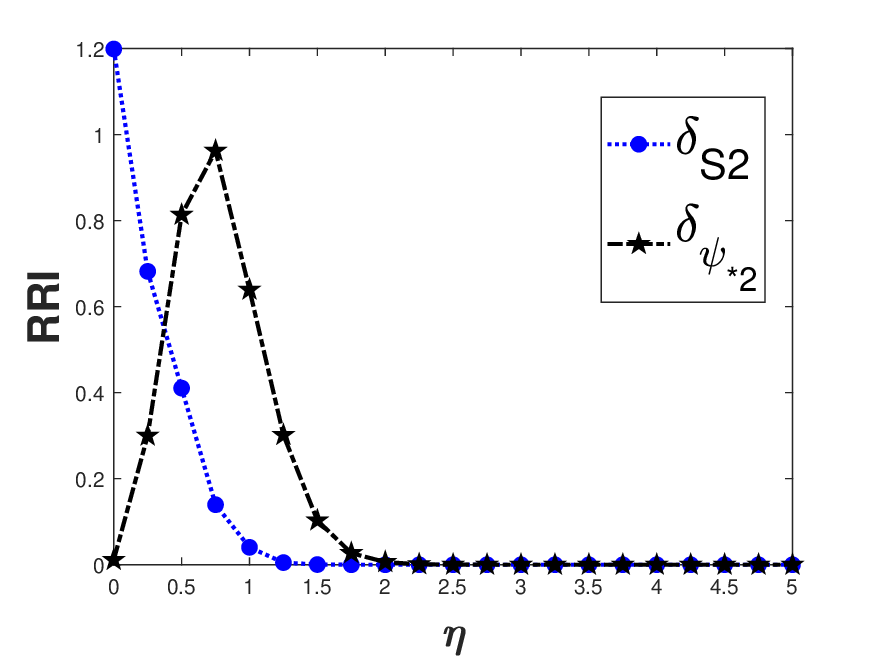}
	}
	\hfil
	\subfigure[$n=26$]{
		\includegraphics[width=0.42\textwidth]{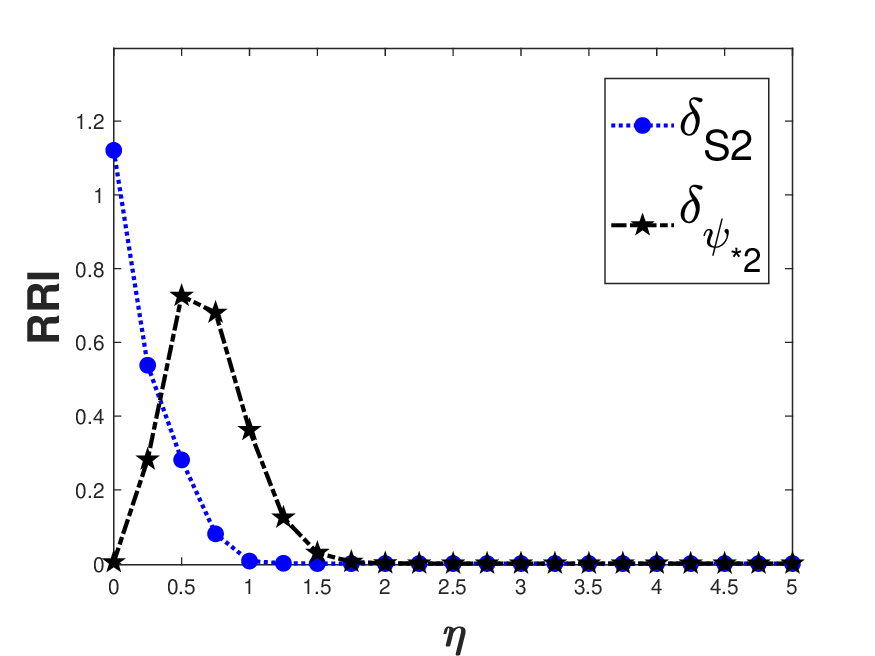}
	}
	\vspace{0.2cm}
	\caption{RRI of various estimators $\tau$ with respect to BAEE under $L_2(t)$ for $a_1=-3$}
	\label{fig:L1}
\end{figure}
%%###########################################################################
\begin{figure}[h!]
	\centering
	% -------- Row 1 --------
	\subfigure[$n=4$]{\includegraphics[width=0.42\textwidth]{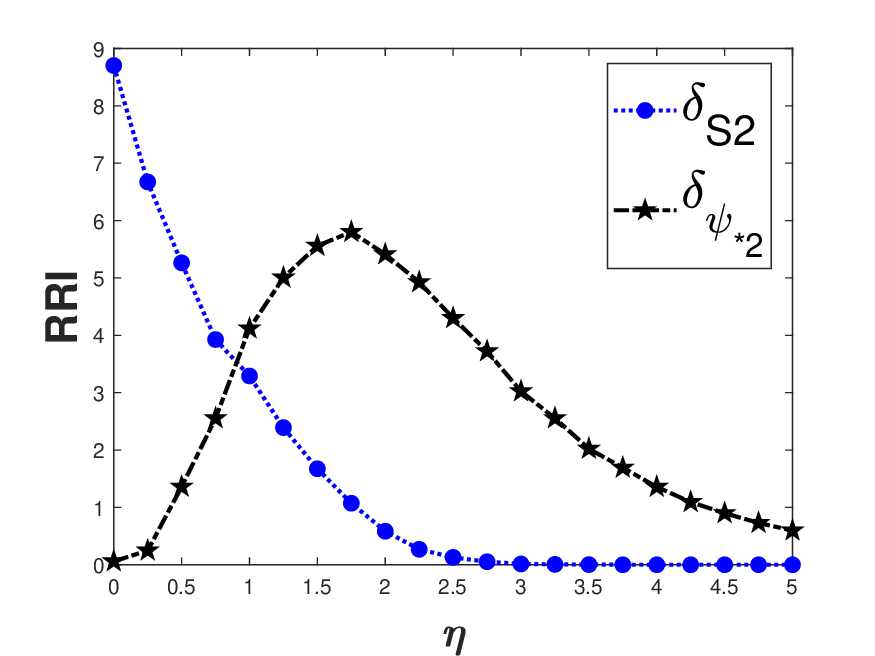}}
	\hfil
	\subfigure[$n=8$]{\includegraphics[width=0.42\textwidth]{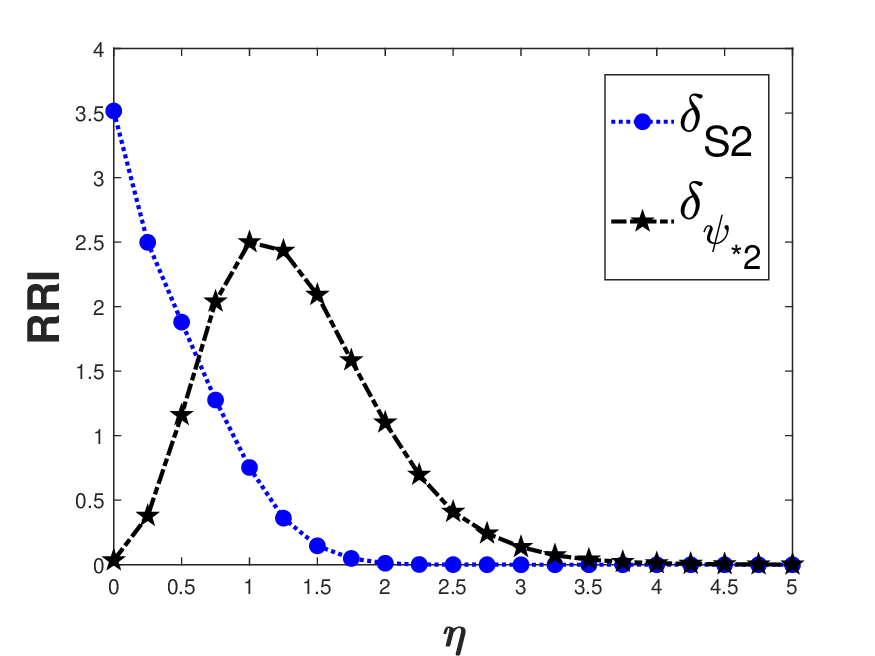}}
	\vspace{0.2cm}
	% -------- Row 2 --------
	\subfigure[$n=10$]{\includegraphics[width=0.42\textwidth]{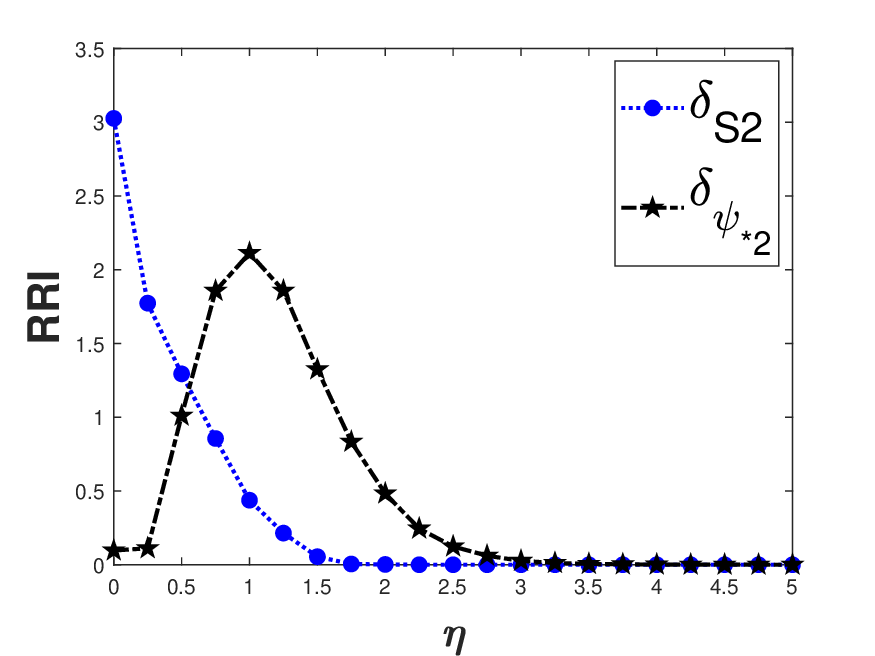}}
	\hfil
	\subfigure[$n=15$]{\includegraphics[width=0.42\textwidth]{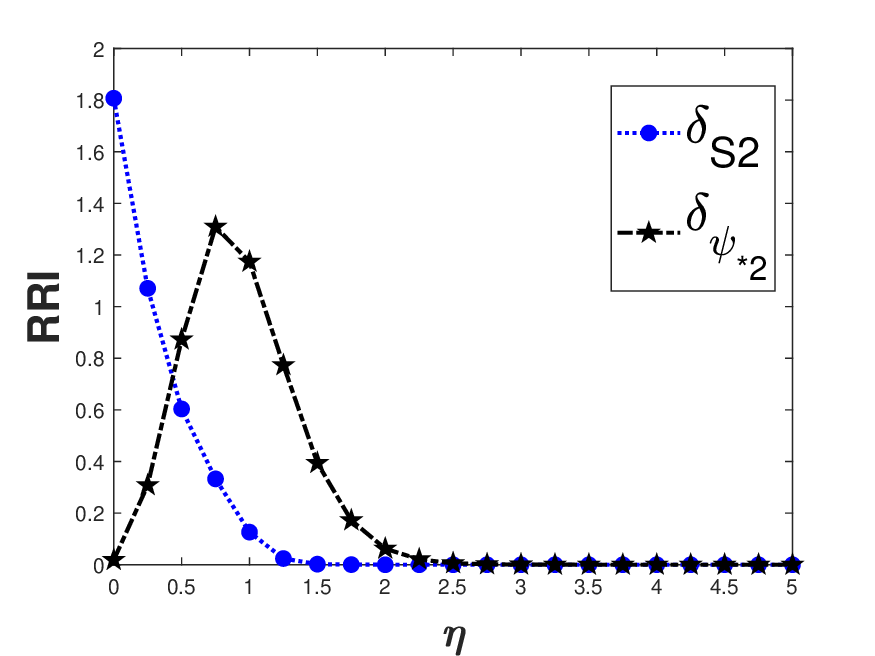}}
	\vspace{0.2cm}
	% -------- Row 3 --------
	\subfigure[$n=21$]{\includegraphics[width=0.42\textwidth]{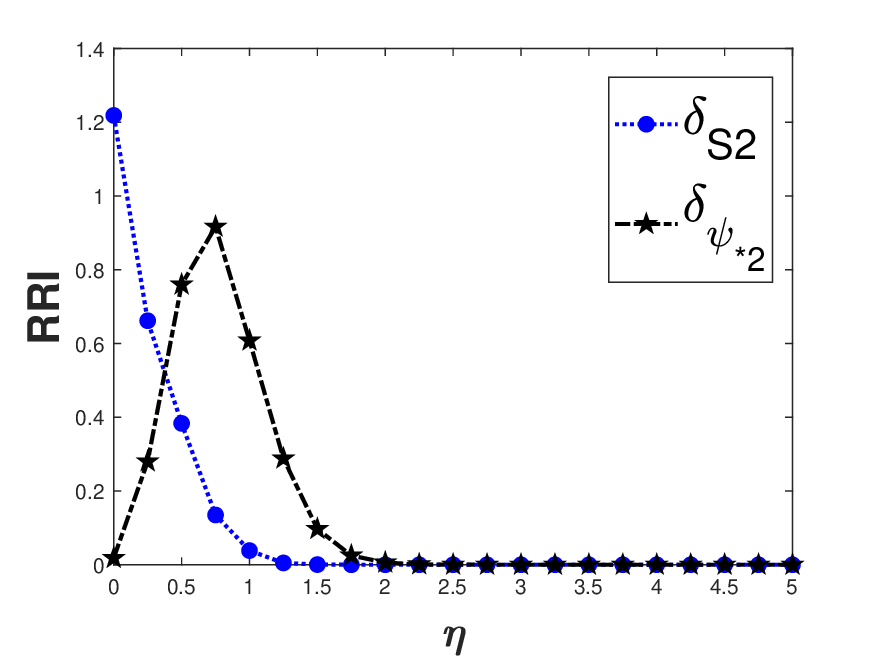}}
	\hfil
	\subfigure[$n=26$]{\includegraphics[width=0.42\textwidth]{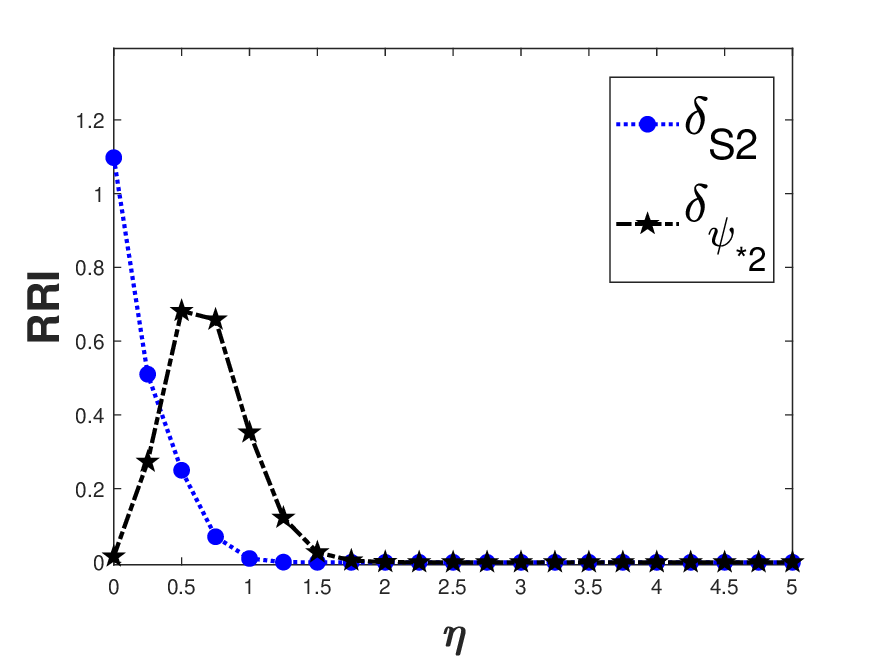}}
	\vspace{0.2cm}
	\caption{RRI of various estimators $\tau$ with respect to BAEE under $L_2(t)$ for $a_1=-\frac{1}{2}$}
	\label{fig:L2}
\end{figure}
%%###########################################################################
%%###########################################################################
\begin{figure}[h!]
	\centering
	% -------- Row 1 --------
	\subfigure[$n=4$]{\includegraphics[width=0.42\textwidth]{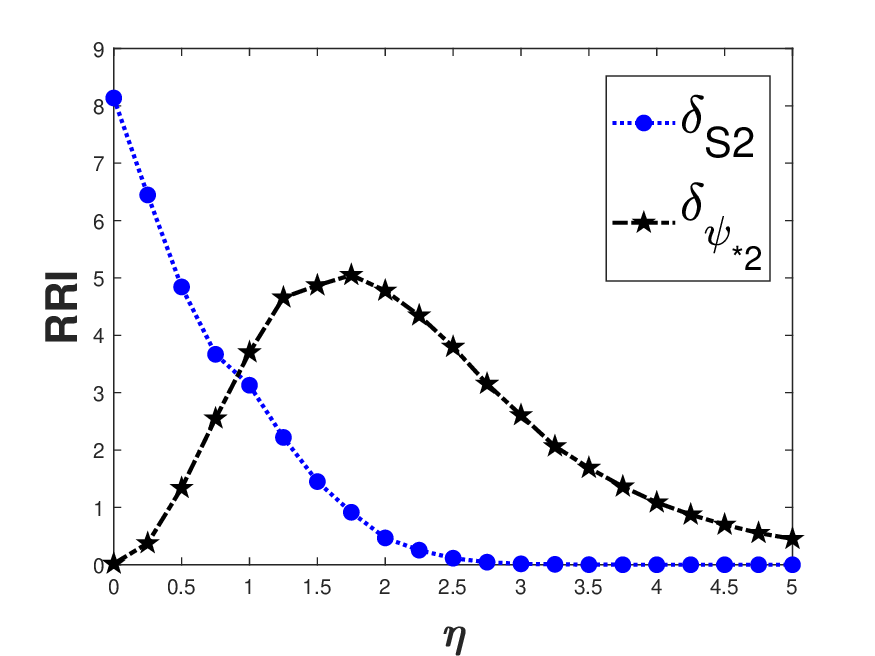}}
	\hfil
	\subfigure[$n=8$]{\includegraphics[width=0.42\textwidth]{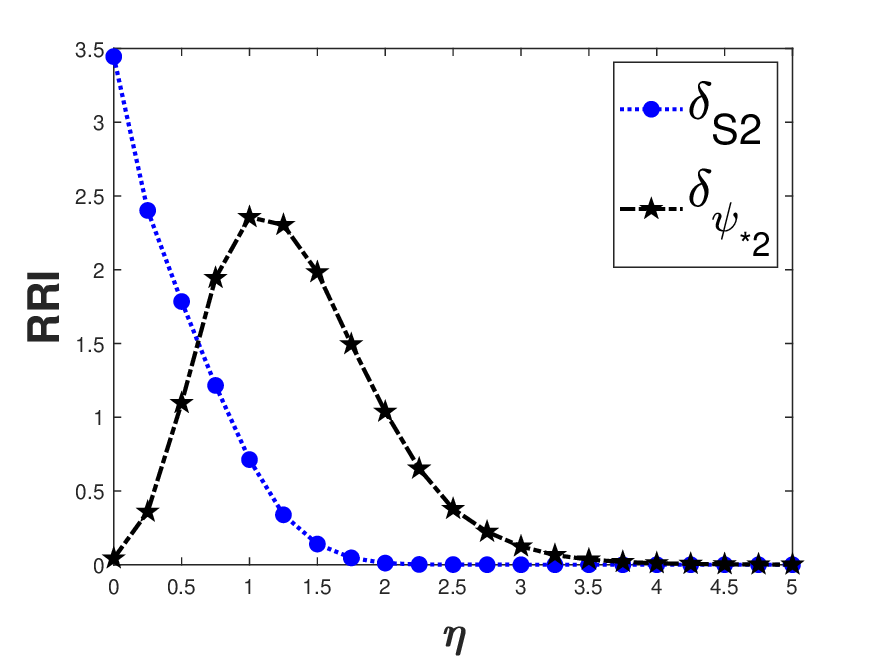}}
	\vspace{0.2cm}
	% -------- Row 2 --------
	\subfigure[$n=10$]{\includegraphics[width=0.42\textwidth]{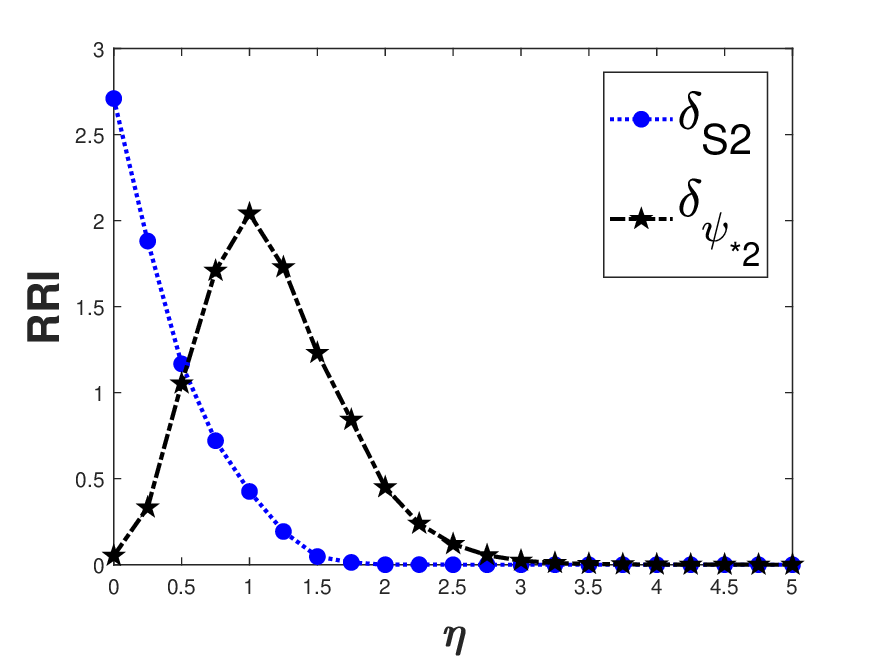}}
	\hfil
	\subfigure[$n=15$]{\includegraphics[width=0.42\textwidth]{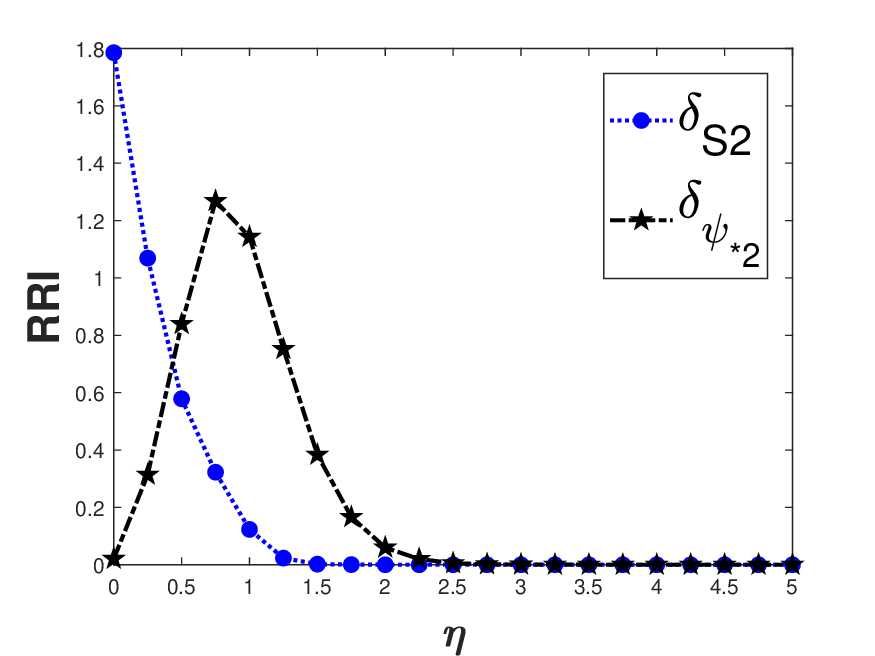}}
	\vspace{0.2cm}
	% -------- Row 3 --------
	\subfigure[$n=21$]{\includegraphics[width=0.42\textwidth]{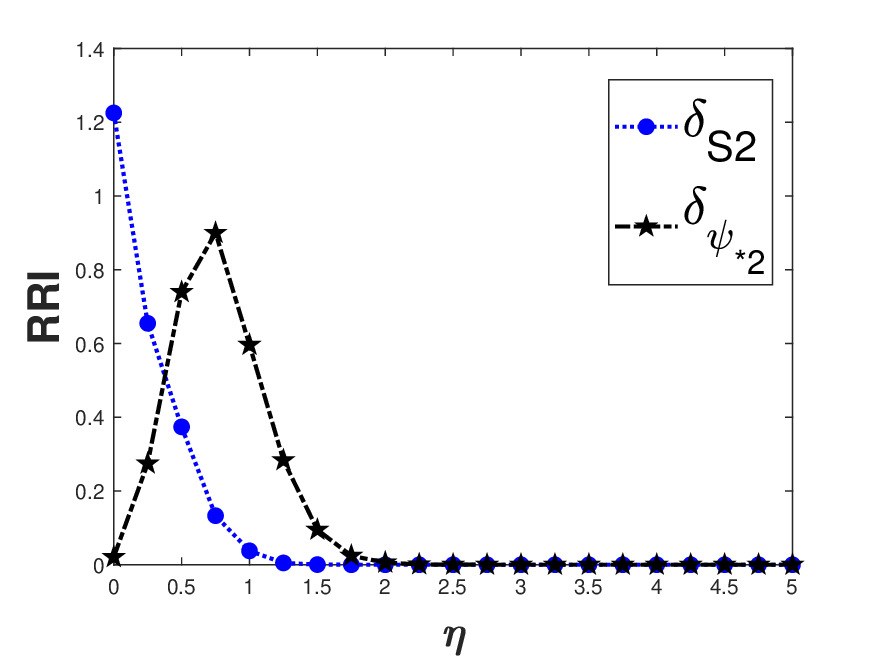}}
	\hfil
	\subfigure[$n=26$]{	\includegraphics[width=0.42\textwidth]{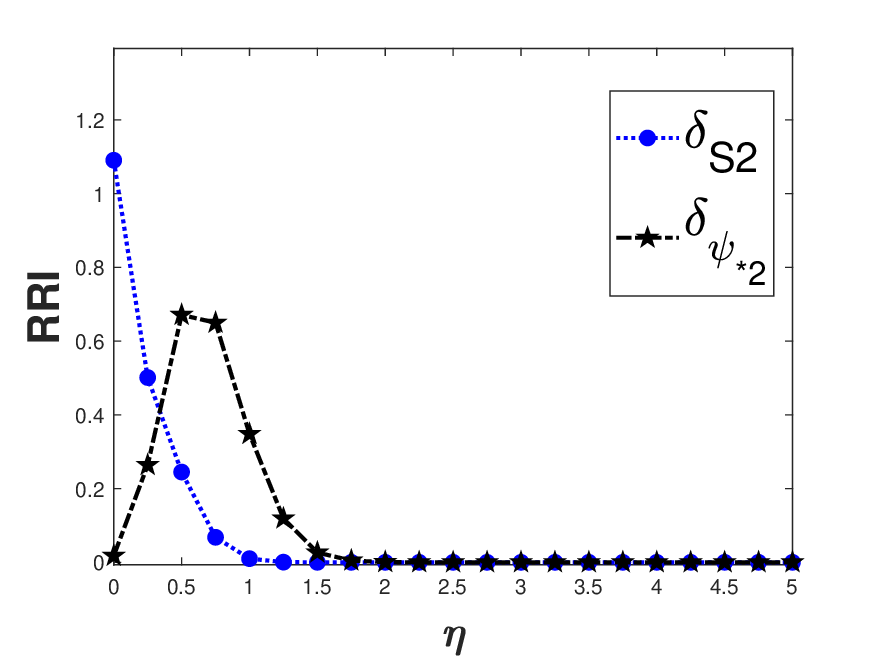}}
	\vspace{0.2cm}
	\caption{RRI of various estimators $\tau$ with respect to BAEE under $L_2(t)$ for $a_1=\frac{1}{2}$}
	\label{fig:L3}
\end{figure}
%%###########################################################################
\begin{figure}[h!]
	\centering
	% -------- Row 1 --------
	\subfigure[$n=4$]{\includegraphics[width=0.42\textwidth]{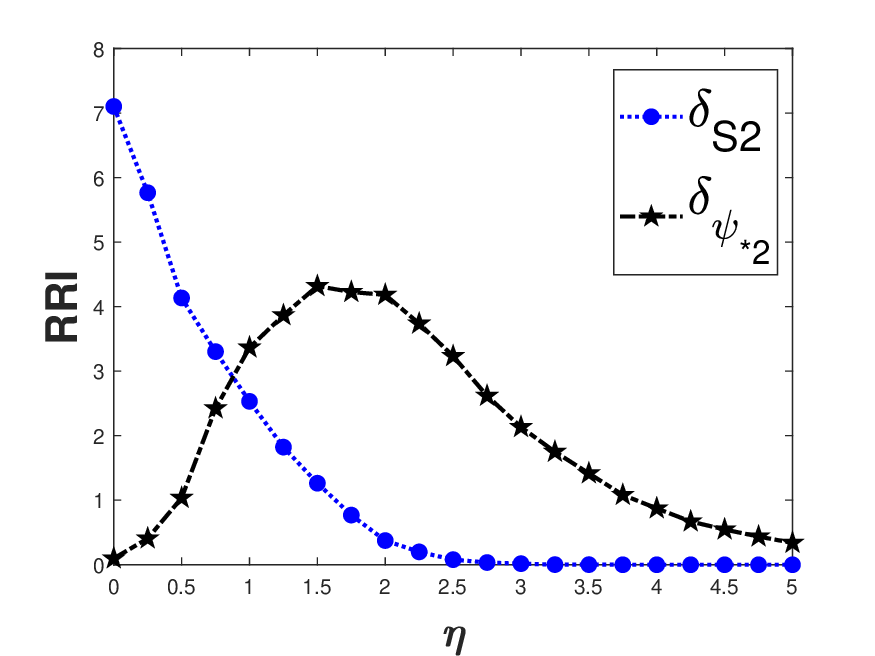}}
	\hfil
	\subfigure[$n=8$]{\includegraphics[width=0.42\textwidth]{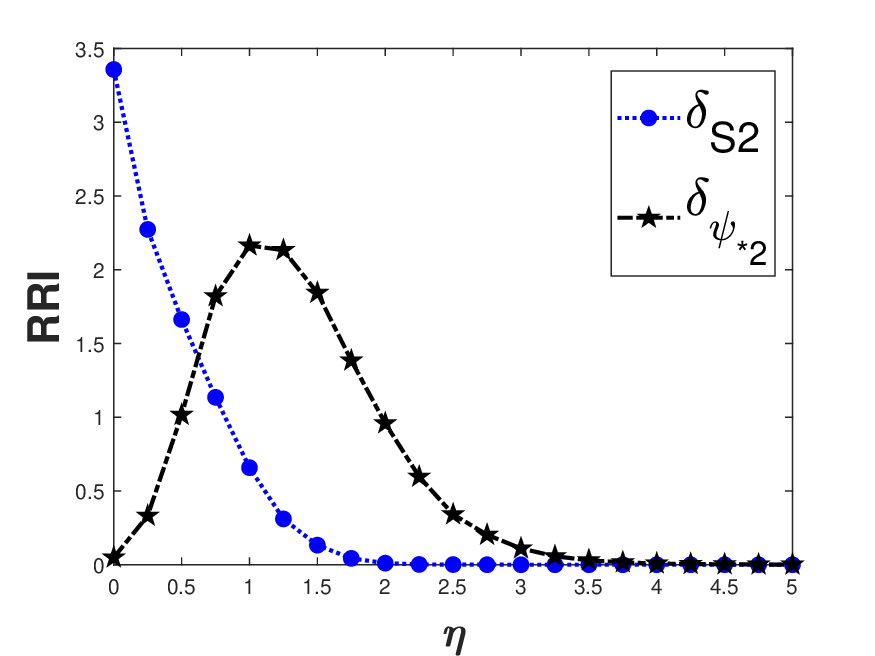}}
	\vspace{0.2cm}
	% -------- Row 2 --------
	\subfigure[$n=10$]{\includegraphics[width=0.42\textwidth]{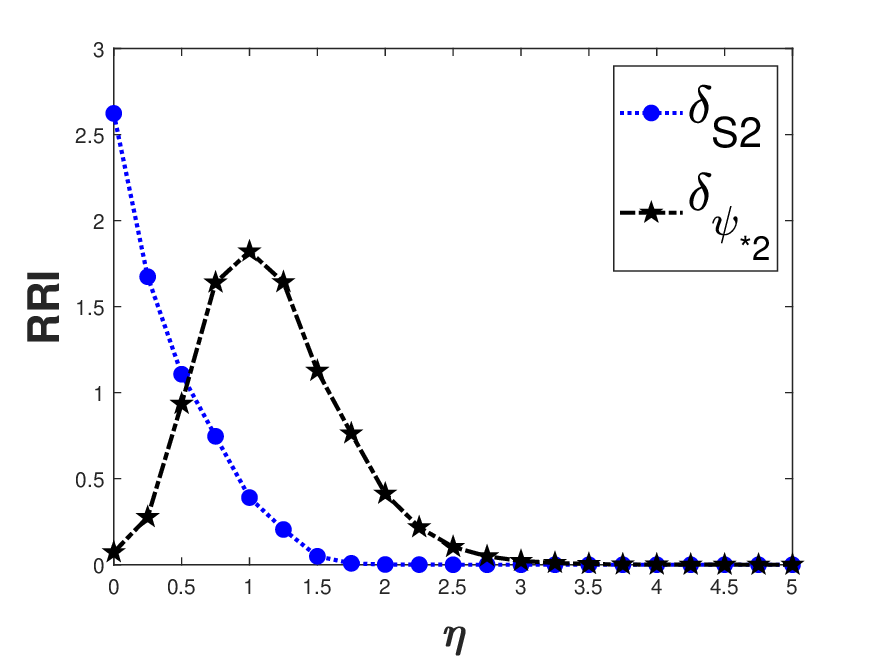}}
	\hfil
	\subfigure[$n=15$]{\includegraphics[width=0.42\textwidth]{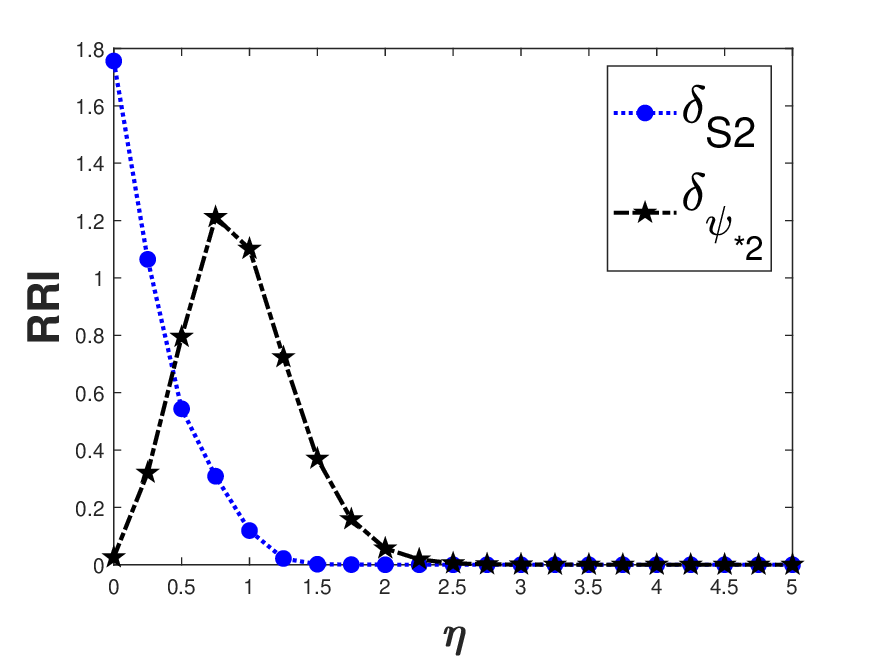}}
	\vspace{0.2cm}
	% -------- Row 3 --------
	\subfigure[$n=21$]{\includegraphics[width=0.42\textwidth]{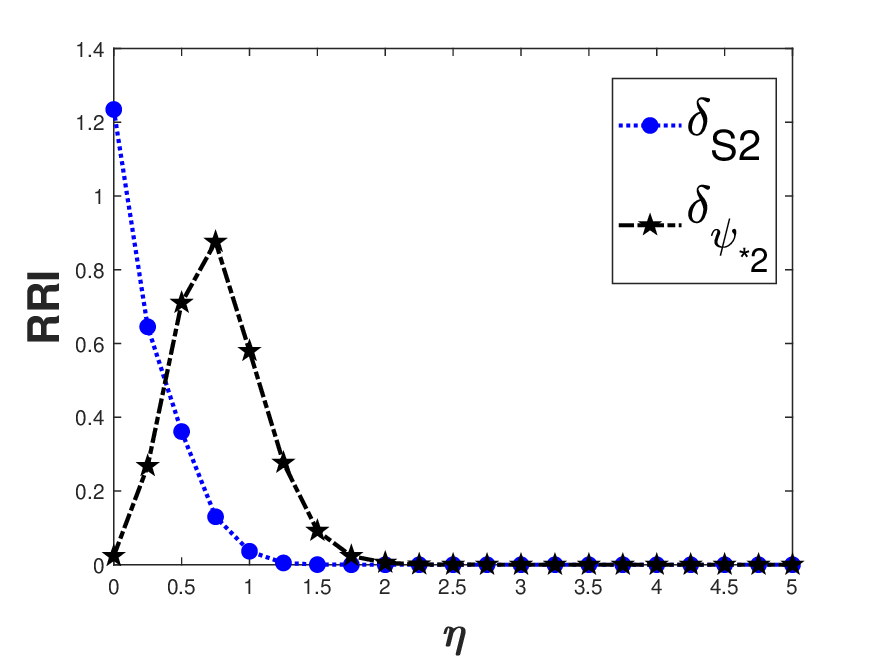}}
	\hfil
	\subfigure[$n=26$]{	\includegraphics[width=0.42\textwidth]{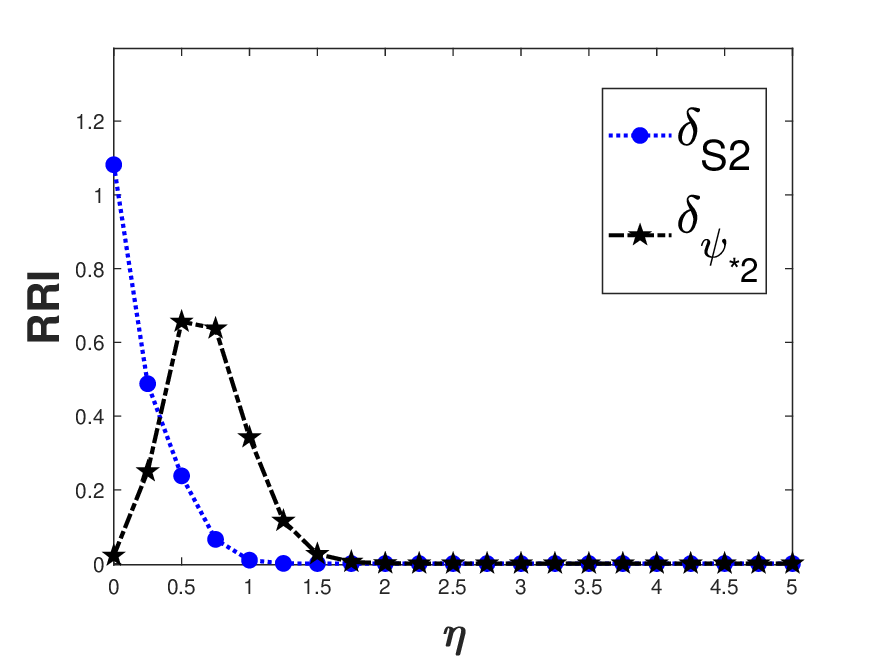}}
	\vspace{0.2cm}
	\caption{RRI of various estimators $\tau$ with respect to BAEE under $L_2(t)$ for $a_1=2$}
	\label{fig:L4}
\end{figure}
%%########################################################################################
\begin{figure}[htbp]
	\centering
	\includegraphics[width=0.5\textwidth]{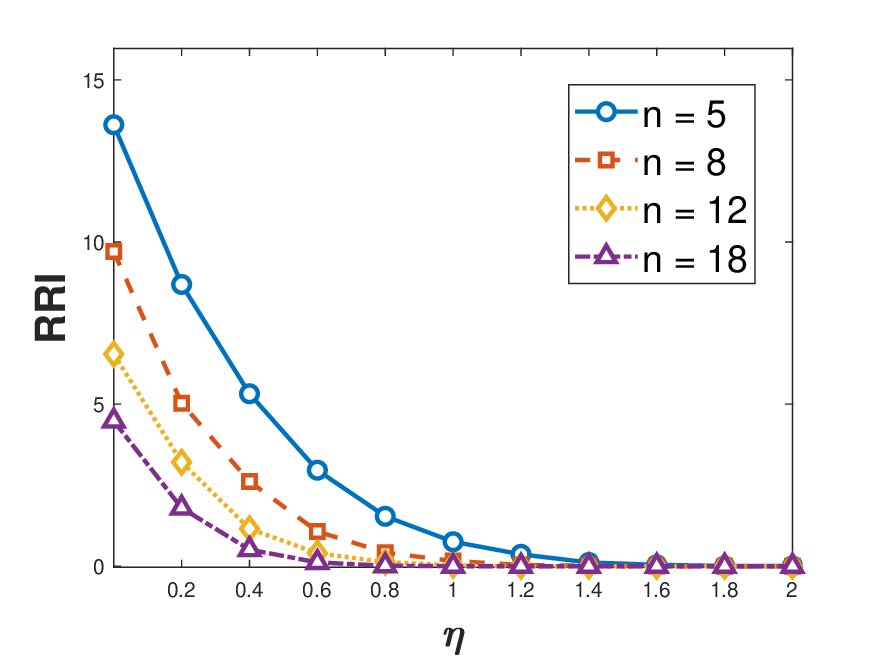}
	\caption{RRI of RMLE w.r.t MLE for sample sizes 
		$n=5,8,12,$ and $18$ under $L_1(t)$}
	\label{RML1}
\end{figure}
%%########################################################################################
\begin{figure}[htbp]
	\centering
	% -------- Row 1 --------
	\subfigure[$a_1=-3$]{
		\includegraphics[width=0.37\textwidth]{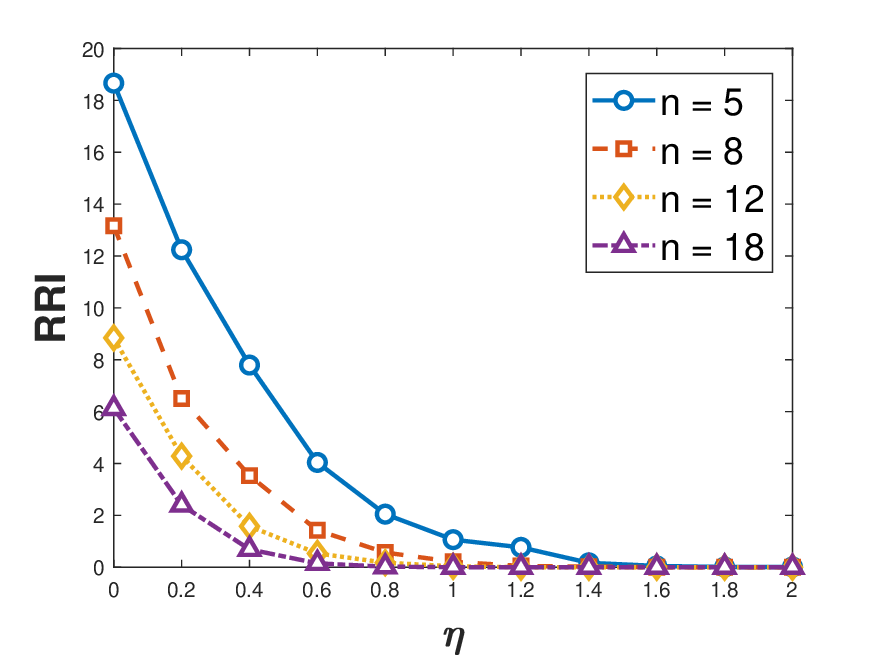}
	}
	\hfil
	\subfigure[$a_1=-2$]{
		\includegraphics[width=0.37\textwidth]{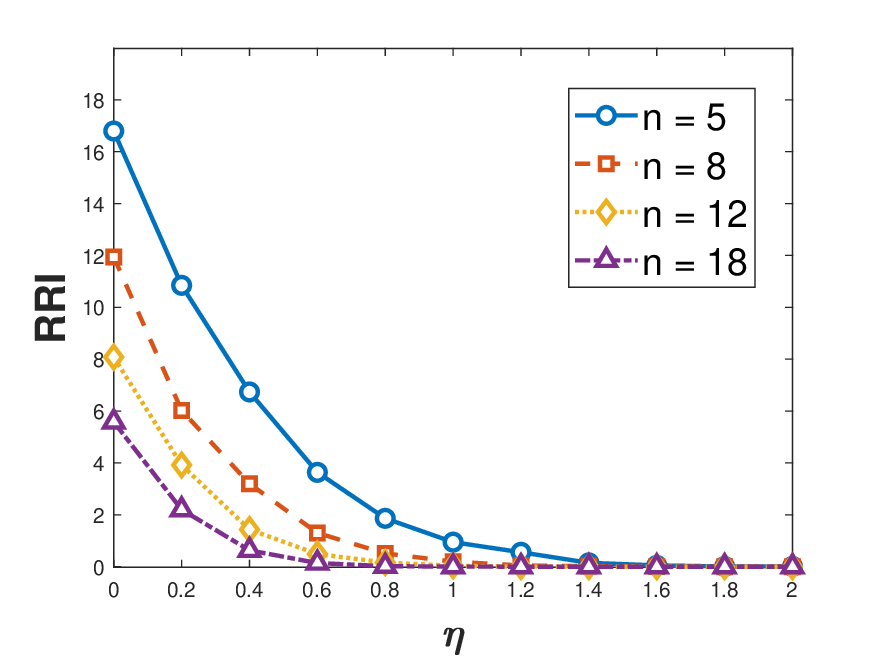}
	}
	\vspace{0.2cm}
	
	% -------- Row 2 --------
	\subfigure[$a_1=2$]{
		\includegraphics[width=0.37\textwidth]{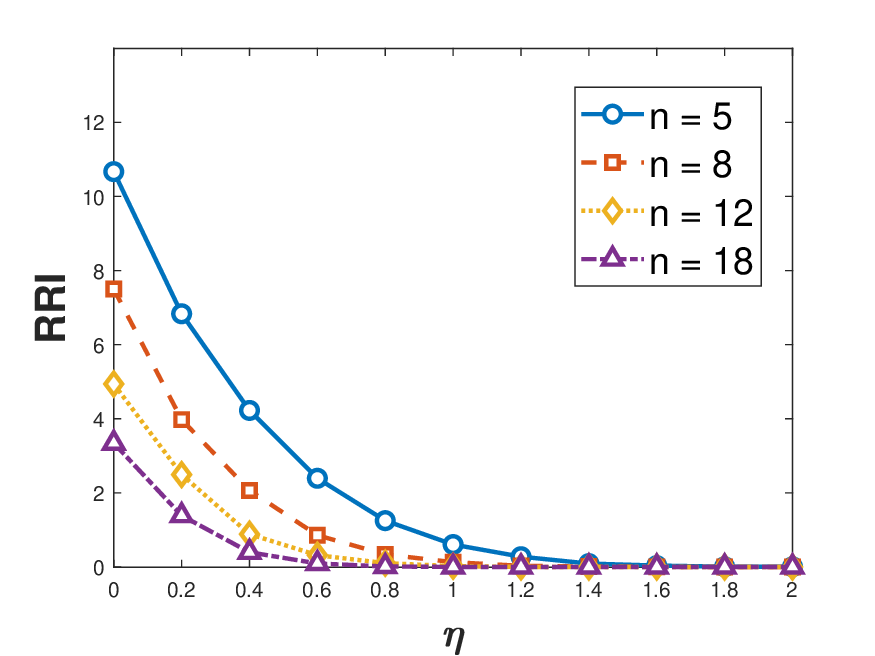}
	}
	\hfil
	\subfigure[$a_1=4$]{
		\includegraphics[width=0.37\textwidth]{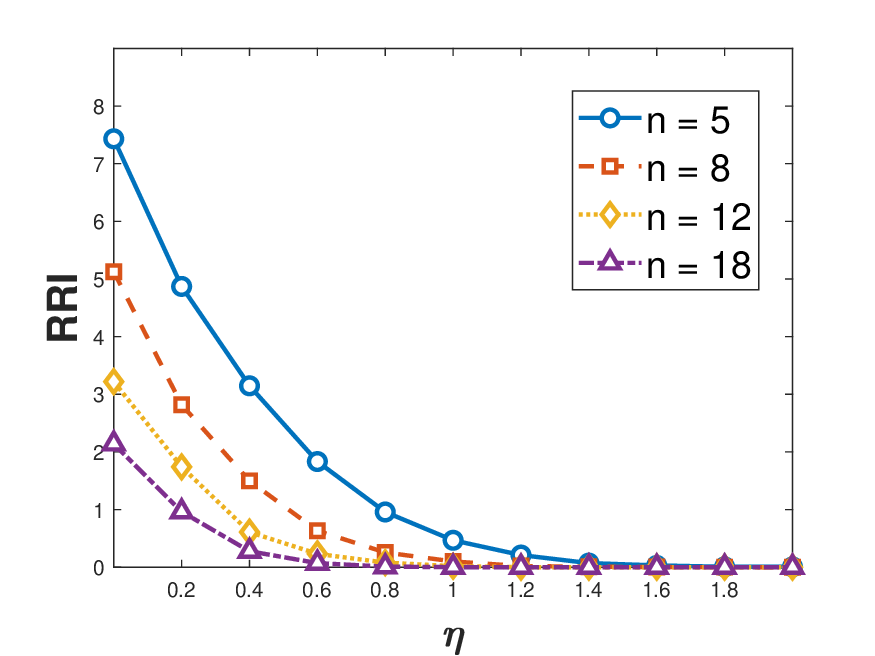}
	}
	\vspace{0.2cm}
	\caption{RRI of RMLE w.r.t MLE for sample sizes $n=5,8,12,$ and $18$ under $L_2(t)$ for various values of $a_1$}
	\label{RML2}
\end{figure}

\section{Performance comparison of confidence intervals through simulation}\label{sec7}
In this section, we will discuss the numerical performance of the derived confidence intervals of the parameter $\tau$. We use Monte Carlo simulation methods to compute the coverage probability (CP) and average lengths (AL). To obtain the CPs and ALs we have generated $30,000$ random samples from two normal population of sizes $n$ with equal variances $\sigma^2$ and means $\mu_1$ and $\mu_2$ respectively. For this study, without loss of generality, we consider $\mu_1=\mu_2=0$ and $\sigma=1$. In this study we have computed CP and AL for various combination of samples. For the computation of bootstrap confidence intervals, the inner loop $(K)$ corresponding to the replications, is carried out up to $3,000$ iteration. For the generalized confidence intervals, the inner loop consists of $10,000$ replications. Similarly, for obtaining HPD credible intervals based on the MCMC algorithm, we consider $10,000$ replications in the inner loop. Moreover, all the confidence intervals are computed at the significance confidence level $(1-\alpha)=0.95$. It can be observed that the average lengths of the bootstrap-p and bootstrap-t confidence intervals coincide, whereas the bootstrap-t method gives greater coverage probability than the bootstrap-p (see Figure \ref{bp} and \ref{bt}). When we examine the performance of the confidence intervals based on CPs and ALs, it is observed that some intervals achieves the significance confidence levels $(0.95)$, while their ALs are quite large. In the other cases, some confidence intervals show their ALs are quit acceptable but their CPs are bellow the significance level. We will use the following methods to dell with this problem. To begin with, we may fix a minimum acceptable confidence level for CP and consider only those intervals which achieves this level, after that we select the interval which have the smallest AL. In the following approach we choose the intervals that attains at least the required minimum significance level and subsequently assess them using a common performance measure, namely the PCD, defined as the ratio of coverage probability (CP) to average length (AL). This unified metric was originally suggested by \cite{unhapipat2016small}. This unified criterion of performance measure will take care of both quantities AL and CP simultaneously. This measure identify the intervals that have smaller ALs along with comparatively high CPs. It is clear that a confidence interval having a higher PCD values demonstrate the most favorable performance. We have carried out simulations across a range of different sample sizes. We conclude the following observation from the Figure \ref{p1}, \ref{bp}, \ref{bt}, \ref{p4} and \ref{p6}. From our simulation results, we can draw the following observation. 
\begin{itemize}
	\item [(i)] All the ALs are decreasing function of $n$. The CPs are not monotonic in $n$. The values of PCD are increasing when $n$ is increasing.
	
	\item [(ii)] We have set a significance level at $0.95$, then bootstrap-t and the generalized confidence intervals has attain it.
	
	\item [(iii)] The range of the coverage probability from $65\%$ to $95\%$. In particular, the CP of the bootstrap-t and generalized confidence intervals remain close to 0.95.
	
	\item [(iv)] If we fixed the confidence level at $95\%$, the bootstrap-t and generalized confidence interval achieve this. If we adopt a more liberal confidence level of $90\%$, then all confidence intervals approximately achieve the nominal level provided that $n>20$ (approximately).
	
	\item [(v)] Ordering the confidence intervals from shortest to longest in terms of their shortest ALs, we obtain the following order: asymptotic, HPD, bootstrap-t (or bootstrap-p) and generalized confidence intervals.
	
	\item [(vi)] The intervals can be ordered from lowest to highest based on their coverage probability, the following order is obtained: generalized confidence intervals, bootstrap-t, HPD, Asymptotic confidence interval and bootstrap-p confidence intervals.
	
	\item [(vii)] If we consider both criteria minimum AL and maximum CP simultaneously, it is not possible to obtain a clear ordering of the intervals. Under these circumstances, the PCD criterion is adopted to determine the most appropriate interval. 
	
	\item [(viii)] For the nominal level of 0.95, the confidence intervals can be ranked based on their PCD values in the following order: Asymptotic CI, HPD, Bootstrap-t CI, Generalized CI, Bootstrap-p CI.
\end{itemize}
%######################################################################
\begin{figure}[!t]
	\centering
	\includegraphics[width=0.8\textwidth]{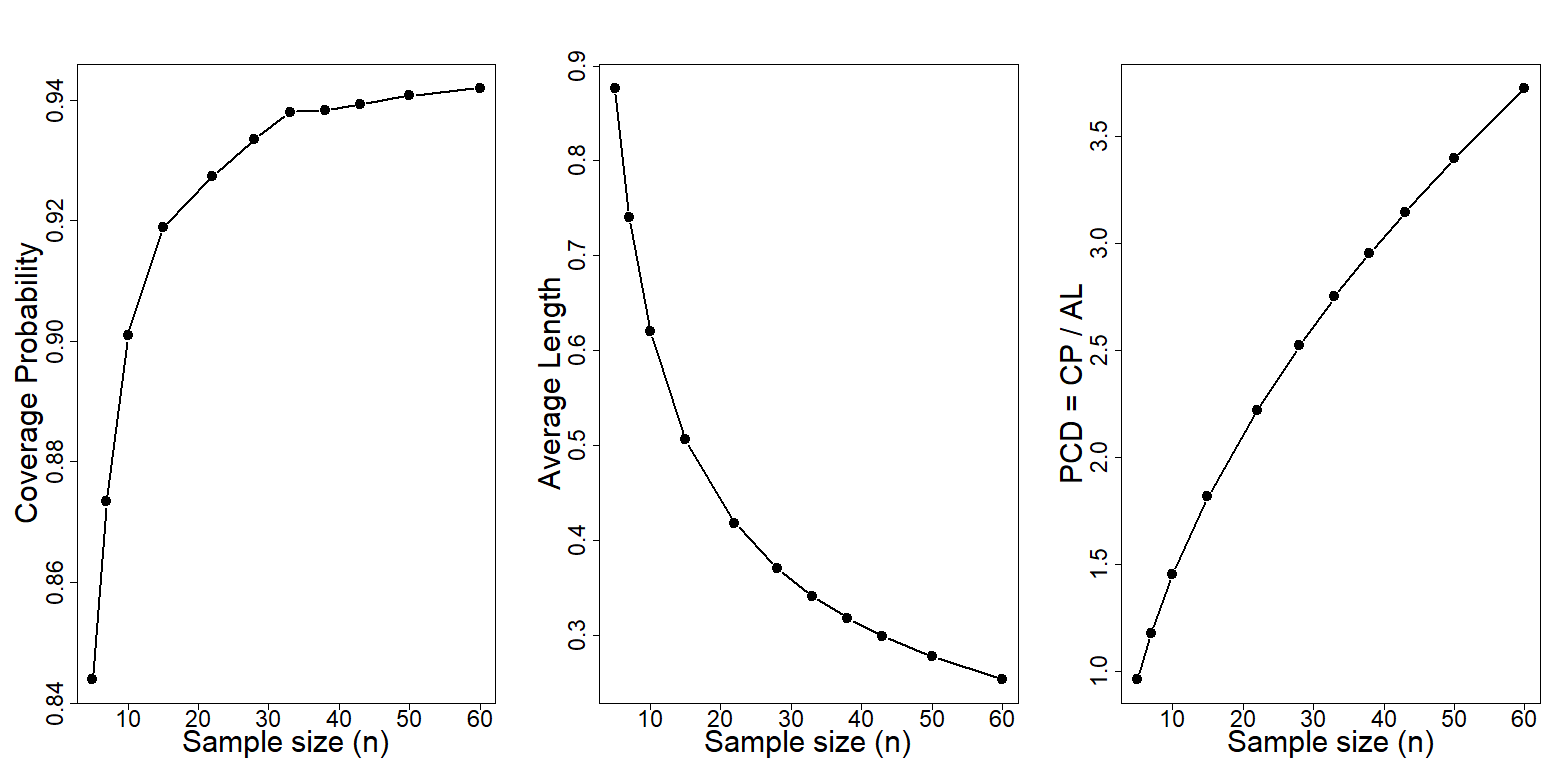}
	\caption{Asymptotic confidence interval}
	\label{p1}
\end{figure}
\vspace{1cm}
\begin{figure}[!t]
	\centering
	\includegraphics[width=0.8\textwidth]{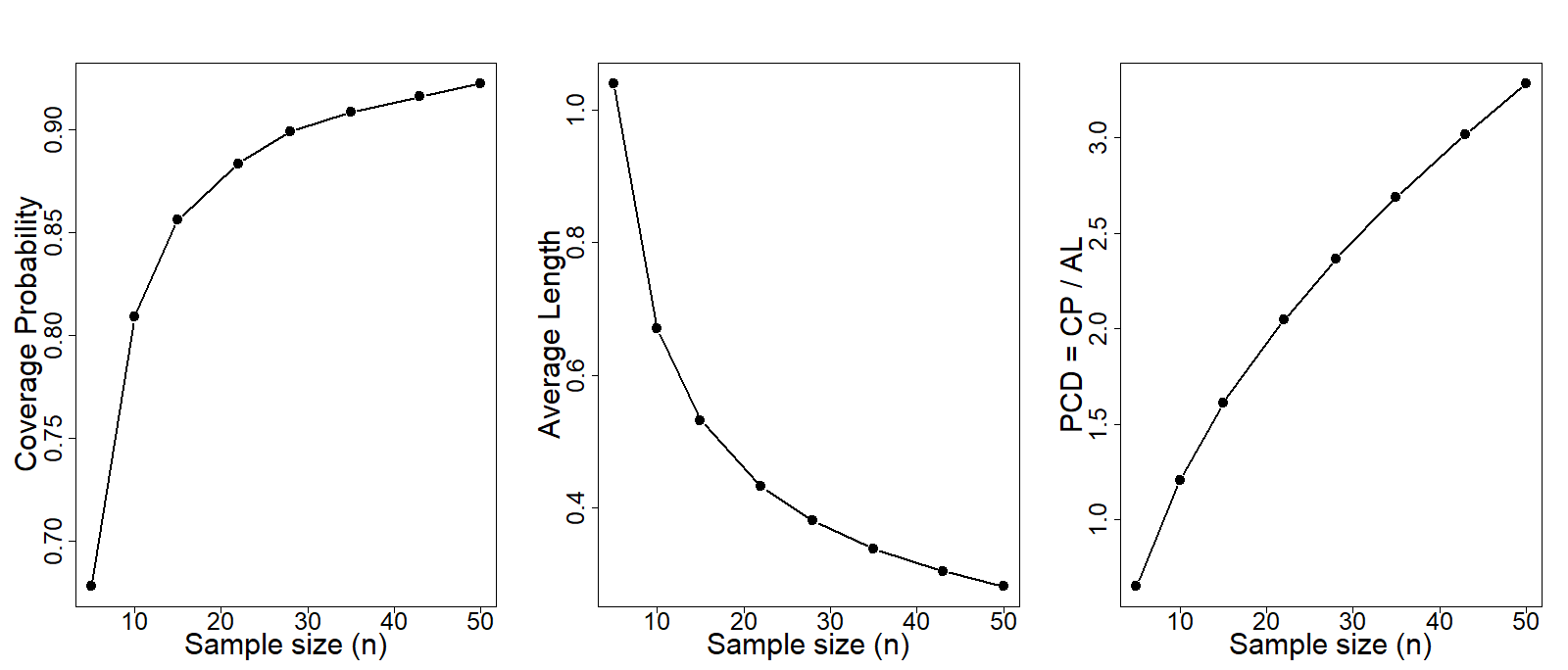}
	\caption{Bootstrap-p confidence interval}
	\label{bp}
\end{figure}
\vspace{1cm}
\begin{figure}[!t]
	\centering
	\includegraphics[width=0.8\textwidth]{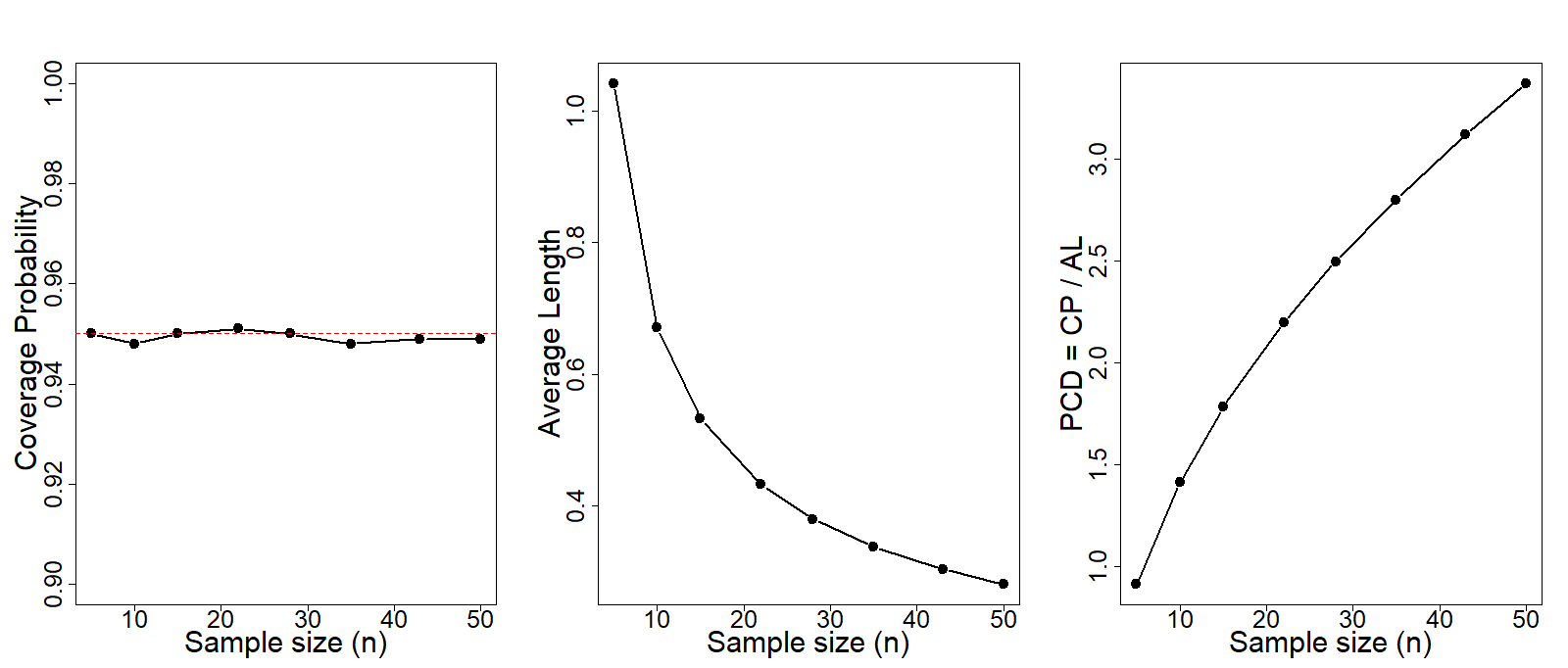}
	\caption{Bootstrap-t confidence interval}
	\label{bt}
\end{figure}
\vspace{1cm}
\begin{figure}[!t]
	\centering
	\includegraphics[width=0.8\textwidth]{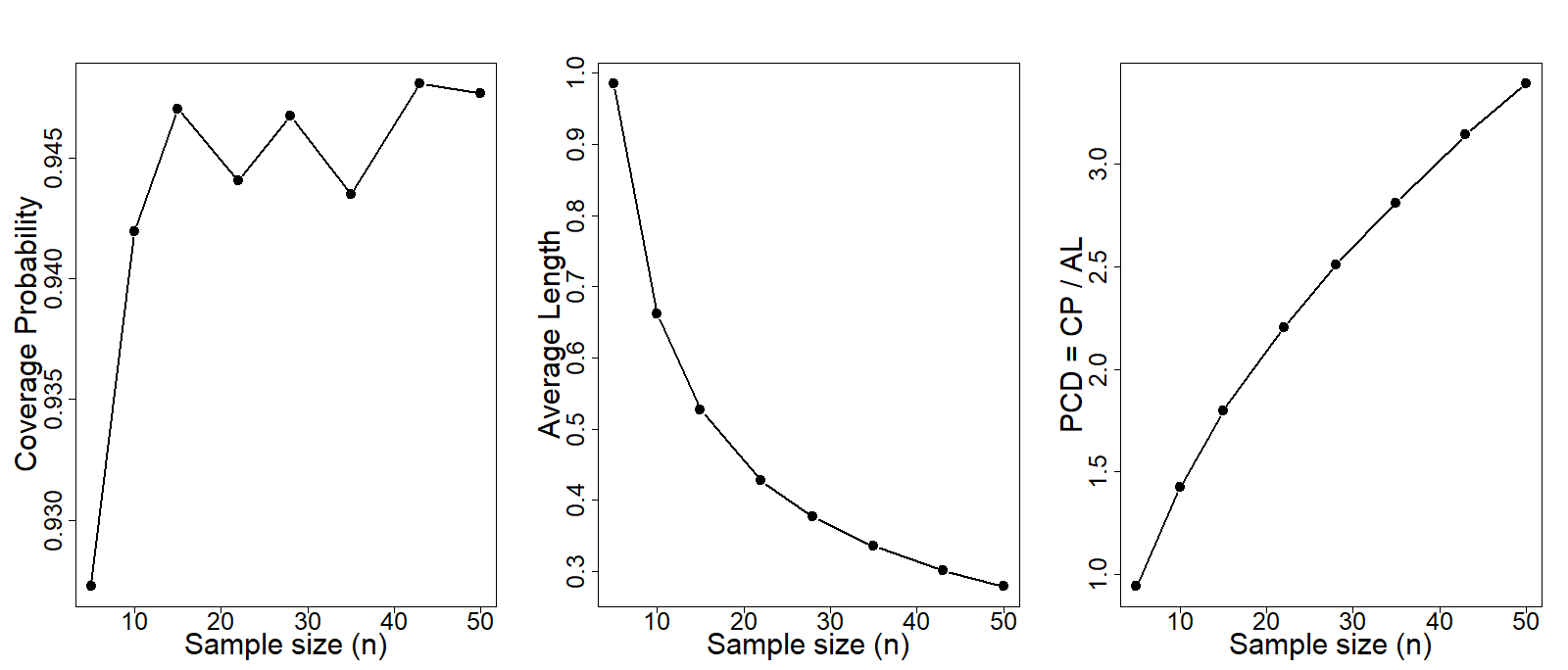}
	\caption{HPD credible interval}
	\label{p4}
\end{figure}
\vspace{1cm}
\begin{figure}[!t]
	\centering
	\includegraphics[width=0.8\textwidth]{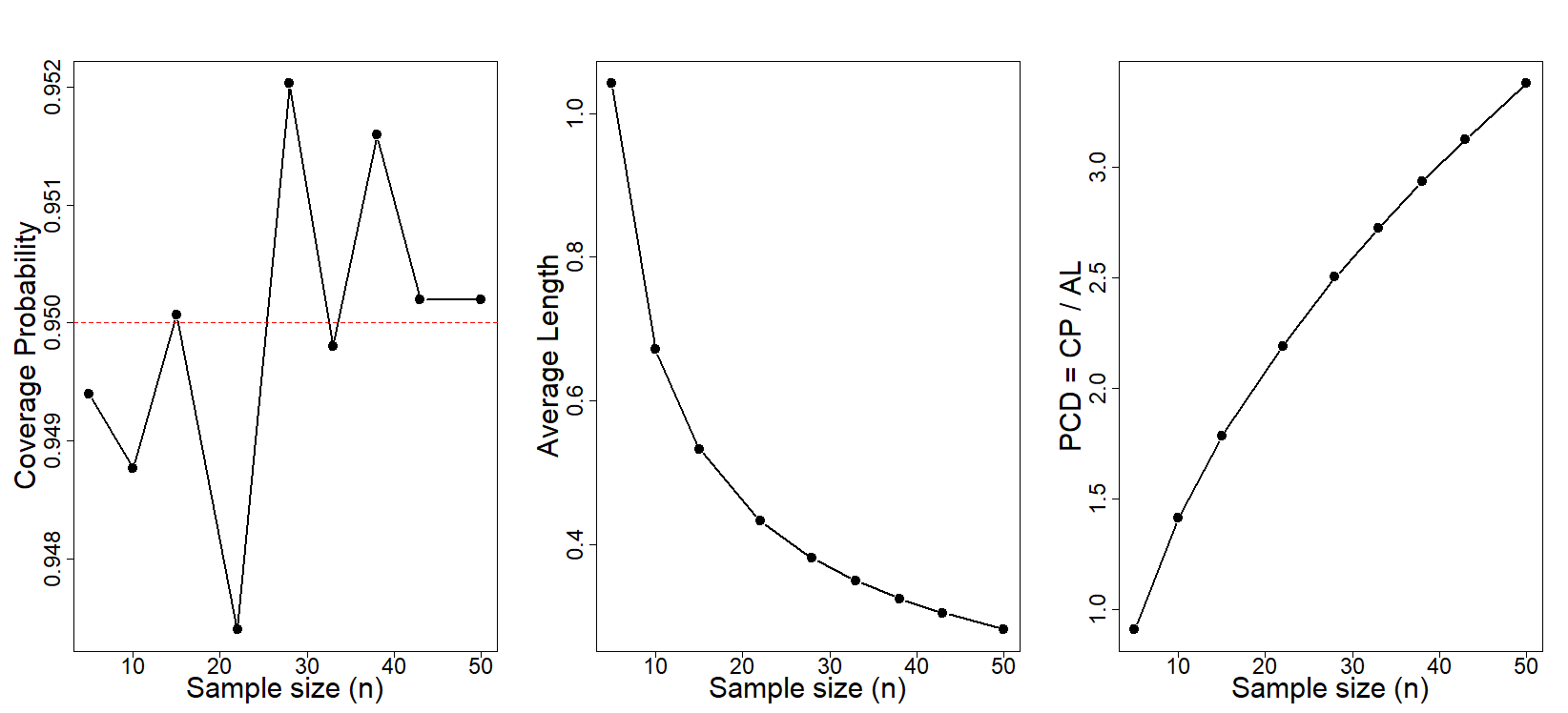}
	\caption{Generalized confidence interval using UMVUE}
	\label{p6}
\end{figure}

\section{Real data analysis}\label{sec8}
In this section, we present two real-life data analyses to illustrate the applicability of the derived results for the proposed results for $\ln \sigma$.
\subsection{Example I}
In this example, we consider the Air Quality Index (AQI) values of two cities in India namely Raipur and Bhilai. The data for Raipur and Bhilai were obtained from AQI\footnote{\url{https://www.aqi.in/in/dashboard/india/chhattisgarh/raipur}} and AQI\footnote{\url{https://www.aqi.in/in/dashboard/india/chhattisgarh/bhilai}} respectively. The data set consists of daily AQI values for both cities, recorded on consecutive days from 1 August 2026 to 7 August 2026. The data set is presented in Table \ref{no0}. Since the proposed results are derived for two-parameter normal populations, we first checked whether both data sets can be normally distributed. For this purpose, we used the Shapiro–Wilk test. At significance level of 0.05, we found that both data sets satisfies the normality assumption with $p$-values of 0.835 and 0.891 for Data 1 and Data 2 respectively. Since both $p$-values are much larger than $0.05$, normality is not rejected for either data set at the $5\%$ level of significance. The normal histogram plots in Figure \ref{no1} and QQ plots in Figure \ref{no2} support this conclusion, as the points lie close to the reference line. Although the sample size is $n = 6$ in each data sets.

To check the equal-variance assumption, we tested $H_0:\sigma_1^2=\sigma_2^2$ against $H_1:\sigma_1^2\neq\sigma_2^2$ using the $F$-test, where $\sigma_1^2$ and $\sigma_2^2$ denote the variances of the daily AQI of Raipur and Bhilai, respectively. The sample variances are $S_1^2 = 102.57$ and $S_2^2 = 89.48$, giving $F = S_1^2/S_2^2 = 1.146$. At the $5\%$ level, the critical values of the $F_{6,6}$ distribution are $0.172$ and $5.820$. Since $0.172 < 1.146 < 5.820$ (equivalently, $p$-value $= 0.873 > 0.05$), $H_0$ is not rejected. Hence, there is no evidence that the two variances differ, and the assumption $\sigma_1^2=\sigma_2^2$ is reasonable for these data. This is also supported by the box plots in Figure \ref{no3}, which show a similar spread for both cities.

To check the order restriction, we tested $H_0:\mu_1\le\mu_2$ against $H_1:\mu_1>\mu_2$ using the two-sample $t$-test. The test gives $T=(\bar X-\bar Y)/\sqrt{(S_1^2+S_2^2)/n}$, which follows a $t_{2(n-1)}$ distribution when $\mu_1=\mu_2$. For our data ($n=7$), $T=-0.218$, which is less than the critical value $t_{0.95;\,12}=1.782$ ($p$-value $=0.585$). Hence, $H_0$ is not rejected at the $5\%$ level, and the data do not contradict the assumption $\mu_1\le\mu_2$, as also reflected in the sample means $\bar X=50.71<\bar Y=51.86$. This is also supported by Figure \ref{no4}, where the sample mean AQI of Raipur lies slightly below that of Bhilai, with largely overlapping 95\% confidence intervals.

Based on the these data sets, we have obtained the values of the estimators of $\ln \sigma$ which is presented in Table \ref{tab1} for the loss functions $L_1(t)$ and $L_2(t)$ respectively. $\delta_{01}$ is better estimates of $\ln \sigma$. Based on the analysis, we infer that, $\delta_{S}$ and $\delta_{\phi_{*}}$ gives a common estimated values of $\ln \sigma$.
\begin{table}[h]
	\centering
	\caption{Air quality index}\label{no0}
	\begin{tabular}{lc}
		\toprule
		City & AQI \\
		\hhline{==}
		Raipur (Data 1) & 36 39 47 57 58 56 62 \\
		Delhi (Data 2) & 36 42 54 60 60 52 59 \\
		\bottomrule
	\end{tabular}
\end{table}

\begin{figure}[h!]
	\centering
	\includegraphics[width=1\textwidth]{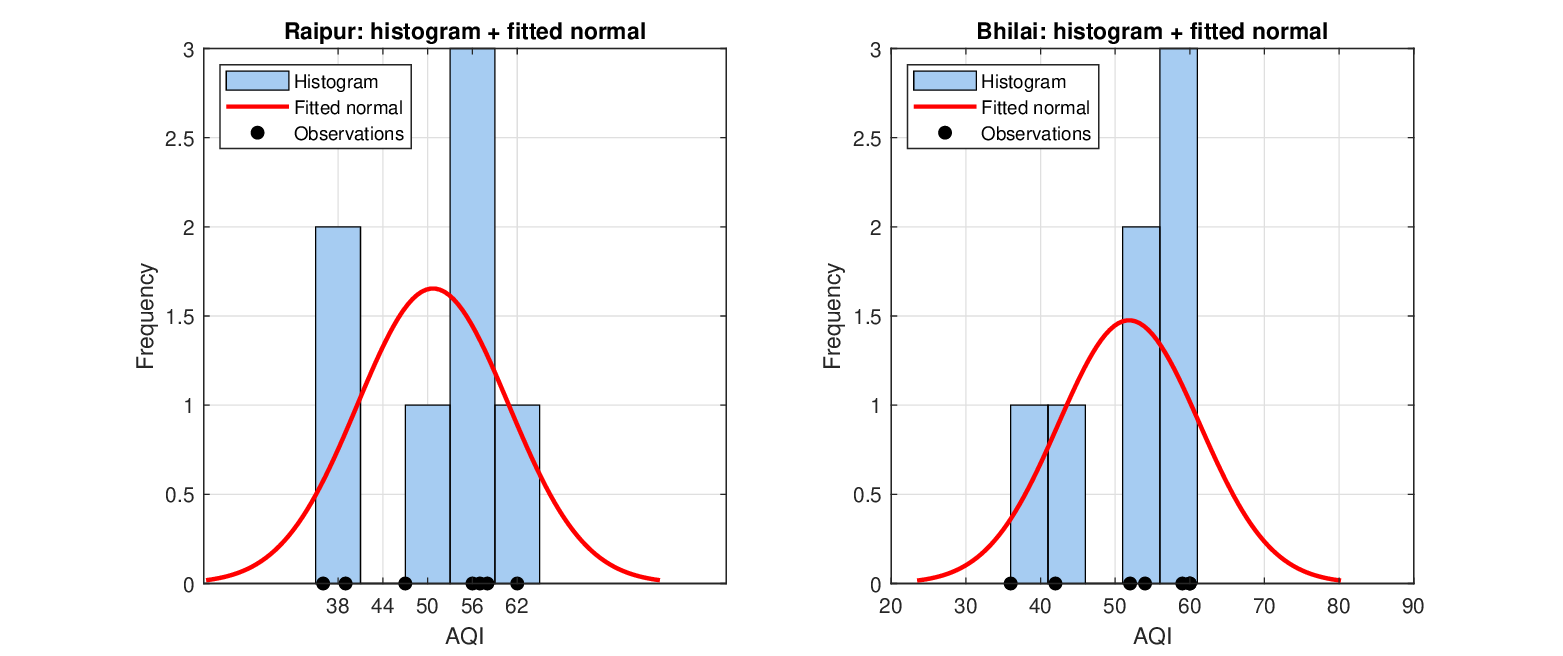}
	\caption{Normality test}
	\label{no1}
	\vfill
	\includegraphics[width=1\textwidth]{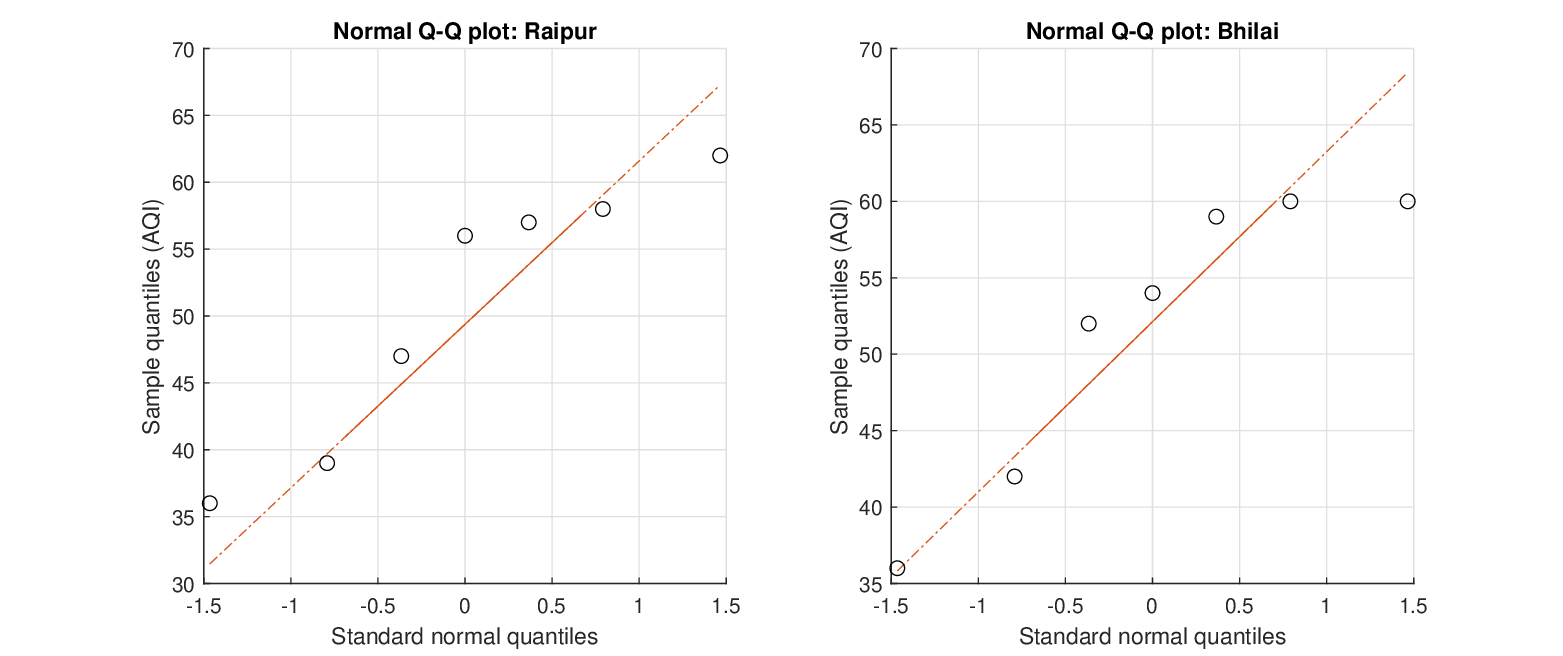}
	\caption{Normal Q--Q ploat}
	\label{no2}
\end{figure}

\begin{figure}[h!]
	\centering
	\includegraphics[width=0.5\textwidth]{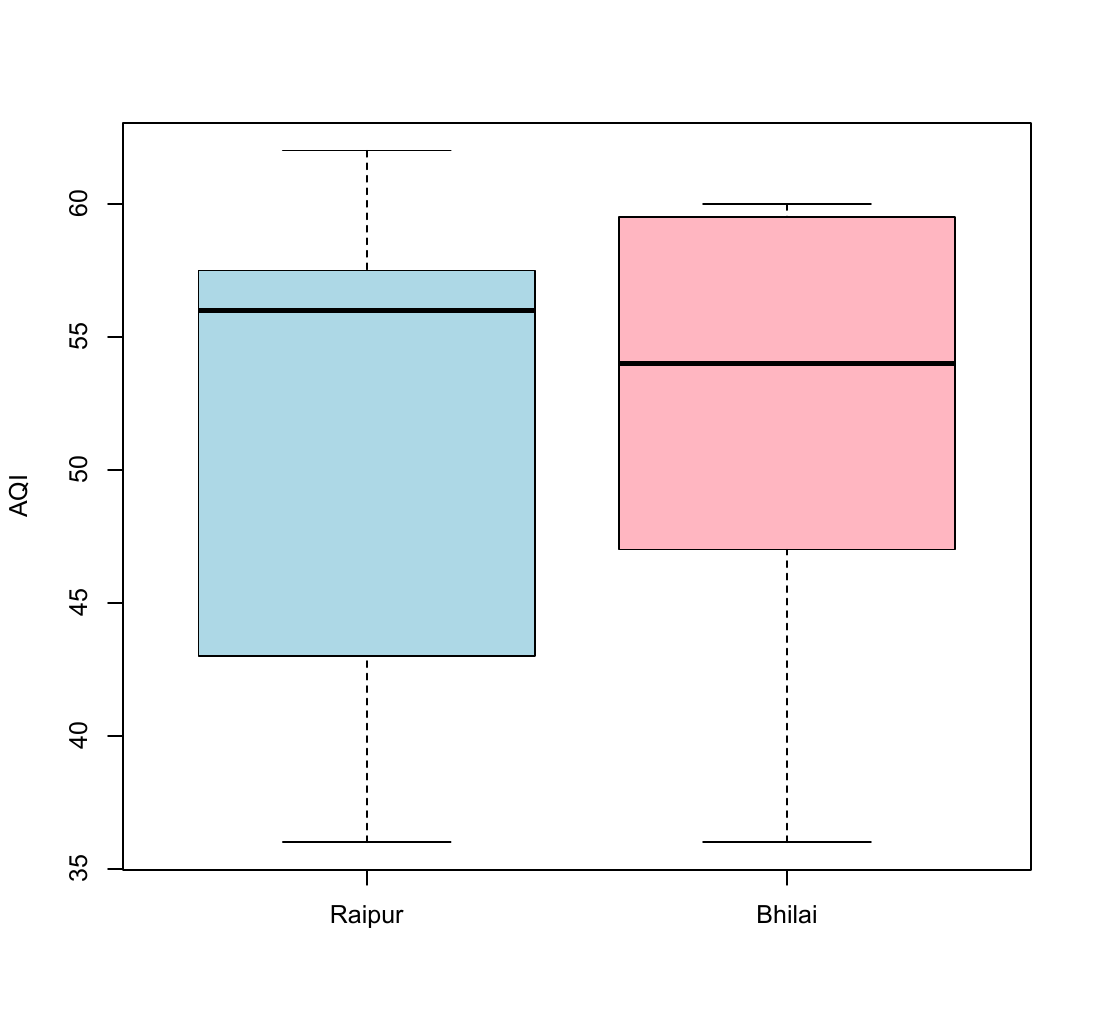}
	\caption{Box plots of daily AQI for Raipur and Bhilai (01/08/2026 – 07/08/2026).}
	\label{no3}
\end{figure}

\begin{figure}[h!]
	\centering
	\includegraphics[width=0.5\textwidth]{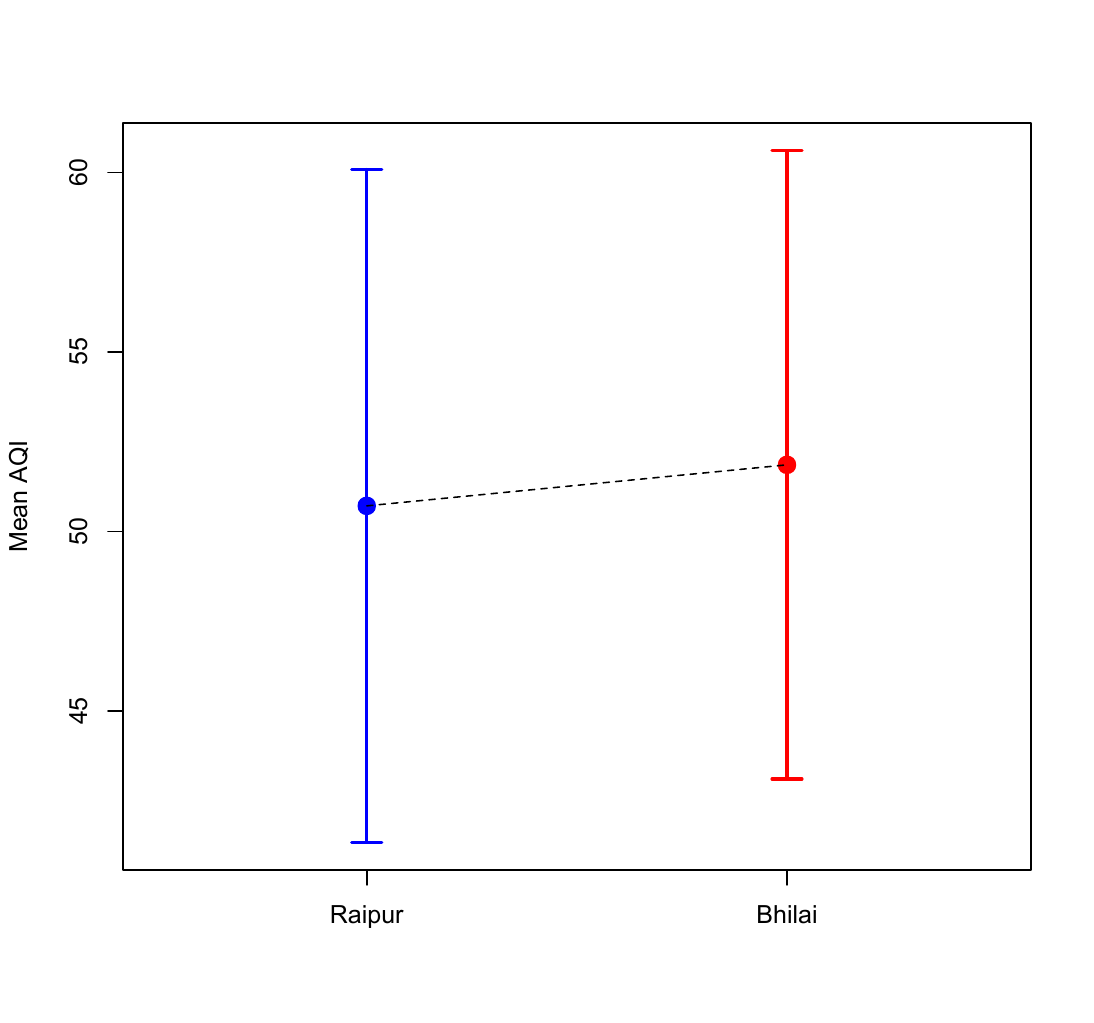}
	\caption{Mean AQI of Raipur and Bhilai with 95\% confidence intervals.}
	\label{no4}
\end{figure}

\begin{table}[h!]
	\centering
	\caption{Computed values of the improved estimators of $\ln \sigma$}\label{tab1}
	\begin{tabular}{@{}cccc@{}}
		\toprule
		Loss~~~~~~~~~~~~~~~& $\delta_{01}$ & $\delta_{S}$ & $\delta_{\phi_{*}}$ \\
		\hhline{====}
		$L_1(t)~~~~~~~~~~~~~~$ & 2.325 & 2.282 & 2.282 \\
		$L_2(t)\ (a_1=-3)$  & 2.4 & 2.35 & 2.35 \\
		$L_2(t)\ (a_1=-2)$  &2.373 & 2.326 & 2.326 \\
		$L_2(t)\ (a_1=2)~~~$  & 2.282& 2.242 & 2.242 \\
		$L_2(t)\ (a_1=4)~~~$  & 2.244 & 2.206 & 2.206 \\
		\bottomrule
	\end{tabular}
\end{table}

\begin{table}[h!]
	\centering
	\caption{Confidence intervals with lengths of $\ln \sigma$}\label{tab2}
	\begin{tabular}{@{}ccccc@{}}
		\toprule
		Method & ACI & Bt & HPD & GCI \\
		\hhline{=====}
		Lower & 1.835& 1.951 &1.98 & 1.561 \\
		Upper & 2.576& 2.783 & 2.254& 2.276\\
		Length & 0.741& 0.832 & 0.274 & 0.715\\
		\bottomrule
	\end{tabular}
\end{table}

\subsection{Example II}
In this example, we consider the ``UK Sprinters'' data set on male professional sprinters. It contains the best times of all men who ran 200 m in under 21.20 seconds without wind assistance and who also recorded a a time without wind assistance for 100 meter time in the same year. Details of the data set are given in ``A Handbook of Small Data Sets'' in \cite{hand1993handbook}  (p.~53). The data set is presented in Table \ref{no01}. To assess whether the normality assumption is reasonable, we applied the Shapiro–Wilk test to the winning times of both events. The Shapiro–Wilk test yields p-values of 0.3319 and 0.1971 for the 100 m and 200 m winning times, respectively. Since both exceed 0.05, normality is not rejected for either event. The normal histogram plots in Figure \ref{no21} and QQ plots in Figure \ref{no22} support this conclusion. To check the equal-variance, we used F-test. The F-test for equality of variances gives $F = 0.5087$ with $p$-value 0.2889. Since the $p$-value exceeds 0.05, the hypothesis of equal variances is not rejected for the 100 m and 200 m winning times. Let $\mu_1$ and $\mu_2$ be the mean average speeds (m/s) of sprinters in the 200 m and 100 m races, respectively. A sprinter usually reaches top speed in the first 100 m and then slows down, so the average speed in the 200 m is generally lower than in the 100 m. Hence, it is reasonable to assume $\mu_ \leq \mu_2$. For example, Usain Bolt's world records of 19.19 s (200 m) and 9.58 s (100 m) give average speeds of 10.42 m/s and 10.44 m/s, respectively. Let $X_1$ and $X_2$ be random variables that represent the sample means of average speed (m/s) of 100 m and 200 m sprints respectively. Then, $X_1$ and $X_2$ be $N(\mu_1,\sigma^2/n)$ and $N(\mu_2,\sigma^2/n)$ respectively where $\mu_1$, $\mu_2$ are means with $\mu_1\leq \mu_2$ and $\sigma^2$ is variances. Based on the these data sets, we have obtained the values of the estimators of $\ln \sigma$ which is presented in Table \ref{tab21} for the loss functions $L_1(t)$ and $L_2(t)$ respectively. $\delta_S$ is better estimates of $\ln \sigma$. Based on the analysis, we infer that, $\delta_{01}$ and $\delta_{\phi_{*}}$ gives a common estimated values of $\ln \sigma$.
\begin{table}[h]
	\centering
	\caption{UK sprinters data set}\label{no01}
	\begin{tabular}{lcc}
		\toprule
		Athlete name & Time in 200 m (in s) & Time in 100 m (in s) \\
		\hhline{===}
		L Christie & 20.09 &9.97 \\
		J Regis & 20.32 & 10.31 \\
		M Rossowess & 20.51 & 10.40\\
		A Carrott & 20.76 & 10.56 \\
		T Bennett & 20.90 & 10.92 \\
		A Mafe & 20.94 & 10.64 \\
		D Reid & 21.00 & 10.54 \\
		P Snoddy & 21.14 & 10.85 \\
		L Stapleton & 21.17 & 10.71\\
		C Jackson &21.19 & 10.56 \\
		\bottomrule
	\end{tabular}
\end{table}

\begin{figure}[h!]
	\centering
	\includegraphics[width=1\textwidth]{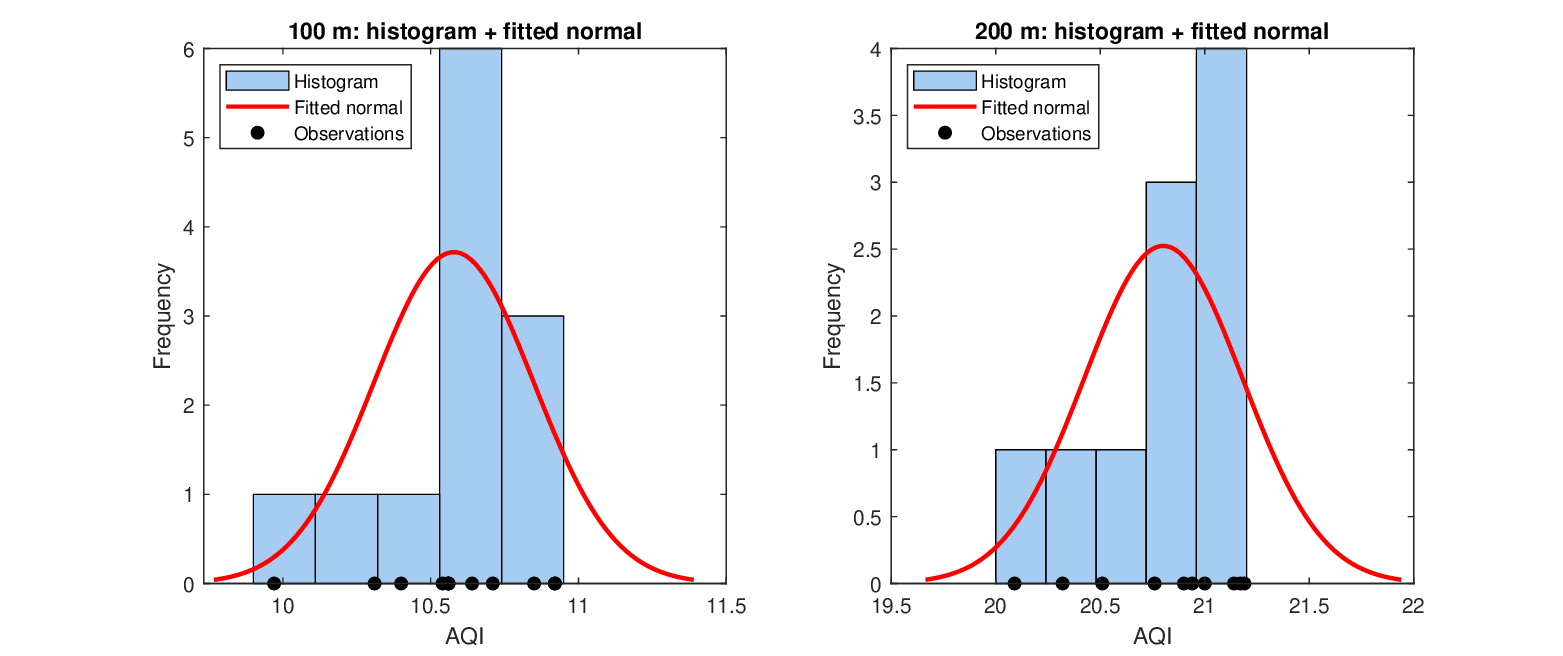}
	\caption{Normality test}
	\label{no21}
	\vfill
	\includegraphics[width=1\textwidth]{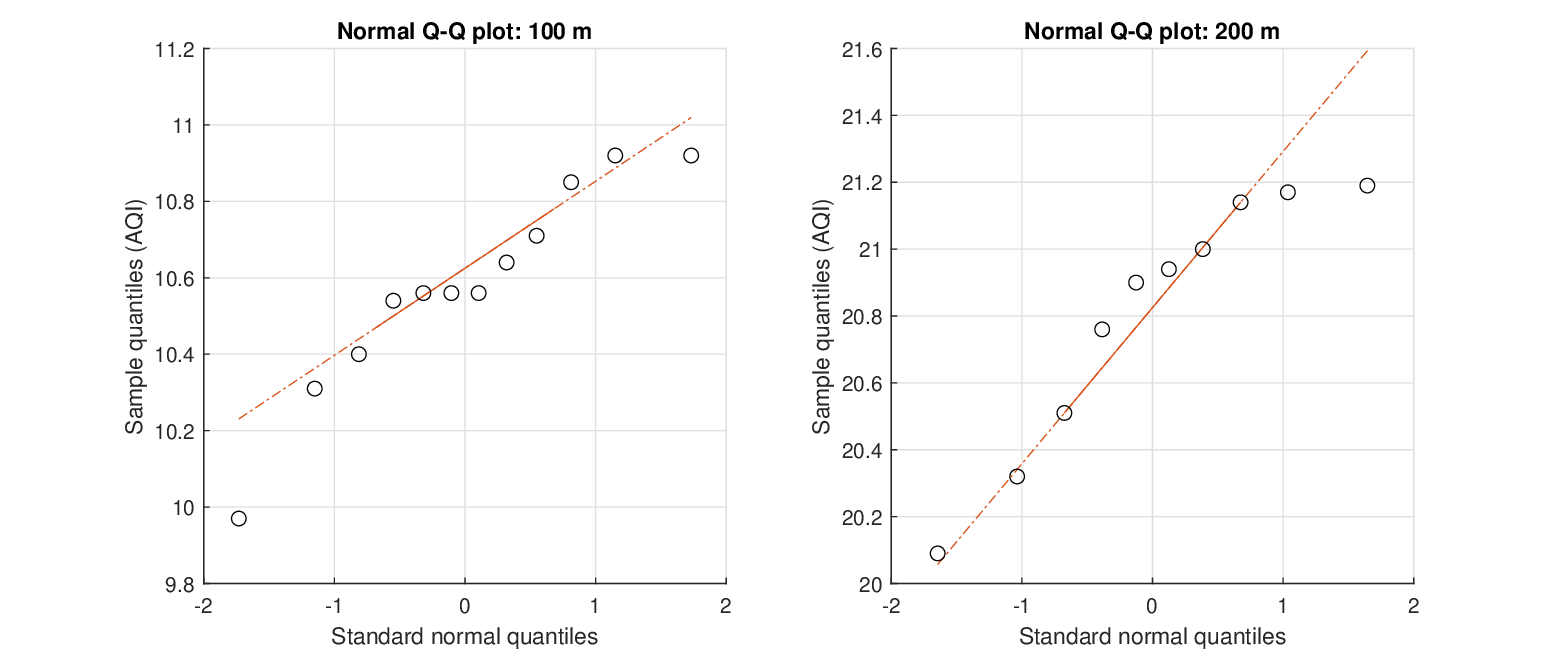}
	\caption{Normal Q--Q ploat}
	\label{no22}
\end{figure}

\begin{table}[h!]
	\centering
	\caption{Computed values of the improved estimators of $\ln \sigma$}\label{tab21}
	\begin{tabular}{@{}cccc@{}}
		\toprule
		Loss~~~~~~~~~~~~~~~& $\delta_{01}$ & $\delta_{S}$ & $\delta_{\phi_{*}}$ \\
		\hhline{====}
		$L_1(t)~~~~~~~~~~~~~~$ & -1.046 & 1.317 & -1.046 \\
		$L_2(t)\ (a_1=-3)$  & -0.999 & 1.362 & -0.999 \\
		$L_2(t)\ (a_1=-2)$  & -1.015 & 1.346 & -1.015 \\
		$L_2(t)\ (a_1=2)~~~$  & -1.074 & 1.291 & -1.074 \\
		$L_2(t)\ (a_1=4)~~~$  & -1.101 & 1.266 & -1.101 \\
		\bottomrule
	\end{tabular}
\end{table}

\begin{table}[h!]
	\centering
	\caption{Confidence intervals with lengths of $\ln \sigma$}\label{tab22}
	\begin{tabular}{@{}ccccc@{}}
		\toprule
		Method & ACI & Bt & HPD & GCI \\
		\hhline{=====}
		Lower & -1.437& -1.356 & -1.368 & -1.355 \\
		Upper & -0.817& -0.693 & -0.691 & -0.683 \\
		Length& 0.620 & 0.663  &  0.677 &  0.671 \\
		\bottomrule
	\end{tabular}
\end{table}

%\subsection{Example-III} we analyse the winning times of male professional sprinters in the 100 m and 200 m events. The data consist of the gold-medal times in the men’s Olympic 100 m and 200 m finals from 1900 to 2022, obtained from Statista\footnote{\url{https://www.statista.com/statistics/1091753/olympics-200m-gold-medal-times-since-1900/}} and Statista\footnote{\url{https://www.statista.com/statistics/1090316/olympics-100m-gold-medal-times-since-1896/}}. 

\clearpage
\section{Concluding remarks}\label{sec9}
Shannon entropy quantifies the degree of uncertainty associated with a random phenomenon. It has wide range of application in various fields such as telecommunications, meteorology, econometrics, and molecular biology. In many of these applications, data are modeled using probability distributions; thus, the estimation of entropy in a parametric framework becomes an important problem. In this paper, we consider the problem point and interval estimation of the entropy of two independent normal populations with unequal means. For point estimation, we obtain several improved estimators under a general location invariant loss function using a decision theoretic approach. In this case, we derive MLE, RMLE, an UMVUE. We also obtain a class of improved estimators that dominates the BAEE. Furthermore, by using Brewster and Zidek technique we derive a class of smooth improved estimators that also dominates BAEE. %Consequently we show that the Brewster and Zidek type estimator is coincides with the Kubokawa IERD type estimator.
We derive explicit expressions of the proposed improved estimators for two specific loss functions namely quadratic loss function and Linex loss function. In addition, we derive a improved estimators under generalized Pitman closeness criterion. A detailed numerical comparison of the risk performance of the estimators have been provided. In the case of interval estimation, we derive asymptotic confidence interval, bootstrap confidence intervals, generalized confidence interval, HPD credible interval using MCMC method. We compare the performance of the proposed confidence intervals for the parameter $\tau$ using a graphical representations of Al and CP in the simulation study. Since it is difficult to rank intervals that consider both the shortest AL and the highest CP, we use the PCD criterion to select the most appropriate confidence interval. Finally, two real world data analysis has been conducted to demonstrate the proposed results.
\section*{Appendix}
\begin{lemma} (Lemma $2.1$, \cite{kubokawa1994unified})\label{app1}
	Let $p(z)$ and $q(z)$ be positive functions such that the ratio $\dfrac{p(z)}{q(z)}$ is non-decreasing in $z$. Suppose there exists a function $r(z)$ satisfying $r(z)<0$ for $z<z^{\ast}$ and $r(z)>0$ for 
	$z>z^{\ast}$. Then, we have
	
	\begin{align*}
		\int_{0}^{\infty} r(z)\,\frac{p(z)}{q(z)}\,dz 
		\;\geq\; \frac{p(z^{\ast})}{q(z^{\ast})}\int_{0}^{\infty} r(z)\,dz.
	\end{align*}
	Equality holds if and only if $\dfrac{p(z)}{q(z)}$ is constant almost everywhere.
\end{lemma}
\begin{lemma}\label{aplm}Let $d_1$, $d_2$, $n$, $\eta$ and $\alpha$ be any fixed positive real numbers with $d_1<d_2$ then the ratio of integral $\displaystyle \frac{I(y-d_2)}{I(y-d_1)}= \frac{\int_{-\alpha}^{\alpha}e^{-\frac{1}{4}\left(\sqrt{n}e^{y-d_2}w-\eta\right)^2}dw}{\int_{-\alpha}^{\alpha}e^{-\frac{1}{4}\left(\sqrt{n}e^{y-d_1}w-\eta\right)^2}dw}$ is increasing in $y$. 
\end{lemma}
\noindent \textit{Proof:} Define
$$I(y)=\int_{-\alpha}^{\alpha}e^{-\frac{1}{4}\left(\sqrt{n}e^yw-\eta\right)^2}dw.$$ By the change of variable $u=(\sqrt{n}e^yw-\eta)/\sqrt{2}$, we obtain,
\begin{align*}
	I(y)=&\frac{\sqrt{2}}{\sqrt{n}e^y}\int_{(-\sqrt{n}e^y\alpha-\eta)/\sqrt{2}}^{(\sqrt{n}e^y\alpha-\eta)/\sqrt{2}}e^{-\frac{1}{2}u^2}du\\
	=&\frac{2\sqrt{\pi}\left[ \Phi\left(\frac{\sqrt{n}e^y\alpha-\eta}{\sqrt{2}}\right) - \Phi\left(\frac{-\sqrt{n}e^y\alpha-\eta}{\sqrt{2}}\right) \right]}{\sqrt{n}e^y},
\end{align*}
where $\Phi$ denotes CDF of the standard normal distribution. Consider the auxiliary function $$Q(t)=\Phi(t-\eta/\sqrt{2})-\Phi(-t-\eta/\sqrt{2}),\ t=\sqrt{n}e^y\alpha/\sqrt{2} $$ we may rewrite $I(y)=\tfrac{2\sqrt{\pi}}{\sqrt{n}e^y}Q(t)$.
For fixed $d_1<d_2$, consider the ratio $R(y)= \frac{I(y-d_2)}{I(y-d_1)}$. Substituting the above expression yields $$R(y)=e^{d_2-d_1}\frac{Q(t_2)}{Q(t_1)},$$ 
where $t_1=\sqrt{n}e^{y-d_1}\alpha$ and $t_2=\sqrt{n}e^{y-d_2}\alpha$. Differentiating $\log R(y)$ gives $$\frac{d}{dx}\log R(y)=\frac{Q'(t_2)}{Q(t_2)}t_2 + \frac{Q'(t_1)}{Q(t_1)}t_1.$$
The function $Q(t)$ is a log-concave function since it is an integral of a log-concave normal density over a symmetric interval. Hence, $\psi(t)=\frac{Q'(t)}{Q(t)}t$ is decreasing in $t>0$. Since $d_1<d_2$, it follows that $t_1>t_2$ and therefore, $\frac{Q'(t_2)}{Q(t_2)}t_2>\frac{Q'(t_1)}{Q(t_1)}t_1$. Thus $\frac{d}{dx}\log R(y)>0$, which proves that $R(y)$ is strictly increasing in $y$.

\noindent \textbf{Proof of Lemma \ref{lmm4.1}} \begin{itemize}
	\item [(i)] To prove the function $\mathcal{H}_{\mu}(r)$ is strictly bowl shaped, it suffices to show that the family $\left\{f_{\mu}(z-r,\alpha),\ r\in (-\infty,\infty)\right\}$ is MLR increasing (see Lemma $2.1$ of \cite{brewster1974improving}). For any $r_1<r_2$, we have 
	\begin{eqnarray*}
		R(z)= \frac{f_{\mu}(z-r_2,\alpha)}{f_{\mu}(z-r_1,\alpha)} = e^{(2n-1)(r_1-r_2)
			-\frac{1}{2}\left(e^{2(z-r_2)}-e^{2(z-r_1)}\right)}	\frac{\displaystyle
			\int_{-\alpha}^{\alpha}	e^{-\frac{1}{4}\left(\sqrt{n}e^{z-r_2}w-\mu\right)^2}\,dw}
		{\displaystyle \int_{-\alpha}^{\alpha} e^{-\frac{1}{4}\left(\sqrt{n}e^{z-r_1}w-\mu\right)^2}\,dw}.
	\end{eqnarray*}
	Since $r_1<r_2$ then the first term is positive and increasing and the second term is increasing follows from the Lemma \ref{aplm}.
	\item[(ii)] For $\mu >0$, to prove $r_{\mu}(\alpha)\leq r_0(\alpha)$, it is sufficient to prove the function $\frac{f_{\mu(z,\alpha)}}{f_0(z,\alpha)}$ is increasing in $z$. Now, we have		
	\begin{align*}
		\frac{f_{\mu}(z,\alpha)}{f_{0}(z,\alpha)} \propto
		\frac{\int_{-\alpha}^{\alpha}
			e^{-\tfrac14(\sqrt{n}e^{z}w-\mu)^2}dw}
		{\int_{-\alpha}^{\alpha}
			e^{-\tfrac14 ne^{2z}w^2}dw} &\propto 
		\frac{\displaystyle
			\int_{0}^{\alpha}
			e^{-\tfrac14(\sqrt{n}e^{z}w+\mu)^2}\,dw}
		{\displaystyle
			\int_{0}^{\alpha}
			e^{-\tfrac14 ne^{2z}w^2}\,dw}+
		\frac{\displaystyle
			\int_{0}^{\alpha}
			e^{-\tfrac14(\sqrt{n}e^{z}w-\mu)^2}\,dw}
		{\displaystyle
			\int_{0}^{\alpha}
			e^{-\tfrac14 ne^{2z}w^2}\,dw}\\
		&\propto
		\frac{\int_{0}^{e^{z}\alpha/2}
			e^{-ny^2}
			\left[e^{-\sqrt{n}y\mu}
			+e^{\sqrt{n}y\mu}\right]dy}
		{\int_{0}^{e^{z}\alpha/2} e^{-ny^2}dy}\\
		&\propto
		E_z\!\left[e^{-\sqrt{n}Y\mu}
		+e^{\sqrt{n}Y\mu}\right].
	\end{align*}
	where $y=e^zw/2$ and the expectation is taken with the respect to the random variable $Y$ with the density function $g_{z}(y) = \frac{e^{-\frac{1}{2}ny^2}}{\int_{0}^{e^z\alpha}e^{-\frac{1}{4}ny^2}dy }$, $0\leq y\leq \alpha e^z/2$. Furthermore for $z_1> z_2 >0$, $\frac{g_{z_1(y)}}{g_{z_2}(y)}$ is non decreasing in $y$. Thus using the Lemma A.2. from \cite{bobotas2010estimation}, we obtain the result.
	\item[(iii)] Let suppose that $0<\alpha_1<\alpha_2$. To show $r_0(\alpha_1)\leq r_0(\alpha_2)$ it is enough to prove that the function $\frac{f_0(z,\alpha_1)}{f_0(z,\alpha_2)}$ is increasing in $z$.
	\begin{equation}
		\label{qz}
		\begin{aligned}
			q(z) = \frac{f_0(z,\alpha_1)}{f_0(z,\alpha_2)} = \frac{\int_{-\alpha_1}^{\alpha_1} e^{-\tfrac14 n e^{2z} w^2}\, dw}
			{\int_{-\alpha_2}^{\alpha_2} e^{-\tfrac14 n e^{2z} w^2}\, dw} = \frac{\int_{0}^{\alpha_1}e^{-\tfrac14 n e^{2z} w^2}\, dw} {\int_{0}^{\alpha_2} e^{-\tfrac14 n e^{2z} w^2}\, dw}.
		\end{aligned}
	\end{equation}
	It is simple to demonstrate that the function 
	\begin{align*}
		Q(\alpha) =  \frac{\int_{0}^{\alpha}nw^2e^{-\frac{1}{4}ne^{2z}w^2}dw}{\int_{0}^{\alpha}e^{-\frac{1}{4}ne^{2z}w^2}dw}
	\end{align*}
	is increasing in $\alpha$, this is immediately implies that $q'(z)>0$. We can express the function $Q(\alpha)$ as $Q(\alpha)=E_{\alpha}\varphi(W)$, where $\varphi(w)=nw^2$ and the expectation is taken with the respect to the density function $\pi_{\alpha}(w) \propto e^{-\frac{1}{4}ne^{2z}w^2}$; where $0\leq w \leq \alpha$. Hence from the Lemma A.2. given in \cite{bobotas2010estimation}, it is directly shows that $E_{\alpha}\varphi(W)$ is increasing in $\alpha$. Hence the result follows.
\end{itemize}

\noindent\textbf{Proof of Theorem \ref{thierd}} Let $RD$ denote the difference in risk between the estimators $\delta_0$ and $\delta_{\phi_2}$. This can be expressed as
\begin{align*}
	RD=&E\int_{1}^{\infty}L'\left(\ln \sqrt{v} +\phi(xt)\right)\phi'(xt)t dx \ \Big(\lim\limits_{x\rightarrow\infty}\phi(x) = d_0\Big)\\
	=&n^2\int_{0}^{\infty}\int_{0}^{\infty}\int_{1}^{\infty}L'\left(\ln \sqrt{v} +\phi(xt)\right)\phi'(xt)tv g_{U_1}(ntv;\eta) h_{U_2}(nv)dx dv dt
\end{align*}
Put $y=xt$, we get
\begin{align*}
	RD=n^2\int_{0}^{\infty}\int_{0}^{\infty}\int_{1}^{\infty}L'\left(\ln \sqrt{v} +\phi(y)\right)\phi'(y)\frac{y}{x}vg_{U_1}(n\frac{y}{x}v;\eta) h_{U_2}(nv)dx dv \frac{dy}{x}
\end{align*}
Take the transformation $p=\frac{y}{x}$, we obtain
\begin{align*}
	RD&=n^2\int_{0}^{\infty}\int_{0}^{\infty}\int_{0}^{y}L'\left(\ln \sqrt{v} + \phi(y)\right)\phi'(y)vg_{U_1}(npv) h_{U_2}(nv)dp dv dy\\
	%	=&n\int_{0}^{\infty}\int_{0}^{\infty}L'\left(\ln \sqrt{v} + \phi(y)\right)\phi'(y)\\
	%	&~~~~~~~~~~~~~~~~~~~~~G_{U_1}(nyv;\eta) h_{U_2}(nv) dv dy\\
	&=n\int_{0}^{\infty}\phi'(y)\int_{0}^{\infty}\frac{G_{U_1}(nyv;\eta)}{G(nyv)}L'\left(\ln \sqrt{v} + \phi(y)\right)G(nyv)h_{U_2}(nv) dv dy.
\end{align*}
From Condition (C1), it follows that $L'(t)<0$ when $t<0$ and $L'(t)>0$ when $t>0$. Moreover $\frac{G_{U_1}(z;\eta)}{G(z)}$ is nondecreasing for $z>0$, because $\frac{g_{U_1}(z;\eta)}{g(z)}$ is nondecreasing for $z>0$. Define $v_0=\left(e^{-\phi(z)}\right)^2$. By applying the Lemma \ref{app1}, we have
\begin{align*}
	RD \geq  n\int_{0}^{\infty}\phi'(y)\frac{G_{U_1}(nyv_0; \eta)}{G(nyv_0)}\int_{0}^{\infty}L'\left(\ln \sqrt{v} + \phi(y)\right)G(nyv)h_{U_2}(nv) dv dy.
\end{align*}
Since we have $\phi^{'}\geq0$ almost every where. Then the risk difference becomes non-negative when
\begin{align*}
	\int_{0}^{\infty}L'\left(\ln \sqrt{v} + \phi_2(y)\right)G(nyv)h_{U_2}(nv) dv\geq0.
\end{align*}

\section*{Data availability statement}
Data used in this study are publicly available in the literature and cited in this manuscript.
\section*{Disclosure statement}
No potential conflict of interest was reported by the authors.
\section*{Funding}
Lakshmi Kanta Patra thanks the Anusandhan National Research Foundation India and NBHM for providing financial support to carry out this research with project numbers MTR/2023/000229 and 02011/38/2023 NBHM (R.P)/R \& DII/13409 respectively.

\end{document}